%% file: main.tex
\documentclass[a4paper,10pt]{article} % Arxiv
\input{preamble_main}

\title{%
    Accuracy of the Ensemble Kalman Filter in the Near-Linear Setting
}

\ifsiamart
\author{%
    E. Calvello\thanks{%
        Department of Computing and Mathematical Sciences, Caltech, USA  (\email{e.calvello@caltech.edu}).
    } \and
    P. Monmarch\'e\thanks{%
        Laboratoire Jacques-Louis Lions, Sorbonne Universit\'e \& Institut Universitaire de France, France (\email{pierre.monmarche@sorbonne-universite.fr}).
    } \and
    A. M. Stuart\thanks{%
        Department of Computing and Mathematical Sciences, Caltech, USA (\email{astuart@caltech.edu}).
    } \and
    U. Vaes\thanks{%
        MATHERIALS, Inria Paris \& CERMICS, \'Ecole des Ponts, France (\email{urbain.vaes@inria.fr})
    }
}

\else
\author[1]{E. Calvello$^{a,}$}
\author[2,3]{P. Monmarch\'e$^{b,}$}
\author[1]{A. M. Stuart $^{c,}$}
\author[4,5]{U. Vaes $^{d,}$}
\affil[ ]{\footnotesize
    $^a$\email{e.calvello@caltech.edu},
    $^b$\email{pierre.monmarche@sorbonne-universite.fr},
    $^c$\email{astuart@caltech.edu},
    $^d$\email{urbain.vaes@inria.fr}
}
\affil[1]{\footnotesize Department of Computing and Mathematical Sciences, Caltech, USA}
\affil[2]{\footnotesize Laboratoire Jacques-Louis Lions, Sorbonne Université, France}
\affil[3]{\footnotesize Institut Universitaire de France, France}
\affil[4]{\footnotesize MATHERIALS project-team, Inria Paris, France}
\affil[5]{\footnotesize CERMICS, \'Ecole des Ponts, France}
\fi

\begin{document}
\maketitle

%\ec{Edo}, \p{Pierre}, \ams{Andrew}, \uv{Urbain}

\begin{abstract}
The filtering distribution captures the statistics of the state of a dynamical system from partial and noisy observations. Classical particle filters provably approximate this distribution in quite general settings; however they behave poorly for high dimensional problems, suffering weight collapse. This issue is circumvented by the ensemble Kalman filter which is an equal-weight
interacting particle system. However, this finite particle system is only proven to approximate the
true filter in the linear Gaussian case. In practice, however, it is applied in much broader settings; as a result, establishing its approximation properties more generally is important. There has been recent progress in the theoretical analysis of the algorithm, establishing stability and error estimates in non-Gaussian settings, but the assumptions on the dynamics and observation models
rule out the unbounded vector fields that arise in practice and the analysis applies only to the mean field limit of the ensemble Kalman filter. The present work establishes error bounds between the filtering distribution and the finite particle ensemble Kalman filter when the dynamics and observation vector fields may be unbounded, allowing linear growth.
\end{abstract}



\section{Introduction}
\label{sec:I}

\subsection{Literature Review, Contributions and Outline}
\label{ssec:O}

Algorithms for filtering employ noisy observations arising from a possibly random dynamical system to estimate the distribution of the state of the dynamical system conditional on the observations.
The Kalman filter~\cite{kalman1960new} determines the filtering distribution exactly for linear Gaussian dynamics and observations. The extended Kalman filter was developed as an extension of the Kalman filtering technique to nonlinear problems and is based on a linearization approximation; see \cite{jazwinski2007stochastic} and \cite{anderson2012optimal}. The linearization approximation leads to an inexact distribution for nonlinear problems, and furthermore requires evaluation of covariance matrices, making the methodology impractical for large-scale geophysical applications \cite{ghil}. Particle filters or sequential Monte Carlo methods \cite{doucet2001sequential,chopin:20} offer an alternative methodology for nonlinear filtering problems, allowing recovery of the exact filtering distribution in the large particle limit. However, this class of methods scales poorly with dimension, and in particular suffers
weight collapse, making its application to geophysical problems prohibitive; see for example \cite{snyder2008obstacles,bickel2008sharp,rebeschini2015can,agapiou2017importance}.

The ensemble Kalman filter was introduced in the seminal paper \cite{evensen1994sequential}.
Its success stems from the low-rank approximation of large covariances, cheaply computed using an ensemble of particles, which allows the methodology to be deployed in geophysical applications.
However analysis of
its accuracy, in relation to the true filtering distribution, remains in its infancy. The papers
\cite{le2009large,mandel2011convergence} studied this issue in the linear Gaussian setting where
the mean field limit of the Kalman filter is exact; they demonstrate that the ensemble Kalman filter
may be viewed as an interacting particle system approximation of the mean field limit, and establish Monte Carlo type error bounds. The recent article \cite{calvello2025ensemble} overviews
the formulation of ensemble Kalman methods using mean field dynamical systems and provides
a platform from which the analyses of \cite{le2009large,mandel2011convergence} may be generalized
beyond the linear Gaussian setting. In the recent paper \cite{carrillo2022ensemble} the authors
establish stability properties of the mean field ensemble Kalman filter and use them to
prove accuracy of the filter in a near-Gaussian setting. However, the paper does not consider
finite particle approximations of the mean field, and the conditions on the dynamics-observation model require boundedness of the vector fields arising. In this paper we address both these issues, establishing
error bounds between the finite particle ensemble Kalman filter and the true filtering
distribution in settings where the dynamics and observation vector fields grow linearly.
We make the following contributions.

\begin{enumerate}
    \item \label{contribution:stability}
    \Cref{theorem:main_theorem} is a stability result for the mean field ensemble Kalman filter in the setting of dynamical models and Lipschitz obervation operators that grow at most linearly at infinity.
    \item \label{contribution:mfenkf}
    \cref{thm:error_estimate_mfenkf} quantifies the error between the mean field ensemble Kalman filter and the true filter in the setting of dynamical models and Lipschitz observation operators that are near-linear.
    \item \label{contribution:enkf}
     \cref{theorem:enkf_theorem} quantifies the error between the finite particle ensemble Kalman filter (found as an interacting particle system approximation of the mean field) and the true filter in the setting of near-linear, Lipschitz dynamical models and linear observation operators.
\end{enumerate}

In going beyond the work in \cite{carrillo2022ensemble},
the current paper simultaneously addresses a more applicable problem class,
by allowing linear growth in the dynamics and observation operators,
and confronts the substantial technical challenges which arise from doing so.
Bounds on moments of the filtering distribution and mean field ensemble Kalman filter must be established; these bounds grow in the number of iterations which is in contrast to \cite{carrillo2022ensemble} where the $L^\infty$ bounds on model and observation operator ensure a uniform bound on moments. A further challenge is presented by the need to establish stability bounds for the conditioning map, giving rise to the true filter, and the transport map giving rise to the ensemble Kalman filter. These results exhibit dependence on moments in the stability constants and require control of the growth given by the dynamics; this is again in contrast to \cite{carrillo2022ensemble} where the $L^\infty$ bounds allow for a uniform control.

After discussing notation that will be used throughout the paper in~\cref{ssec:N}, we introduce the filtering problem in~\cref{ssec:FD}. In~\cref{sec:EKF} we outline the main results of the paper concerning the ensemble Kalman filter. In~\cref{ssec:MF} we formulate the ensemble Kalman filter that we consider in this paper along with the relevant assumptions we will use in the analysis. We state a stability theorem for the mean field formulation of the ensemble Kalman filter in~\cref{ssec:SEKF}, hence Contribution \ref{contribution:stability}. We leverage this result in~\cref{ssec:AEKF} to derive a theorem quantifying the error between the mean field ensemble Kalman filter and the true filter, Contribution \ref{contribution:mfenkf}. Finally, in~\cref{ssec:AEnKF}, we make use of the results of the previous subsections to state a theorem quantifying the error between the (finite particle) ensemble Kalman filter itself and the true filter, yielding Contribution \ref{contribution:enkf}. Various technical results, used in the
proof of our three main theorems, may be found in the appendices. We conclude with closing remarks in~\cref{sec:discussion}.

\subsection{Notation}
\label{ssec:N}

We denote by~$\vecnorm{\placeholder}$ the Euclidean norm on $\real^n$,
while the induced operator norm on matrices is denoted by~$\matnorm{\placeholder}$.
For symmetric positive definite matrix $\mat S \in \real^{n \times n}$,
we denote by~$\vecnorm{\placeholder}_{\mat S}$ the weighted norm $\vecnorm{ \mat S^{-1/2} \placeholder }$.
Given a function $f \colon \real^n \to \real$ and $r \geq 0$,
we let $B_{L^{\infty}}(f,r)$ denote the $L^\infty$ ball of radius $r$ centered at~$f$.
% i.e. the set of all $\mathfrak f\in L^\infty$ satisfying
% \(
%     \|\mathfrak f-f\|_{L^\infty(\real^{n})}\leq r.
% \)
Similarly, for functions $f, g$ and $r\geq 0$,
we denote by $B_{L^{\infty}}\bigl((f,g),r\bigr)$ the $L^\infty$ ball of radius $r$ centered at $(f, g)$,
on the product space, i.e. the set of all~$(\mathfrak f, \mathfrak g) \in L^\infty(\real^n) \times L^\infty(\real^d)$ satisfying
\[
    \|\mathfrak f-f\|_{L^\infty(\real^{n})}\leq r,\quad \|\mathfrak g-g\|_{L^\infty(\real^{d})}\leq r.
\] We write $|\placeholder|_{C^{0,1}}$ for the $C^{0,1}$ semi-norm, referring in particular to the Lipschitz constant.

We apply the notation $\ind$ to denote independence of two random variables.
We write $\normal(\vect m, \mat C)$
to denote the Gaussian distribution with mean~$\vect m \in \real^n$ and covariance~$\mat C \in \real^{n \times n}$.
We use $\mathcal P(\real^n)$ to denote the space of probability measures over~$\real^n$,
while we write $\mathcal G(\real^n)$ for the space of Gaussian probability measures over~$\real^n$.
With the exception of empirical measures formed from finite ensembles, this manuscript primarily
deals with probability measures that have a Lebesgue density, because of the assumptions
concerning the noise structure in the dynamics model and the data acquisition model. In this setting
we abuse notation by using the same symbols for probability measures and their densities.
For~$\mu \in \mathcal P(\real^n)$,
the notation $\mu(x)$ for $x \in \real^n$ refers to the Lebesgue density of~$\mu$ evaluated at~$x$,
while~$\mu[f]$ for a function~$f\colon \real^n \to \real$ will be notation for $\int_{\real^n} f \, \d \mu$.
For function $f:\real^n\to\real^m$ and measure $\mu \in \mathcal P(\real^n)$, we denote by $f_\sharp\mu \in \mathcal P(\real^m)$ the measure given by the pushforward, under $f$, of $\mu$.

For a probability measure~$\mu \in \mathcal P(\real^n)$,
the notation $\mathcal{M}_q(\mu)$ denotes the $q$th polynomial moment under the measure $\mu$, defined as
\begin{equation}
\label{eq:mmoments}
\mathcal{M}_q(\mu) \coloneqq \int_{\real^{n}}|x|^q \mu(\d x).
\end{equation}
We denote by~$\mathcal M(\mu)$ and~$\mathcal C(\mu)$ the mean and covariance under $\mu$, respectively:
\[
    \mathcal M(\mu) = \mu[x], \qquad
    \mathcal C(\mu) = \mu \Bigl[\bigl(x - \mathcal M(\mu)\bigr) \otimes \bigl(x - \mathcal M(\mu)\bigr)\Bigr].
\]
We use $\mathcal P_R(\real^n)$ with $R \geq 1$ to denote the subset of $\mathcal P(\real^n)$ of probability measures
with mean and covariance satisfying the bounds
\begin{equation}
    \label{eq:space_local_lipschitz}
    \lvert \mathcal M(\mu) \rvert \leq R,
    \qquad
    \frac{1}{R^2} \mat I_n \preccurlyeq \mathcal C(\mu) \preccurlyeq R^2 \mat I_n,
\end{equation}
where $\mat I_n$ denotes the identity matrix in $\real^{n \times n}$,
and~$\preccurlyeq$ is the partial ordering defined by the convex cone of positive semi-definite matrices.
Additionally, $\mathcal G_R(\real^n)$ is the subset of $\mathcal G(\real^n)$ of probability measures satisfying~\eqref{eq:space_local_lipschitz}.

For $\pi \in \mathcal P(\real^{\dimu} \times \real^{\dimy})$ defined on the joint space of state and data
associated with random variable $(u,y) \in \real^{\dimu} \times \real^{\dimy}$
we denote by $\mathcal M^u(\pi)$, $\mathcal M^y(\pi)$ the means of the marginal distributions,
and by $\mathcal C^{uu}(\pi)$, $\mathcal C^{uy}(\pi)$ and~$\mathcal C^{yy}(\pi)$
the blocks of the covariance matrix~$\mathcal C(\pi)$.
In particular, we write
\begin{equation}
\label{eq:covj}
    \mathcal M(\pi) =
    \begin{pmatrix}
        \mathcal M^u(\pi) \\
        \mathcal M^y(\pi)
    \end{pmatrix},
    \qquad
    \mathcal C(\pi) =
    \begin{pmatrix}
        \mathcal C^{uu}(\pi) & \mathcal C^{uy}(\pi) \\
        \mathcal C^{uy}(\pi)^\t & \mathcal C^{yy}(\pi)
    \end{pmatrix}.
\end{equation}
For $h\colon \real^{\dimu} \to \real^{\dimy}$ we also define ${\mathcal C}^{hh}(\pi)$ to be the covariance of vector $h(u)$,
for $u$ distributed according to the marginal of $\pi$ on $u$,
and ${\mathcal C}^{uh}(\pi)$ to be the covariance between $u$ and $h(u)$. Given this notation, we also introduce the set $\mathcal P^2_{\succ 0}(\real^{\dimu \times \dimy})$ of probability measures~$\pi$ with finite second moment satisfying $\mathcal C^{yy}(\pi) \succ 0$. %\ams{Do you mean $\mathcal C(\pi) \succ 0$?}

Throughout, we will denote operators acting on the space of probability measures via the $\mathsf{mathsf}$ font. We introduce the Gaussian projection operator
$\op G \colon \mathcal P(\real^n) \to \mathcal G(\real^n)$ defined by
$\op G \mu = \normal\bigl(\mathcal M(\mu), \mathcal C(\mu)\bigr).$
We refer to $\op G$ as a projection because $\op G \mu$ is the Gaussian distribution closest to $\mu$ with respect to~$\kl{\mu}{\placeholder}$~\cite{MR2247587},
where~$\kl{\mu}{\nu}$ denotes the Kullback--Leibler (KL) divergence of $\mu$ from $\nu$.
We note that~$\op G$ defines a nonlinear mapping.
Throughout the paper we make use of the following weighted total variation distance as in \cite{carrillo2022ensemble}:
\begin{definition}
    \label{definition:weighted_total_variation}
    Let $g\colon \real^n \to [1, \infty)$ define the function $g(v) := 1 + \vecnorm{v}^2$.
    The weighted total variation metric~$d_g \colon \mathcal P(\real^n) \times \mathcal P(\real^n) \to \real$ is defined by
    \[
    d_g(\mu_1, \mu_2) = \sup_{|f|\leq g}\bigl| \mu_1[f]-\mu_2[f]\bigr|,
    \]
    where the supremum is taken over all functions $f\colon \real^n \to \real$ bounded from above by~$g$
    pointwise and in absolute value.
\end{definition}

\begin{remark}
    \label{remark:total_variation}
    The following remarks concern the distance $d_g$:
    \begin{itemize}
        \item
            If $\mu_1, \mu_2$ have Lebesgue densities $\rho_1$, $\rho_2$,
            then
            \[
                d_g(\mu_1, \mu_2) = \int g(v) \abs*{\rho_1(v) - \rho_2(v)} \, \d v.
            \]
        \item
            Unlike the usual total variation distance,
            the weighted metric in~\cref{definition:weighted_total_variation}
            enables control of the differences $\lvert \mathcal M(\mu_1) - \mathcal M(\mu_2) \rvert$ and $\lVert \mathcal C(\mu_1) - \mathcal C(\mu_2) \rVert$.
            This is the content of~\cref{lemma:moment_bound},
            stated in~\cref{sub:moment_bounds},
            which is used in the proof of a key auxiliary result~(\cref{lemma:lipschitz_mean_field_affine}). %\ec{(It said that it was used for Lemma B.2 before.)}
            \end{itemize}
\end{remark}

\subsection{Filtering Distribution}
\label{ssec:FD}

Here we introduce the hidden Markov model that gives rise to the filtering distribution. We employ
notation similar to that established in \cite{calvello2025ensemble,carrillo2022ensemble}.
We consider $\{u_j\}_{j \in \llbracket 0, J \rrbracket}\subset \real^{\dimu}$ to be unknown states
to be determined from associated observations $\{y_j\}_{j \in \llbracket 1, J \rrbracket}\subset \real^{\dimy}$. We assume the states and observations to be governed by the following stochastic dynamics and data model:
\begin{subequations}
    \label{eq:filtering}
    \begin{alignat}{2}
        \label{eq:stochasic_dynamics}
        u_{j+1} &= \vect \Psi(u_j) + \xi_j, \qquad &&\xi_j \sim \normal(0, \mat \Sigma), \\
        \label{eq:data_model}
        y_{j+1} &= \vect h(u_{j+1}) + \eta_{j+1}, \qquad &&\eta_{j+1} \sim \normal(0, \mat \Gamma).
    \end{alignat}
\end{subequations}
We assume that the initial state is a Gaussian random variable $u_0 \sim \normal(\vect m_0, \mat C_0)$ and that the following independence condition is satisfied:
\begin{equation}
    \label{eq:independence}
    u_0 \ind \xi_0 \ind \dotsb \ind \xi_{J-1} \ind \eta_1 \ind \dotsb \ind \eta_{J}.
\end{equation}
We define the conditional distribution of the state $u_j$ given
a realization $Y_j^\dagger := \{y_1^\dagger, \dotsc, y_j^\dagger \}$ as the \textit{filtering distribution}, which we denote by $\mu_j$. We also refer to $\mu_j$ as the \textit{true filter}.
In the context of this paper, data~$Y_j^\dagger$ is assumed to be a collection of realizations from \eqref{eq:filtering}. %; in general, the case of model misspecification, where~$Y_j^\dagger$ does not necessarily arise from \eqref{eq:filtering}, is also of interest.
As formulated in~\cite{MR3363508,reich2015probabilistic}, the evolution of the filtering distribution may be written as
\begin{equation}
    \label{eq:true_filtering}
    \mu_{j+1} = \op L_j \op P \mu_j,
\end{equation}
where $\op P$ and $\op L_j$ are maps from $\mathcal P(\real^{\dimu})$ into itself and which effect what are
referred to in
the data assimilation community as the \emph{prediction} and \emph{analysis} steps, respectively \cite{asch2016data}.
The prediction operator $\op P$ is linear and is determined by the Markov kernel
arising from~\eqref{eq:stochasic_dynamics}. On the other hand, the operator $\op L_j$ is nonlinear and encodes the incorporation of the new data point~$y^\dagger_{j+1}$ using Bayes' theorem. These operators are defined via action on probability measures with Lebesgue density~$\mu$ as
\begin{subequations}
\begin{align}
    \label{eq:Markov_kernel}
    \op P \mu(u) &= \frac{1}{\sqrt{(2\pi)^{\dimu} \det \mat \Sigma}}\int_{\real^{\dimu}} \exp \left( - \frac{1}{2} \vecnorm*{u - \vect \Psi(v)}_{\mat \Sigma}^2 \right) \mu(v)\, \d v,\\
    \label{eq:decomposition_of_the_analysis}
    \op L_j \mu(u) &=
    \frac
    {\displaystyle \exp \left( - \frac{1}{2} \bigl\lvert y_{j+1}^{\dagger} - \vect h(u) \bigr\rvert_{\mat \Gamma}^2 \right)  \mu(u)}
    {\displaystyle \int_{\real^{\dimu}} \exp \left( - \frac{1}{2} \bigl\lvert y_{j+1}^{\dagger} - \vect h(U) \bigr\rvert_{\mat \Gamma}^2 \right)  \mu(U) \, \d U}
\end{align}
\end{subequations}
%We note that the operator $\op L_j$ reweighs the measure to which it applies, by assigning more weight to the state values that are consistent with the observation.
We may rewrite the analysis map $\op L_j$
as the composition $\op B_j \op Q$,
where the operators $\op Q\colon \mathcal P(\real^{\dimu}) \to \mathcal P(\real^{\dimu} \times \real^{\dimy})$
and $\op B_j \colon \mathcal P(\real^{\dimu} \times \real^{\dimy}) \to \mathcal P(\real^{\dimu})$ are defined by
\begin{subequations}
\begin{align}
    \label{eq:definition_Q}
    \op Q \mu (u, y) &= \frac{1}{\sqrt{(2\pi)^{\dimy} \det \mat \Gamma}} \exp \left( - \frac{1}{2} \bigl\lvert y - \vect h(u) \bigr\rvert_{\mat \Gamma}^2 \right)  \mu(u),\\
    \label{eq:definition_B}
    \op B_j \mu(u) &= \frac{\mu(u, y_{j+1}^{\dagger})}{\displaystyle \int_{\real^{\dimu}} \mu(U, y_{j+1}^{\dagger}) \, \d U}.
\end{align}
\end{subequations}
Linear operator $\op Q$ maps a probability measure with density $\mu$ into the law of random variable $\bigl(U,\vect h(U) + \eta\bigr)$, with $U \sim \mu$ independent of~$\eta \sim \normal(0, \mat \Gamma)$. Nonlinear operator~$\op B_j$ effects conditioning on the datum~$y^\dagger_{j+1}$. Hence, we may reformulate \eqref{eq:true_filtering} as
\begin{equation}
    \label{eq:true_filtering2}
    \mu_{j+1} = \op B_j \op Q \op P \mu_j.
\end{equation}

\section{The Ensemble Kalman Filter}
\label{sec:EKF}

We begin in~\cref{ssec:MF} by introducing the specific form of the mean field ensemble Kalman filter
that we consider in this paper; other versions leading to implementable algorithms and for which similar analysis may be developed can be found in \cite{calvello2025ensemble}. In this subsection we also outline the various assumptions that will be employed for the results of the paper. We note that the subsections that follow proceed with increasingly restrictive assumptions on dynamics and observation operator. However, all of the theorems allow for linear growth of the vector fields
defining the dynamics and the observation processes. In~\cref{ssec:SEKF} we state and prove a stability theorem, which shows that the error
between the true filter and the mean field ensemble Kalman filter may be controlled by the error between the true filter and its Gaussian projection on the joint space of state and observations.
In~\cref{ssec:AEKF} we establish an error estimate which shows that the mean field
ensemble Kalman filter provides an accurate approximation of the true filtering distribution for dynamics and observation operators that are close-to-linear. In~\cref{ssec:AEnKF} we deploy the results of the two preceding subsections to state a theorem quantifying the error between the finite particle ensemble Kalman filter itself and the true filtering distribution.

\subsection{The Algorithm}
\label{ssec:MF}

The ensemble Kalman filter as implemented in practice may be derived as a particle approximation of
various mean field dynamics~\cite{calvello2025ensemble}.
The particular mean field ensemble Kalman filter that
we consider in this paper may be written as
\begin{subequations}
\label{eq:ensemble_kalman_mean_field}
\begin{alignat}{2}
    \widehat u_{j+1} &= \vect \Psi (u_j) + \xi_j, \qquad & \xi_j \sim \normal(0, \mat \Sigma), \\
    \widehat y_{j+1} &= \vect h(\widehat u_{j+1}) + \eta_{j+1},  \qquad & \eta_{j+1} \sim \normal(0, \mat \Gamma), \\
    \label{eq:mean_field_ekf_analysis}
    u_{j+1} &= \widehat u_{j+1} + \mathcal C^{uy}\left(\widehat \pi_{j+1}^{\kalman}\right) \mathcal C^{yy}\left(\widehat \pi_{j+1}^{\kalman}\right)^{-1} \bigl(y^\dagger_{j+1} - \widehat y_{j+1} \bigr),
\end{alignat}
\end{subequations}
where $\widehat \pi^{\kalman}_{j+1} = {\rm Law}(\widehat u_{j+1}, \widehat y_{j+1})$
and where the independence condition~\eqref{eq:independence} holds.
We refer the reader back to~\cref{ssec:N} for the definition of the covariance matrices appearing in~\eqref{eq:mean_field_ekf_analysis}.
We denote by~$\mu^{\kalman}_j$ the law of $u_j$.
We aim to formulate the evolution of $\mu_j^{\kalman}$ in terms of maps on probability measures, hence we introduce,
for a given $y^\dagger_{j+1}$,
the map~$\op T_j \colon \mathcal P^2_{\succ 0}(\real^{\dimu} \times \real^{\dimy}) \to \mathcal P(\real^{\dimu})$ defined by
\begin{equation}\label{eq:Tmap}
\op T_j\pi =  \mathscr T(\placeholder, \placeholder; \pi, y^\dagger_{j+1})_{\sharp} \pi.
\end{equation}
Here for given $\pi\in {\mathcal P^2_{\succ 0}}(\real^{\dimu} \times \real^{\dimy})$, $z\in \real^{\dimy}$, the transport map $ \mathscr T$ is defined as
\begin{subequations}
  \label{eq:wenan}
\begin{align}
    \mathscr T(\placeholder, \placeholder; \pi, z) \colon
    &\real^{\dimu} \times \real^{\dimy} \to \real^{\dimu}; \\
    \label{eq:mean_field_map}
    (u, y) \mapsto &\,\, u + \mathcal C^{uy}(\pi) \mathcal C^{yy}(\pi)^{-1} \bigl(z - y \bigr).
\end{align}
\end{subequations}
Evolution of the probability measure $\mu_j^{\kalman}$ may then be written (see \cite{calvello2025ensemble}) as
\begin{equation}
    \label{eq:compact_mf_enkf}
    \mu_{j+1}^{\kalman} = \op T_j\op Q \op P \mu_j^{\kalman}.
\end{equation}
%We now discuss this evolution on probability measures in relation to \eqref{eq:true_filtering2}.
We note that a specific instance of map $\mathscr T$ defined in \eqref{eq:wenan} appears in~\eqref{eq:mean_field_ekf_analysis} and is determined by probability measure~$\pi:=\widehat \pi^{\kalman}_{j+1}$ (here equal to $\op Q \op P \mu_j^{\kalman}$) and data $z:=y^\dagger_{j+1}$.
We refer to \cite{carrillo2022ensemble} for a step-by-step argument detailing  why the law of~$u_j$ in~\eqref{eq:ensemble_kalman_mean_field} evolves according to~\eqref{eq:compact_mf_enkf}.
The operator~$\op T_j$ is equivalent, over the Gaussians $\mathcal G(\real^{\dimu} \times \real^{\dimy}) \subset \mathcal P(\real^{\dimu} \times \real^{\dimy}),$ to the conditioning operator~$\op B_j$; see~\cref{lemma:mean_field_map}.
We recall that in the particular case where~$\mu_0$ is Gaussian, which we assume throughout the paper,
and the operators~$\vect \Psi$ and~$\vect h$ are linear,
the mean field ensemble Kalman filter~\eqref{eq:compact_mf_enkf} coincides with the filtering distribution~\eqref{eq:true_filtering2}. %, given by the Kalman filter itself.
However, this paper focuses on the aim of analyzing the accuracy of ensemble Kalman methods when~$\vect \Psi$ and~$\vect h$ are not assumed to be linear.

The ensemble Kalman filter as implemented in practice may be derived as a particle approximation of the mean field dynamics as defined by sample paths in~\eqref{eq:ensemble_kalman_mean_field} or as an evolution on measures in~\eqref{eq:compact_mf_enkf}.
To this end, we observe that for any $\pi$ in the range of $\op Q$,
$\mathcal C^{yy}(\pi)=\mathcal C^{hh}(\pi)+\Gamma$ and $\mathcal C^{uy}(\pi)= \mathcal C^{uh}(\pi)$, using the notation
introduced in, and following, \eqref{eq:covj}. This is since the noise in the observation component defining $\pi$ is then independent of the state component defining $\pi.$
Using this observation, the particle approximation of
\eqref{eq:ensemble_kalman_mean_field} takes the form
\begin{subequations}
\label{eq:ensemble_kalman_particle}
\begin{alignat}{2}
    \widehat u^{(i)}_{j+1} &= \vect \Psi \bigl(u^{(i)}_j\bigr) + \xi^{(i)}_j, \qquad & \xi^{(i)}_j \sim \normal(0, \mat \Sigma), \\
    \widehat y^{(i)}_{j+1} &= \vect h\bigl(\widehat u^{(i)}_{j+1}\bigr) + \eta^{(i)}_{j+1},  \qquad & \eta^{(i)}_{j+1} \sim \normal(0, \mat \Gamma), \\
    u^{(i)}_{j+1} &= \widehat u^{(i)}_{j+1} + \mathcal C^{uh}\left(\widehat \pi_{j+1}^{\kalmanN}\right) \Bigl(\mathcal C^{hh}\Bigl(\widehat \pi_{j+1}^{\kalmanN}\Bigr)+\Gamma\Bigr)^{-1} \bigl(y^\dagger_{j+1} - \widehat y^{(i)}_{j+1} \bigr).
\end{alignat}
\end{subequations}
where $\widehat \pi^{\kalmanN}_{j+1}$ is the empirical measure
\[
    \widehat \pi^{\kalmanN}_{j+1} = \frac1N\sum_{i=1}^N\delta_{\bigl(\widehat u^{(i)}_{j+1}, \widehat y^{(i)}_{j+1}\bigr)},
\]
and $\xi^{(i)}_j \sim \normal(0, \mat \Sigma)$ i.i.d.\ in both $i$ and $j$ and $\eta^{(i)}_{j+1} \sim \normal(0, \mat \Gamma)$ i.i.d.\ in both $i$ and $j$; furthermore, the set of $\{\xi^{(i)}_j\}$ are independent from the set of $\{ \eta^{(i)}_{j+1}\}$. Choosing to express the particle approximation of the covariance in observation space through $\mathcal C^{hh}$ and $\Gamma$ ensures invertibility, provided $\Gamma$ is invertible.
From the particles evolving according to the dynamics in~\eqref{eq:ensemble_kalman_particle} we define the empirical measure
\begin{equation}
\label{eq:kalman_measure}
\mu_{j+1}^{\kalmanN} = \frac1N\sum_{i=1}^N\delta_{u_{j+1}^{(i)}},
\end{equation}
whose evolution describes the ensemble Kalman filter.

Our theorems  in~\cref{ssec:SEKF,ssec:AEKF} regard the relationship between
the true filter \eqref{eq:true_filtering2} and the mean field ensemble Kalman filter
\eqref{eq:compact_mf_enkf}. In~\cref{ssec:SEKF} we consider the setting in which $\vect \Psi$ and~$\vect h$ exhibit linear growth but are not assumed to be linear, hence the true filter is not Gaussian; in~\cref{ssec:AEKF}
we then consider small perturbations of the Gaussian setting that arise when the vector fields
$\vect \Psi$ and~$\vect h$  are close to affine. The theorem in~\cref{ssec:AEnKF} concerns the relationship between the true filter \eqref{eq:true_filtering2} and the ensemble Kalman filter \eqref{eq:kalman_measure}. In this subsection, we combine existing analysis
on the convergence of the ensemble Kalman filter to the mean field ensemble Kalman filter with the results from the previous subsections to state and prove an error estimate between the ensemble Kalman filter and the true filter in the nonlinear setting. Specifically, we study the case of a vector field $\vect \Psi$ that is a bounded perturbation away from affine and an affine vector field $\vect h$. To state our theorems we will use the following set of assumptions.

\begin{assumption}
    \label{assumption:ensemble_kalman}
    There exist positive constants $\kappa_y, \kappa_\Psi, \kappa_h, \ell_h, \sigma$ and $\gamma$ such that the data $\{y_j^\dagger\}$, the vector fields $(\vect \Psi, \vect h)$
    and the covariances $(\mat \Sigma, \mat \Gamma)$ satisfy:

    \begin{assumpenum}[label=(H\arabic*)]
    \item
        \label{assumpenum:assumption1}
        data $Y^\dagger =\{y_j^\dagger\}_{j=1}^J$ lies in set $B_y \subset \real^{KJ}$ defined by
        \[
            B_y := \left\{Y^\dagger \in \real^{KJ}: \max_{j \in \range{1}{J}} \lvert y^\dagger_j \rvert \le \kappa_y\right\};
        \]

    \item
        {
        \label{assumpenum:assumption5}
        covariance matrices $\mat \Sigma$ and $\mat \Gamma$ are positive definite:
        $\mat \Sigma \succcurlyeq \sigma \mat I_{\dimu}$ and~$\mat \Gamma \succcurlyeq \gamma \mat I_{\dimy}$ for positive $\sigma$ and $\gamma$.
    }

    \item
        \label{assumpenum:assumption2}
        function $\Psi \colon \real^{\dimu} \to \real^{\dimu}$ satisfies $\bigl\lvert \Psi(u) \bigr\rvert \leq \kappa_{\Psi} \bigl( 1  + \abs{u}  \bigr)$ for all $u\in\real^{d_u}$;
    \item
        \label{assumpenum:assumption3}
        function $h\colon \real^{\dimu} \to \real^{\dimy}$ satisfies $\bigl\lvert h(u) \bigr\rvert \leq \kappa_{h} \bigl( 1  + \abs{u}  \bigr)$ for all $u\in\real^{d_u}$;

    \item
        \label{assumpenum:assumption_lipschitz}
        function $h\colon \real^{\dimu} \to \real^{\dimy}$ satisfies $|h|_{C^{0,1}} \leq \ell_{\vect h} < \infty$;
    \end{assumpenum}
\end{assumption}

\renewcommand{\theassumption}{V}
\begin{assumption}
    \label{assumption:affine_vector_fields}
    The vector fields $(\Psi,h)$ are affine, i.e. they satisfy the following

    \begin{assumpenum}[label=(V\arabic*)]

    \item
        \label{assumpenum:affine_assumption1}
        The function $\Psi \colon \real^{\dimu} \to \real^{\dimu}$ satisfies $ \Psi(u) \coloneqq Mu + b$, for some $M\in\real^{\dimu \times \dimu}$ and $b\in \real^{\dimu}$.

    \item
        \label{assumpenum:affine_assumption2}
        The function $h\colon \real^{\dimu} \to \real^{\dimy}$ satisfies $ h(u) \coloneqq Hu + w$, for some $H\in \real^{\dimy \times \dimu}$ and $w\in\real^{\dimy}$.

    \end{assumpenum}

\end{assumption}

\subsection{Stability Theorem:  Mean Field Ensemble Kalman Filter}
\label{ssec:SEKF}

In this section we consider the setting of dynamics and observation operator satisfying the linear growth assumptions contained in~\cref{assumption:ensemble_kalman}. Informally the result of \cref{theorem:main_theorem} shows that, under these assumptions, and if the true filtering distributions $(\mu_j)_{j \in \range{0}{J-1}}$ are close to Gaussian after lifting to the joint state and data space,
then the distribution $\mu^{\kalman}_j$ given by the mean field ensemble Kalman filter~\eqref{eq:compact_mf_enkf}
is close to the filtering distribution~$\mu_j$  given by \eqref{eq:true_filtering2}
for all $j \in \llbracket 0, J \rrbracket$.

\begin{theorem}
    [Stability: Mean Field Ensemble Kalman Filter]
    \label{theorem:main_theorem}
    Assume that the probability measures~$(\mu_j)_{j \in \range{0}{J}}$ and  $(\mu^{\kalman}_j)_{j \in \range{0}{J}}$ are given respectively by the dynamical systems~\eqref{eq:true_filtering2} and~\eqref{eq:compact_mf_enkf},
    initialized at the Gaussian probability measure $\mu_0 = \mu_0^{\kalman}\in~\mathcal G(\real^{\dimu})$.
    That~is,
    \[
        \mu_{j+1} = \op B_j \op Q \op P \mu_j, \qquad
        \mu^{\kalman}_{j+1} = \op T_j \op Q \op P \mu_j^{\kalman}.
    \]
    If \cref{assumption:ensemble_kalman} holds,
    then there exists~$C = C\bigl(\mathcal{M}_{\max\{3+\dimu, 4+\dimy\}}(\mu_0),\kappa_y, \kappa_{\vect \Psi}, \kappa_{\vect h}, \ell_{\vect h}, \mat \Sigma, \mat \Gamma, J\bigr)$
    such that
    \[
        d_g(\mu^{\kalman}_J, \mu_J) \leq C  \max_{j \in \range{0}{J-1}} d_g( \op Q \op P \mu_j, \op G \op Q \op P \mu_j).
    \]
\end{theorem}

Proof of Theorem \ref{theorem:main_theorem} below relies on the following auxiliary results,
all proved in~\cref{appendix:A}.
\begin{enumerate}
    \item
        \label{item1}
        For any probability measure~$\mu$ with finite first and second order polynomial moments $\mathcal{M}_1(\mu)$ and $\mathcal{M}_2(\mu)$, the means of~$\op P \mu$ and~$\op Q \op P \mu$ are bounded from above,
        and their covariances are bounded both from above and from below.
        The constants in these bounds depend only on the parameters~$\kappa_{\vect \Psi}$, $\kappa_{\vect h}$, $\mat \Sigma, \mat \Gamma$ and on $\mathcal{M}_1(\mu)$ and $\mathcal{M}_2(\mu)$. See~\cref{lemma:bound_on_Pmu_new,lemma:bound_on_QPmu_new}.

    \item
        \label{item2}
        Let $(\mu_j)_{j \in \range{1}{J}}$ and  $(\mu^{\kalman}_j)_{j \in \range{1}{J}}$ denote the probability measures obtained respectively from the dynamical systems~\eqref{eq:true_filtering2} and~\eqref{eq:compact_mf_enkf},
        initialized at the same Gaussian measure $\mu_0 = \mu_0^{\kalman} \in~\mathcal G(\real^{\dimu})$.
        Then for any integer $q\geq 2$,
        there exist constants $M_q,M_q^{\kalman}<\infty$ depending on $\mathcal{M}_q(\mu_0),\kappa_y, \kappa_{\vect \Psi}, \kappa_{\vect h}, \mat \Sigma, \mat \Gamma, J$ so that
        \[
        \max_{j \in \range{0}{J}} \mathcal{M}_q(\mu_j)\leq M_q \quad \text{and} \quad \max_{j \in \range{0}{J}} \mathcal{M}_q\bigl(\mu^{\kalman}_j\bigr)\leq M_q^{\kalman}.
        \]
        See~\cref{lemma:true_filter_bounded_moments,theorem:approximate_filter_bounded_moments}. This will facilitate use of the stability results from~\cref{item6,item7}.

    \item
        \label{item3}
        For any Gaussian measure $\mu \in \mathcal G(\real^{\dimu} \times \real^{\dimy})$,
        we have that~$\op B_j \mu = \op T_j \mu$.
        See~\cref{lemma:mean_field_map} and~\cite{calvello2025ensemble,carrillo2022ensemble}.

     \item
        \label{item4}
        The map~$\op P$ is Lipschitz on $\mathcal P(\real^{\dimu})$ for the metric $d_g$. The Lipschitz constant~$L_{\op P}$ depends only on the parameters $\kappa_{\vect \Psi}$ and $\mat\Sigma$.
        See~\cref{lemma:lipschitz_p}.

    \item
        \label{item5}
        The map~$\op Q$ is Lipschitz on $\mathcal P(\real^{\dimu})$ for the metric $d_g$. The Lipschitz constant~$L_{\op Q}$ depends only on the
        parameters $\kappa_{\vect h}$ and $\mat \Gamma$.
        See~\cref{lemma:lipschitz_Q}.

    \item
        \label{item6}
        The map $\op B_j$ satisfies that, for all $\pi \in\{\op Q \op P \mu: \mu \in \mathcal P(\real^{\dimu}) \,\text{and }\mathcal{M}_{2}(\mu) < \infty\}\subset \mathcal P(\real^{\dimu}\times\real^{\dimy})$, it holds
        \[
            \forall j \in \llbracket 0, J \rrbracket,
            \qquad
            d_{g}(\op B_j \op G \pi, \op B_j \pi)
            \leq C_{\op B} d_{g}(\op G \pi, \pi),
        \]
        where $C_{\op B} = C_{\op B}\bigl(\mathcal{M}_{2}(\mu), \kappa_y, \kappa_{\vect \Psi}, \kappa_{\vect h}, \mat \Sigma, \mat \Gamma\bigr)$.
        Namely, this item regards the stability of the operator $\op B_j$ between a probability measure and its Gaussian projection.
        See~\cref{lemma:missing_lemma}.

    \item
        \label{item7}
        The map $\op T_j$ satisfies the following bound:
        for all $R \ge 1$, it holds for all~$\pi \in \mathcal P_R(\real^{\dimu} \times \real^{\dimy})$ and $p \in \{\op Q \op P \mu: \mu \in \mathcal P(\real^{\dimu}) \,\text{and }\mathcal{M}_{2\cdot\max\{3+\dimu, 4+\dimy\}}(\mu) < \infty\}\subset \mathcal P(\real^{\dimu}\times\real^{\dimy})$ that
        \[
            \forall j \in \llbracket 0, J \rrbracket,
            \qquad
            d_g(\op T_j \pi, \op T_j p)
            \leq L_{\op T} \, d_g(\pi, p),
        \]
        for a constant $L_{\op T} = L_{\op T}\bigl(R,\mathcal{M}_{2\cdot\max\{3+\dimu, 4+\dimy\}}(\mu), \kappa_y, \kappa_{\vect \Psi}, \kappa_{\vect h}, {\ell_h}, \mat \Sigma, \mat \Gamma\bigr)$.
        Namely, this item may be regarded as a local Lipschitz continuity result. %in the case where the second argument of $d_g$ is restricted to the range of $\op Q \op P$.
        See~\cref{lemma:lipschitz_mean_field_affine}.
\end{enumerate}

\begin{proof}
    [Proof of \cref{theorem:main_theorem}]
    Within the following argument we make reference to the list of items outlined above. We begin by defining for ease of exposition the following measure of difference
    between the true filter and its Gaussian projection:
    \begin{equation}
        \label{eq:epsilon}
        \varepsilon:=  \max_{j \in \range{0}{J-1}} d_g( \op Q \op P \mu_j, \op G \op Q \op P \mu_j).
    \end{equation}
    We assume throughout the argument that $j \in \range{0}{J-1}.$
    The proof rests upon the following application of the triangle inequality:
    \begin{subequations}
    \begin{align}
        \notag
        d_g(\mu^{\kalman}_{j+1}, \mu_{j+1})
        &= d_g\left(\op T_j \op Q \op P \mu_j^{\kalman}, \op B_j \op Q \op P \mu_j\right) \\
        &\leq d_g\left(\op T_j \op Q \op P \mu_j^{\kalman}, \op T_j \op Q \op P \mu_j\right) +  d_g\left(\op T_j \op Q \op P \mu_j, \op T_j \op G \op Q \op P \mu_j\right)
        +  d_g\left(\op B_j \op G \op Q \op P \mu_j, \op B_j \op Q \op P \mu_j\right).
        \label{eq:main_idea_main_proofB}
    \end{align}
    \end{subequations}
    Here we have used the fact that $\op T_j \op G \op Q \op P \mu_j = \op B_j \op G \op Q \op P \mu_j$ by~\cref{item3} (\cref{lemma:mean_field_map}).
    \Cref{item2} (\cref{lemma:true_filter_bounded_moments,theorem:approximate_filter_bounded_moments}) shows that, for any integer $q\geq 2$, there exist constants $M_q,M_q^{\kalman}<\infty$ depending on $\mathcal{M}_q(\mu_0),\kappa_y, \kappa_{\vect \Psi}, \kappa_{\vect h}, \mat \Sigma, \mat \Gamma, J$ so that
        \[
            \max_{j \in \range{0}{J}} \mathcal{M}_q(\mu_j)\leq M_q \quad \text{and} \quad \max_{j \in \range{0}{J}} \mathcal{M}_q\bigl(\mu^{\kalman}_j\bigr)\leq M_q^{\kalman}.
        \]
        Therefore by~\cref{item1} (\cref{lemma:bound_on_QPmu_new}),
    there is a constant~$R \ge 1$, depending on the covariances~$\mat \Sigma$, $\mat \Gamma$, the bounds~$\kappa_{\vect \Psi}$ and~$\kappa_{\vect h}$
    from \cref{assumption:ensemble_kalman}, and the moment bounds $M_2$ and $M_2^{\kalman}$, such that for any $j\in\range{0}{J-1}$ it holds that $\op Q \op P \mu_j,\op Q \op P \mu_j^{\kalman}  \in \mathcal P_R(\real^{\dimu} \times \real^{\dimy})$.
    In view of~\cref{item4,item5,item7} (\cref{lemma:lipschitz_mean_field_affine,lemma:lipschitz_Q,lemma:lipschitz_p}), the first term in~\eqref{eq:main_idea_main_proofB} satisfies
    \begin{align}
        \label{eq:main_theorem_1}
        d_g\left(\op T_j \op Q \op P \mu_j^{\kalman}, \op T_j \op Q \op P \mu_j \right)
        &\leq L_{\op T}\bigl(R, M_{2\cdot\max\{3+\dimu, 4+\dimy\}}\bigr) L_{\op Q} L_{\op P} d_g\left(\mu_j^{\kalman}, \mu_j \right),
    \end{align}
    where, for conciseness, we have omitted the dependence of the constants in the bound on~$\kappa_y, \kappa_{\vect \Psi}, \kappa_{\vect h}, \mat \Sigma, \mat \Gamma$.
    %Equation~\eqref{eq:main_theorem_1} establishes that the composition of maps $\op T_j \op Q \op P$ is globally Lipschitz over~$\mathcal P(\real^{\dimu})$.
    By definition of $R$ it holds that~$\op G \op Q \op P \mu_j\in \mathcal G_R(\real^{\dimu} \times \real^{\dimy})$, hence
    the second term in~\eqref{eq:main_idea_main_proofB} may be bounded using \cref{item7} (\cref{lemma:lipschitz_mean_field_affine}) and the definition in of $\varepsilon$ in~\eqref{eq:epsilon}. Indeed, it holds that
    \begin{equation*}
        d_g\left(\op T_j \op Q \op P \mu_j, \op T_j \op G \op Q \op P \mu_j\right)
        \leq L_{\op T}\bigl(R,  M_{2\cdot\max\{3+\dimu, 4+\dimy\}}\bigr) d_g\left(\op Q \op P \mu_j, \op G \op Q \op P \mu_j\right)
        \le L_{\op T}\bigl(R,  M_{2\cdot\max\{3+\dimu, 4+\dimy\}}\bigr) \varepsilon.
    \end{equation*}
    Using~\cref{item6} (\cref{lemma:missing_lemma}) and the definition of $\varepsilon$ in~\eqref{eq:epsilon} we establish the following bound on the third term in~\eqref{eq:main_idea_main_proofB}:
    \[
        d_g\left(\op B_j \op G \op Q \op P \mu_j, \op B_j \op Q \op P \mu_j\right)
        \leq C_{\op B} d_g\left(\op G \op Q \op P \mu_j, \op Q \op P \mu_j\right) \le  C_{\op B} \varepsilon.
    \]
    Therefore, setting~$\ell = L_{\op T}\bigl(R, M_{2\cdot\max\{3+\dimu, 4+\dimy\}}\bigr) L_{\op Q} L_{\op P}$, we have that
    \begin{align*}
        d_g(\mu^{\kalman}_{j+1}, \mu_{j+1})
        &\leq \ell d_g(\mu^{\kalman}_{j}, \mu_{j}) + \bigl(L_{\op T}\bigl(R, M_{\max\{3+\dimu, 4+\dimy\}}\bigr) + C_{\op B}\bigr) \varepsilon.
       % &\leq \dots \leq \left(\frac{\ell^{j+1} - 1}{\ell - 1}\right) \bigl(L_{\op T}(R) + C_{\op B}\bigr) \varepsilon,
    \end{align*}
    The conclusion then follows from the discrete Gr\"onwall lemma, using the fact that $\mu_0^{\kalman} = \mu_0.$
\end{proof}

\subsection{Error Estimate: Mean Field Ensemble Kalman Filter}
\label{ssec:AEKF}

It is possible to deduce from the result of \cref{theorem:main_theorem} that the error between the true filter and the mean field ensemble Kalman filter can be arbitrarily small
if the true filter is arbitrarily close to its Gaussian projection, in state-observation space.
This condition on the true filter, which we refer to as ``closeness to Gaussian'', can be satisfied in the setting of unbounded vector fields by considering small perturbations of affine vector fields, as we prove in~\cref{prop:auxiliary_close_to_affine}. Combining~\cref{prop:auxiliary_close_to_affine} with \cref{theorem:main_theorem} gives an error estimate for the mean field ensemble Kalman filter, yielding~\cref{thm:error_estimate_mfenkf}.
% \uv{%
%     In the statements of~\cref{prop:auxiliary_close_to_affine,thm:error_estimate_mfenkf},
%     we denote by~$B_{\Psi_0,h_0}(r)$ the set of all functions $(\Psi,h)$ satisfying
%     $\Psi \in B_{L^{\infty}}(\Psi_0,r)$ and $h \in B_{L^{\infty}}(h_0,r)$.
% }

%\ams{Am I missing something? The following proposition statement seems to refer only to the affine case. Yet the proof seems to implicitly imply that an $\epsilon$ neighbourhood of affine is being considered. Indeed, in the statement, what is $\epsilon$? Furthermore the theorem that follows proceeds as if the proposition concerned a small neighbourhood of affine.}

%\uvquestion{Do we need Assumption H4 in proposition that follows? If not, no need to fix $\ell_h$ at the beginning}
\begin{proposition}[Approximation Result]
\label{prop:auxiliary_close_to_affine}
    Suppose that the data~$Y^\dagger =\{y_j^\dagger\}_{j=1}^J$ and the matrices
    $(\mat \Sigma, \mat \Gamma)$ satisfy~\cref{assumpenum:assumption1,assumpenum:assumption5}.
    Fix $\kappa_{\vect \Psi},\kappa_{\vect h}>0$ and assume that $\Psi_0 \colon \real^{\dimu} \to \real^{\dimu}$ and $\vect h_0 \colon \real^{\dimu} \to \real^{\dimy}$ are affine functions and hence satisfy~\cref{assumpenum:affine_assumption1,assumpenum:affine_assumption2}, respectively,
    while also satisfying~\cref{assumpenum:assumption2,assumpenum:assumption3} with $\kappa_{\vect \Psi},\kappa_{\vect h}$.
    Let $(\mu_{j})_{j \in \range{0}{J}}$ denote the true filtering distribution associated with
    functions~$(\Psi, \vect h)$, initialized at the  Gaussian probability
    measure $\mu_0 = \normal(m_0, C_0) \in \mathcal G(\real^{\dimu})$.
    Then,
    for any $J \in \mathbb{Z}^+$,
    there is $C = C(m_0, C_0, \kappa_y, \kappa_{\Psi}, \kappa_h, \mat \Sigma, \mat \Gamma, J) > 0$ such that for all $\varepsilon \in [0, 1]$
    and all~$(\Psi, h) \in B_{L^\infty}\bigl((\Psi_0,h_0),\varepsilon \bigr)$,
    it holds that
    \begin{equation}
        \label{eq:second_equation_corollary_new}
        \max_{j \in \range{0}{J-1}}
        d_g(\op G \op Q \op P \mu_j, \op Q \op P \mu_j)
        \leq C \varepsilon.
    \end{equation}
    \end{proposition}
\begin{proof}
In what follows $(\mu^0_{j})_{j \in \range{0}{J}}$ and $(\mu_{j})_{j \in \range{0}{J}}$ denote the true filtering distributions associated with functions~$(\Psi_0, \vect h_0)$ and~$(\Psi, \vect h)$, respectively,
initialized at the same Gaussian measure $\normal(m_0, C_0)$.
Furthermore, we let $\op P_0$ and $\op Q_0$ denote the kernel integral operators \eqref{eq:Markov_kernel} and \eqref{eq:definition_Q} defined by the specific vector fields~$(\Psi_0, \vect h_0).$
By~\cref{lemma:true_filter_bounded_moments}, the filtering distributions have bounded second moments. Let
    \[
        \mathcal R = \max_{j \in \range{0}{J-1}}
        \left( \left\lvert y_{j+1}^\dagger \right\rvert^2,\, 1 + \mathcal{M}_2\bigl(\mu_j^0\bigr),\, 1+ \mathcal{M}_2\bigl(\mu_j\bigr) \right).
    \]
    Throughout this proof, $C$ denotes a constant whose value is irrelevant in the context,
    depends only on the constants $m_0, C_0, \kappa_y, \kappa_{\Psi}, \kappa_h, \mat \Sigma, \mat \Gamma, j$
    (in particular, it does not depend on $\varepsilon$, $\vect \Psi$, $\vect h$), and may change from line to line.
    Fix $j \in \range{0}{J-1}$. Note that the filtering distribution defined by $(\Psi_0, \vect h_0)$ is Gaussian. Then,
    using the triangle inequality and Gaussianity of $\op Q_0 \op P_0 \mu_j^0$ we obtain
    \begin{align*}
        d_g(\op G \op Q \op P \mu_j, \op Q \op P \mu_j)
        &\leq d_g(\op G \op Q \op P \mu_j, \op G \op Q_0 \op P_0 \mu^0_j)
        + d_g(\op G\op Q_0 \op P_0 \mu^0_j, \op Q \op P \mu_j) \\
         &= d_g(\op G \op Q \op P \mu_j, \op G \op Q_0 \op P_0 \mu^0_j)
        + d_g(\op Q_0 \op P_0 \mu^0_j, \op Q \op P \mu_j).
    \end{align*}
    We note that since the filters have bounded first and second order polynomial moments, by~\cref{lemma:bound_on_Pmu_new,lemma:bound_on_QPmu_new} we may deduce that
    there exists $R\geq 1$ such that%
    \begin{equation}
        \label{eq:inclusion_moment_bounds}
        \forall j \in \range{0}{J-1}, \qquad
        \op P_0\mu^0_j, \in \mathcal P_R(\real^{d_u}),
        \qquad
        \op Q\op P\mu_j, \op Q_0\op P_0\mu^0_j \in \mathcal P_R(\real^{d_u}\times \real^{d_y}).
    \end{equation}%
    The second part of this display
    allows direct application of~\cref{lemma:stabilityG},
    which concerns the local Lipschitz continuity result for the Gaussian projection operator $\op G$, to obtain
    \begin{align*}
        d_g(\op G \op Q \op P \mu_j, \op Q \op P \mu_j)
        %&\leq d_g(\op G \op Q \op P \mu_j, \op G \op Q_0 \op P_0 \mu^0_j)
        %+ d_g(\op Q_0 \op P_0 \mu^0_j, \op Q \op P \mu_j)\\
         &\leq C d_g(\op Q_0 \op P_0 \mu^0_j, \op Q \op P \mu_j)\\
         &\leq C d_g(\op Q_0 \op P_0 \mu^0_j, \op Q \op P_0 \mu^0_j) + C d_g(\op Q \op P_0 \mu^0_j, \op Q \op P \mu_j).
    \end{align*}
    Using~\eqref{eq:statement_Q} from~\cref{lemma:3bounds},
        noting that since $\op P_0 \mu^0 \in \mathcal P_R(\real^{\dimu})$ by~\eqref{eq:inclusion_moment_bounds}
        it holds that~$\mathcal M_2(\op P_0 \mu^0_j) \leq \dimu R$, and using the Lipschitz continuity of~$\op Q$ (\cref{lemma:lipschitz_Q})
    we deduce that
    \begin{align*}
        d_g(\op G \op Q \op P \mu_j, \op Q \op P \mu_j)
        &\leq C \varepsilon (1 + \dimu R) + C d_g(\op P_0 \mu^0_j, \op P \mu_j) \\
        &\leq C \varepsilon (1 + \dimu R) + C d_g\bigl(\op P_0 \mu^0_{j}, \op P \mu^0_{j}\bigr) + C d_g\bigl(\op P \mu^0_{j}, \op P \mu_{j}\bigr) \\
        &\leq C \varepsilon (1 + \dimu R) + C \varepsilon \mathcal R + C d_g(\mu^0_j, \mu_j),
    \end{align*}
    where the second inequality follows by the triangle inequality and the third inequality follows from
    bounding~$d_g(\op P_0 \mu^0_j, \op P \mu_j)$
    using~\eqref{eq:statement_P} from~\cref{lemma:3bounds}
    and from the Lipschitz continuity of~$\op P$ (\cref{lemma:lipschitz_Q}). The statement then follows from~\cref{lemma:auxiliary_close_to_affine_new}.
\end{proof}

\begin{theorem} [Error Estimate: Mean Field Ensemble Kalman Filter]
    \label{thm:error_estimate_mfenkf}
    Assume that the probability measures $(\mu_j)_{j \in \range{0}{J}}$ and  $(\mu^{\kalman}_j)_{j \in \range{0}{J}}$ are given respectively by the dynamical systems~\eqref{eq:true_filtering2} and~\eqref{eq:compact_mf_enkf}
    initialized at the Gaussian measure $\mu_0 = \mu_0^{\kalman}\in~\mathcal G(\real^{\dimu})$.
    That~is,
    \[
        \mu_{j+1} = \op B_j \op Q \op P \mu_j, \qquad
        \mu^{\kalman}_{j+1} = \op T_j \op Q \op P \mu_j^{\kalman}.
    \]
    Suppose that the data~$Y^\dagger =\{y_j^\dagger\}_{j=1}^J$ and the matrices
    $(\mat \Sigma, \mat \Gamma)$ satisfy~\cref{assumpenum:assumption1,assumpenum:assumption5}.
    Let the vector fields $\vect \Psi_0,\vect h_0$ be affine functions satisfying~\cref{assumpenum:affine_assumption1,assumpenum:affine_assumption2}, respectively,
    while also satisfying~\cref{assumpenum:assumption2,assumpenum:assumption3} with some constants $\kappa_{\vect \Psi},\kappa_{\vect h} > 0$.
    Additionally, let~$\ell_{\vect \Psi}>0$ be a constant.
    Then there exists a constant ~$C = C\bigl(\mathcal{M}_q(\mu_0),\kappa_y, \kappa_{\vect \Psi}, \kappa_{\vect h}, \ell_{\vect h}, \mat \Sigma, \mat \Gamma, J\bigr)$, where~$q\coloneqq {2\cdot\max\{3+\dimu, 4+\dimy\}}$,
    such that for all~$\varepsilon\in[0,1]$ and all $(\vect \Psi,\vect h) \in B_{L^\infty}\bigl((\Psi_0,h_0),\varepsilon \bigr)$ with $h$ satisfying~\cref{assumpenum:assumption_lipschitz},
    it holds that
    \[
        d_g(\mu^{\kalman}_J, \mu_J) \leq C\varepsilon.
    \]
\end{theorem}
\begin{proof}
    Since~\cref{assumption:ensemble_kalman} is satisfied, it is possible to apply the result of~\cref{theorem:main_theorem} to deduce that there exists~$C = C\bigl(\mathcal{M}_{2\cdot\max\{3+\dimu, 4+\dimy\}}(\mu_0),J,\kappa_y, \kappa_{\vect \Psi}, \kappa_{\vect h}, \ell_{\vect h}, \mat \Sigma, \mat \Gamma\bigr)$
    such that
    \[
        d_g(\mu^{\kalman}_J, \mu_J) \leq C  \max_{j \in \range{0}{J-1}} d_g( \op Q \op P \mu_j, \op G \op Q \op P \mu_j).
    \]
    Additionally, since $(\vect \Psi,\vect h) \in B_{L^\infty}\bigl((\Psi_0,h_0),\varepsilon \bigr)$ for $\vect \Psi_0,\vect h_0$ satisfying~\cref{assumpenum:affine_assumption1,assumpenum:affine_assumption2}, and moreover ~\cref{assumpenum:assumption2,assumpenum:assumption3}, we may apply~\cref{prop:auxiliary_close_to_affine} to deduce the result.
\end{proof}

\subsection{Error Estimate: Finite Particle Ensemble Kalman Filter}
\label{ssec:AEnKF}

In this subsection, we combine the results from the work in \cite{le2009large} with stability \cref{theorem:main_theorem}, together with approximation \cref{thm:error_estimate_mfenkf},
to derive a quantitative error estimate between the
finite particle ensemble Kalman filter and the true filter in the nonlinear setting. In order to define an appropriate metric we introduce the following class of vector fields.

\renewcommand{\theassumption}{P1}
\begin{assumption}
\label{assumption:metric_vector_fields}
The vector field $\phi:\real^{\dimu}\to \real$ satisfies for any $u,v\in \real^{\dimu}$ the condition
\[
\left|\phi(u) - \phi(v)\right| \leq L_\phi |u-v|\bigl(1+|u|^{{\varsigma}}+|v|^{{\varsigma}}\bigr),
\]
for some ${\varsigma}\geq 0$ and for some $L_\phi > 0$. We note that for any such $\phi$, there exists $R_\phi>0$ which depends on $L_\phi$ so that $|\phi(u)|\leq R_\phi\bigl(1+|u|^{{\varsigma}+1} \bigr)$ for any $u\in\real^{\dimu}$.
\end{assumption}

We will prove various technical lemmas under the~\cref{assumption:metric_vector_fields}; these may be useful beyond the confines of this paper. However for our theorems we use the
more specific~\cref{assumption:metric_vector_fields_particular} which enables the control of first and second moments.

%We note that functions satisfying~\cref{assumption:metric_vector_fields} will serve as a class of test functions against which the relevant probability measures will be integrated. Since we only seek to control first and second order moments, we may consider the particular case of $\uv{\varsigma}=1$, which is the object of~\cref{assumption:metric_vector_fields_particular}. We will prove an error estimate for the ensemble Kalman filter in~\cref{theorem:enkf_theorem} for a metric that is defined using test functions satisfying~\cref{assumption:metric_vector_fields_particular}. The theorem relies on the result of~\cref{lemma:convergence_to_MF}, which can be proved for a metric that is defined using test functions satisfying the more general~\cref{assumption:metric_vector_fields}. We next introduce the particular instance of~\cref{assumption:metric_vector_fields} under which the error estimate theorem will hold.

\renewcommand{\theassumption}{P2}
\begin{assumption}
\label{assumption:metric_vector_fields_particular}
The vector field $\phi:\real^{\dimu}\to \real$ satisfies for any $u,v\in \real^{\dimu}$ the condition
\[
\left|\phi(u) - \phi(v)\right| \leq L_\phi |u-v|\bigl(1+|u|+|v|\bigr),
\]
for some $L_\phi > 0$. Note that for any such $\phi$, there exists $R_\phi>0$ depending on $L_\phi$ so that $|\phi(u)|\leq R_\phi\bigl(1+|u|^2 \bigr)$ for any $u\in\real^{\dimu}$.
\end{assumption}

\begin{theorem}
    [Error Estimate: Finite Particle Ensemble Kalman Filter]
    \label{theorem:enkf_theorem}
    Assume that the probability measures $(\mu_j)_{j \in \range{0}{J}}$ and $(\mu^{\kalmanN}_j)_{j \in \range{0}{J}}$ are obtained respectively from the dynamical systems~\eqref{eq:true_filtering2} and~\eqref{eq:kalman_measure},
    initialized at the Gaussian probability measure $\mu_0\in~\mathcal G(\real^{\dimu})$ and at the empirical measure~$\mu_{0}^{\kalmanN} =\frac1N\sum_{i=1}^N\delta_{u_0^{(i)}}$ for $u_0^{(i)}\sim \mu_0$ i.i.d.\ samples, respectively.
    That is,
    \[
        \mu_{j+1} = \op B_j \op Q \op P \mu_j, \qquad \mu_{j+1}^{\kalmanN} = \frac1N\sum_{i=1}^N\delta_{u_{j+1}^{(i)}},
    \]
    where $u_{j+1}^{(i)}$ evolve according to the iteration in~\eqref{eq:ensemble_kalman_particle}.
    Suppose that the data~$Y^\dagger =\{y_j^\dagger\}_{j=1}^J$ and the matrices~$(\mat \Sigma, \mat \Gamma)$
    satisfy~\cref{assumpenum:assumption1,assumpenum:assumption5}.
    Assume that the vector field $\vect h$ is linear,
    and let $\kappa_h, \ell_h$ be positive constants such that~\cref{assumpenum:assumption3,assumpenum:assumption_lipschitz} are satisfied.
    Furthermore, let the vector field $\vect \Psi_0$ be an affine function satisfying~\cref{assumpenum:affine_assumption1}
    as well as~\cref{assumpenum:assumption2} with $\kappa_{\Psi} > 0$.
    Additionally, let~$\ell_{\vect \Psi}>0$ be a constant.
    Then,
    there exists a constant $C = C\bigl(\mathcal{M}_q(\mu_0),R_\phi,L_\phi,\kappa_y, \kappa_{\vect \Psi}, \kappa_{\vect h}, \ell_{\vect h},\ell_{\vect \Psi}, \mat \Sigma, \mat \Gamma, J\bigr)$,
    where $q\coloneqq {2\cdot\max\{3+\dimu, 4+\dimy, 2\cdot 4^{J}\}}$,
    such that for all~$\varepsilon\in[0,1]$ and all $\vect \Psi \in B_{L^\infty}\bigl(\Psi_0,\varepsilon \bigr)$
    satisfying $|\Psi|_{C^{0,1}} \leq \ell_{\vect \Psi} < \infty$,
    the following bound holds for any $\phi$ satisfying~\cref{assumption:metric_vector_fields_particular}:
    \[
       \left(\mathbb{E}~\Bigl|\mu^{\kalmanN}_J[\phi] -\mu_J[\phi]\Bigr|^2\right)^{1/2} \leq C  \Bigl(\frac{1}{\sqrt{N}} + {\varepsilon} \Bigr).
    \]
    %\p{Does $C$ depends on $L_\phi$ ?} \ec{It does, indeed. In what I wrote, $C$ depends on $L_\phi$ through $R_\phi$. I have made more explicit the dependence of $R_\phi$ on $L_\phi$ in \cref{assumption:metric_vector_fields}. Do you agree?} \p{I am surprised that the constant depends on $L_\phi$ only through $R_\phi$. If this is the case, then we can replace Assumption P by the the condition that $|\phi(u)|\leqslant R_\phi(1+|u|^2)$}
    %\p{Shouldn't that be the expectation of the Wasserstein distance ?} \ec{I agree, thanks for the correction. I have updated Lemma 29 to give a more useful result, but the error estimate that is the object of this theorem still needs some work - look forward to discussing.}
\end{theorem}
The proof presented hereafter relies on the following elements.
\begin{enumerate}

    \item \label{item:triangle} We apply the triangle inequality in order to employ two distinct results concerning the mean field ensemble Kalman filter. The proof thus involves quantifying the error between finite particle ensemble Kalman filter and the mean field ensemble Kalman filter, as discussed in~\cref{item:mean_field_convergence}, and the error between the mean field ensemble Kalman filter and the true filter, as discussed in \cref{item:3}.

    \item \label{item:mean_field_convergence} The work in \cite{le2009large} establishes a
    Monte Carlo error estimate, with rate of $1/\sqrt{N}$,
    between the empirical measure $\mu^{\kalmanN}$,
    representing the particle ensemble Kalman filter, and the measure $\mu^{\kalman}$ describing the evolution of the mean field ensemble Kalman filter.
    In particular, this holds under~\cref{assumpenum:assumption1,assumpenum:assumption5} on the data and covariances of the noise processes,
    the linearity of the vector field $\vect h$ and~\cref{assumpenum:assumption2} on $\vect \Psi$ with the additional assumption that $|\Psi|_{C^{0,1}} \leq \ell_{\vect \Psi} < \infty$.
    In~\cref{lemma:convergence_to_MF} we provide a self-contained proof of~\cite[Theorem 5.2]{le2009large}
    % to make explicit the number of moments to be controlled in order to obtain the estimate,
    to gain insight into the dependence of the constant prefactor multiplying $1/\sqrt{N}$ on
    the parameters of the Gaussian initial condition and the number of steps~$J$.

    \item \label{item:3} We assume that the data~$Y^\dagger =\{y_j^\dagger\}_{j=1}^J$ and the matrices
    $(\mat \Sigma, \mat \Gamma)$ satisfy~\cref{assumpenum:assumption1,assumpenum:assumption5}, and assume that $\vect h$ satisfies~\cref{assumpenum:affine_assumption2}.
    Furthermore, we assume $\vect \Psi$ satisfies~\cref{assumpenum:assumption2}
    and $\vect \Psi \in B_{L^{\infty}}(\vect \Psi_0, \varepsilon)$ with $\vect \Psi_0$ satisfying~\cref{assumpenum:affine_assumption1}. These assumptions
    allow us to apply the the result from~\cref{thm:error_estimate_mfenkf}.

\end{enumerate}

\begin{proof} Recall that $\mu_j^{\kalman}$ is the mean field ensemble Kalman filter, here initialized at the same Gaussian $\mu_0$ as the true filter.
    We fix a function $\phi$ satisfying~\cref{assumption:metric_vector_fields_particular} and apply the triangle inequality to deduce that
    \begin{equation}
    \label{eq:enkf_error_triangle}
        \left(\mathbb{E}~\Bigl|\mu^{\kalmanN}_J[\phi] -\mu_J[\phi]\Bigr|^2\right)^{1/2} \leq \left(\mathbb{E}~\Bigl|\mu^{\kalmanN}_J[\phi] -\mu^{\kalman}_J[\phi]\Bigr|^2\right)^{1/2} + \left(\mathbb{E}~\Bigl|\mu^{\kalman}_J[\phi] -\mu_J[\phi]\Bigr|^2\right)^{1/2}.
    \end{equation}
    Since $\mu^{\kalman}_J$ and $\mu_J$ are deterministic probability measures, it holds that
    \[\left(\mathbb{E}~\Bigl|\mu^{\kalman}_J[\phi] -\mu_J[\phi]\Bigr|^2\right)^{1/2} = \Bigl|\mu^{\kalman}_J[\phi] -\mu_J[\phi]\Bigr| . \]
    Since $\phi$ satisfies~\cref{assumption:metric_vector_fields_particular}, it follows that
\[\Bigl|\mu^{\kalman}_J[\phi] -\mu_J[\phi]\Bigr| \leq  \sup_{|\phi|\leq R_\phi(1+|u|^2)}\Bigl|\mu^{\kalman}_J[\phi] -\mu_J[\phi]\Bigr|= R_\phi \cdot d_g\bigl(\mu^{\kalman}_J, \mu_J\bigr).\]
We note that \cref{assumption:ensemble_kalman} holds as $\vect h$ is assumed to be affine; %hence it is possible to apply the analysis of~\cref{theorem:main_theorem} to deduce that \ams{Isn't it possible to skip mention of~\cref{theorem:main_theorem}, and hence the next line, as what you actually do is apply~\cref{thm:error_estimate_mfenkf} which of course uses~\cref{theorem:main_theorem}.}
%    \[
%    d_{g}(\mu^{\kalman}_J, \mu_J) \leq C   \max_{j \in \range{0}{J-1}} d_{g}( \op Q \op P \mu_j, \op G \op Q \op P \mu_j)  ,
%    \]
%    where $C$ is a constant depending on $\mathcal{M}_{\max\{3+\dimu, 4+\dimy\}}(\mu_0),J,\kappa_y, \kappa_{\vect \Psi}, \kappa_{\vect h}, \ell_{\vect h}, \mat \Sigma, \mat \Gamma$.
since we additionally assume that $\vect \Psi \in B_{L^\infty}\bigl(\Psi_0,\varepsilon \bigr)$, where $\Psi_0$ is an affine function, we may apply the result of~\cref{thm:error_estimate_mfenkf} to deduce that
    \begin{equation}
    \label{eq:enkf_error_triangle1}
       d_g\bigl(\mu^{\kalman}_J, \mu_J\bigr) \leq C{\varepsilon},
    \end{equation}
    where $C$ is a constant depending on $\mathcal{M}_{2\cdot\max\{3+\dimu, 4+\dimy\}}(\mu_0),J,\kappa_y, \kappa_{\vect \Psi}, \kappa_{\vect h}, \ell_{\vect h}, \mat \Sigma, \mat \Gamma$.

    The fact that $\vect h$ is assumed to be linear and the additional assumption that $|\Psi|_{C^{0,1}} \leq \ell_{\vect \Psi} < \infty$ allows us to apply the result of~\cref{lemma:convergence_to_MF}, so that
    \begin{equation}
    \label{eq:enkf_error_triangle2}
       \left(\mathbb{E}~\Bigl|\mu^{\kalmanN}_J[\phi] -\mu^{\kalman}_J[\phi]\Bigr|^2\right)^{1/2} \leq \frac{C}{\sqrt{N}},
    \end{equation}
    where $C$ is a constant depending on $\mathcal{M}_{4^{J+1}}(\mu_0),J,R_\phi,L_\phi,\kappa_y, \kappa_{\vect \Psi}, \kappa_{\vect h}, \mat \Sigma, \mat \Gamma$.
    In view of~\eqref{eq:enkf_error_triangle}, combining~\eqref{eq:enkf_error_triangle1} and \eqref{eq:enkf_error_triangle2} yields the desired result.
\end{proof}

\section{Discussion and Future Directions}
\label{sec:discussion}

In this paper we have presented the first analysis, in the setting of a filtering problem defined by nonlinear state space dynamics, quantifying the error between the empirical measure obtained by the finite particle ensemble Kalman filter and the true filtering distribution. The analysis for the EnKF outlined is based on the proof methodology developed for particle filters in \cite{rebeschini2015can} and developed for the mean field ensemble Kalman filter in \cite{carrillo2022ensemble}.
Our work goes beyond \cite{carrillo2022ensemble} by considering finite particle
filters and by allowing for dynamical models and observation operators that are unbounded. The work builds substantially on the new avenue for analysis of the ensemble Kalman methodology that was introduced in  \cite{carrillo2022ensemble} but
leaves open numerous avenues of investigation for further work; we detail these.

\begin{enumerate}
    \item As surveyed in \cite{calvello2025ensemble}, the ensemble Kalman filtering methodology may be used for solving inverse problems and for sampling; it is of interest to extend the analysis in our paper to the ensemble Kalman based inversion algorithms outlined in that paper.

    \item There is a substantial body of literature studying the continuous time limits of ensemble Kalman methods; see \cite{calvello2025ensemble} for a review. Performing analysis
    analogous to that presented here, but in the continuous time setting, would be of interest.

    \item For large scale applications, there has been recent wide interest in replacing the dynamical model $\vect \Psi$, representing the solution operator obtained via a high fidelity numerical solver, with a cheap to evaluate surrogate. Multifidelity ensemble methods
    \cite{bach2023multi} allow the use of a small number of particles evolved according to the high fidelity solver and a large number evolved according to the surrogate. Extending our error analysis to incorporate the effect of model error would be of interest in this context.
\end{enumerate}

\vspace{0.1in}
\paragraph{Acknowledgments}
The research of PM is supported by the projects SWIDIMS (ANR-20-CE40-0022) and CONVIVIALITY (ANR-23-CE40-0003) of the French National Research Agency.
The work of AMS is supported by a Department of Defense Vannevar Bush Faculty Fellowship,
and by the SciAI Center, funded by the Office of Naval Research (ONR), Grant Number N00014-23-1-2729. The work of EC was supported by the Resnick Sustainability Institute, and also by the SciAI Center and by the Vannevar Bush Faculty Fellowship held by AMS.
UV is partially supported by the European Research Council (ERC) under the EU Horizon 2020 programme (grant agreement No 810367),
and by the Agence Nationale de la Recherche under grants ANR-21-CE40-0006 (SINEQ) and ANR-23-CE40-0027 (IPSO). The work was initiated during The Jacques-Louis Lions Lectures given by AMS at  Sorbonne Universit\'e in December 2023; the authors are grateful to Albert Cohen for hosting this event.

\printbibliography

\appendix

%%%%%%%%%%%%%%%%%%%%%%%%%%%%%%%%%
%%%%%%%%%%%%%%%%%%%%%%%%%%%%%%%%%

\section{Auxiliary Results}
\label{app:A2}

We first state and prove a standard result that will be used throughout the paper.
\begin{lemma}
    \label{lemma:auxiliary_ineq_second_moment}
    Let $X$ be a random vector in $\real^{\dimu}$ with finite second moment.
    Then it holds that
    \begin{equation}
        \label{eq:second_moment_vector}
        \expect \Bigl[ \bigl(X - \expect [X]\bigr)\bigl(X - \expect [X]\bigr)^\t\Bigr] = \expect \bigl[XX^\t\bigr] - \expect [X] \expect [X]^\t
    \end{equation}
    and
    \begin{equation}
        \label{eq:inequality_second_moment}
        \forall \vect a \in \real^{\dimu},
        \qquad
        \expect \Bigl[ \bigl(X - \vect a\bigr)\bigl(X - \vect a\bigr)^\t\Bigr]
        \succcurlyeq \expect \Bigl[ \bigl(X - \expect [X] \bigr)\bigl(X - \expect [X]\bigr)^\t\Bigr].
    \end{equation}
\end{lemma}
\begin{proof}
    The statements follow from the following equalities:
    \begin{align*}
        \expect \Bigl[ (X - \vect a)(X - \vect a)^\t\Bigr]
        &= \expect \Bigl[ \Bigl(\bigl(X - \expect [X]\bigr) + \bigl(\expect [X] - \vect a\bigr)\Bigr) \Bigl(\bigl(X - \expect [X]\bigr) + \bigl(\expect [X] - \vect a\bigr)\Bigr)^\t\Bigr] \\
        &= \expect \Bigl[ \bigl(X - \expect [X]\bigr) \bigl(X - \expect [X]\bigr)^\t\Bigr] + \bigl(\expect [X] - \vect a\bigr)\bigl(\expect [X] - \vect a\bigr)^\t. \qedhere
    \end{align*}
\end{proof}

\begin{lemma}
    \label{lemma:gaussians_relation_1_and_inf_norms}
    Let $g(\placeholder; \vect m, \vect S)$ be the Lebesgue density of $\normal(\vect m, \mat S)$ and $\mathcal S^{n}_{\alpha}$ the set of symmetric $n \times n$ matrices~$\mat M$ satisfying
    \begin{equation*}
        % \label{eq:bnd}
        \frac{1}{\tau} \mat I_{n}  \preccurlyeq \mat M \preccurlyeq \tau \mat I_{n} .
    \end{equation*}
    Then for any $n \in \nat^+$ and $\tau \geq 1$, there is $L_{n,\tau} > 0$ such that
    for all parameters $(c_1, \vect m_1, \mat S_1) \in \real \times \real^{n} \times \mathcal S^{n}_{\tau}$ and $(c_2, \vect m_2, \mat S_2) \in \real \times \real^{n} \times \mathcal S^{n}_{\tau}$ it holds that
    \begin{equation*}
        % \label{eq:gaussian_auxiliary_lemma}
        \norm{h}_{\infty} \leq L_{n,\tau} \norm{h}_{1},
        \qquad h(y) = c_1 g(y; \vect m_1, \mat S_1) - c_2 g(y; \vect m_2, \mat S_2).
    \end{equation*}
\end{lemma}

\iffalse
\begin{remark}
    When $c_2=0$ then equation~\eqref{eq:gaussian_auxiliary_lemma} may be viewed as an inverse inequality.
    It then simply states that the $L^\infty$ norm of the density of a normal random variable is bounded from above by the $L^1$ norm,
    uniformly for all densities from the set of Gaussians with a covariance matrix satisfying \eqref{eq:bnd}.
\end{remark}
\fi
\begin{proof}
    The lemma as stated may be found in \cite[Lemma A.4]{carrillo2022ensemble}, where a complete proof is given.
\end{proof}

\begin{lemma}
    \label{lemma:lipschitz_density_new}
    Let~$\op P$ and $\op Q$ denote the operators on probability measures given respectively in~\eqref{eq:Markov_kernel} and~\eqref{eq:definition_Q}.
    Let~\cref{assumpenum:assumption2,assumpenum:assumption3,assumpenum:assumption_lipschitz,assumpenum:assumption5} be satisfied and suppose that $\mu \in \mathcal P(\real^{\dimu})$
    satisfies, for some $q > 0$, the moment bound
    \begin{equation}
        \label{eq:moment_assumption}
        \mathcal{M}_q(\mu)= \int_{\real^{\dimu}} \abs{x}^q \, \mu(\d x) < \infty.
    \end{equation}
    Then there is~$L = L(\mathcal{M}_q(\mu), \kappa_{\Psi}, \kappa_h, \ell_h, \mat \Sigma, \mat \Gamma) > 0$ such that
    for all $(u_1, u_2, y) \in \real^{\dimu} \times \real^{\dimu} \times \real^{\dimy}$,
    the probability density of $p = \op Q \op P \mu$ satisfies
    \begin{align}
        \label{eq:lemma_new_main_statement}
        \bigl\lvert p(u_1, y) - p(u_2, y) \bigr\rvert
        &\leq L \abs{u_1 - u_2} \min \left\{ \max \left\{ \frac{1}{1 + \abs{u_1}^q}, \frac{1}{1 + \abs{u_2}^q} \right\} , \frac{1}{1 + \abs{y}^q} \right\}.
    \end{align}
\end{lemma}
\begin{proof}
    Throughout this proof, $C$ denotes a constant whose value is irrelevant in the context,
    depends only on~$\mathcal{M}_q(\mu), \kappa_{\Psi}, \kappa_h, \ell_h, \mat \Sigma, \mat \Gamma$,
    and may change from line to line.
    Sometimes we write the dependence explicitly to indicate which parameters are involved.

    \paragraph{Step 1. Bounding the density of $\op P \mu$}
    We first rewrite
    \[
        \bigl(1 + |u|^q\bigr) \op P \mu(u)
        = C(\Sigma) \int_{\real^{\dimu}} \exp \left( - \frac{1}{2} \lvert u - \Psi(v) \rvert_{\Sigma}^2 \right) \frac{1 + \abs{u}^q}{1 + \abs{v}^q} \, \bigl(1 + \abs{v}^q\bigr) \mu(\d v).
    \]
    We note that
    \[
        S(\Sigma, \kappa_{\Psi}) :=
        \sup_{(u, v) \in \real^{\dimu} \times \real^{\dimu}}
        \exp \left( - \frac{1}{2} \vecnorm*{u - \Psi(v)}_{\mat \Sigma}^2 \right) \, \frac{1 + |u|^q}{1 + |v|^q}
        < \infty;
    \]
    this may be seen observing that
    \begin{equation}
    \label{eq:intermediate_step}
        \exp \left( - \frac{1}{2} \vecnorm*{u - \Psi(v)}_{\mat \Sigma}^2 \right) \, \frac{1 + |u|^q}{1 + |v|^q} \leq \exp \left( - \frac{1}{2} \vecnorm*{u - \Psi(v)}_{\mat \Sigma}^2 \right) \, \frac{(1 + 2^{q-1}|u - \Psi(v)|^q) + 2^{q-1}|\Psi(v)|^q}{1 + |v|^q}.
    \end{equation}
    The first term on the righthand side of \eqref{eq:intermediate_step} may be bounded by noting that $1+2^{q-1}e^{-x^2}x^q$ is uniformly bounded in $x\in \real$; on the other hand the second term may be bounded by applying~\cref{assumpenum:assumption2}. Now, using~\eqref{eq:moment_assumption},
    we obtain that
    \begin{equation}
        \label{eq:bound_markov_kernel_density}
        \op P \mu(u) \leq \frac{C\bigl(\Sigma, \kappa_{\Psi}, \mathcal{M}_q(\mu)\bigr)}{1 + |u|^q}.
    \end{equation}

    \paragraph{Step 2. Establishing Lipschitz continuity of $u \mapsto \op P \mu(u)$}
    Since $g(x) := \e^{-x^2}$ has derivative~$2 x \e^{-x^2}$ and $\lvert x \e^{-x^2} \rvert \leq \e^{-\frac{x^2}{2}}$ for all $x \in \real$,
    it holds for some $\xi$ between $|a|$ and $|b|$ that
    \begin{align}
        \label{eq:elementary_inequality_gaussian_density}
        \forall  (a, b) \in \real \times \real, \qquad
        \left\lvert \e^{-a^2} - \e^{-b^2} \right\rvert
        =  \abs{b - a} \, \lvert g'(\xi) \rvert
        \leq 2 \abs{b - a} \left( \e^{-\frac{a^2}{2}} + \e^{-\frac{b^2}{2}} \right).
    \end{align}
    Using this inequality with $a^2 = \frac{1}{2} \vecnorm*{u_1 - \vect \Psi(v)}_{\mat \Sigma}^2$ and $b^2 = \frac{1}{2} \vecnorm*{u_2 - \vect \Psi(v)}_{\mat \Sigma}^2$,
    the triangle inequality,
    and equivalence of norms,
    we deduce that, for all $(u_1, u_2, v) \in \real^{\dimu} \times \real^{\dimu} \times \real^{\dimu}$,
    \begin{align*}
        &\left\lvert \exp \left( - \frac{1}{2} \vecnorm*{u_1 - \vect \Psi(v)}_{\mat \Sigma}^2 \right) -\exp \left( - \frac{1}{2} \vecnorm*{u_2 - \vect \Psi(v)}_{\mat \Sigma}^2 \right) \right\rvert \\
        \label{eq:initial_inequality}
        & \hspace{2cm} \leq C \vecnorm{u_2 - u_1} \left( \exp \left( - \frac{1}{4} \vecnorm*{u_1 - \vect \Psi(v)}_{\mat \Sigma}^2 \right) + \exp \left( - \frac{1}{4} \vecnorm*{u_2 - \vect \Psi(v)}_{\mat \Sigma}^2 \right) \right)
    \end{align*}
    for constant $C=C(\Sigma).$
    Integrating out the $v$ variable with respect to $\mu$
    we obtain that
    \begin{align*}
        \lvert \op P \mu(u_1) - \op P \mu(u_2) \rvert
        &\leq
           C \abs{u_1 - u_2} \int_{\real^{\dimu}} \exp \left( - \frac{1}{4} \vecnorm*{u_1 - \vect \Psi(v)}_{\mat \Sigma}^2 \right) \, \mu(\d v) \\
        & \quad + C \abs{u_1 - u_2} \int_{\real^{\dimu}} \exp \left( - \frac{1}{4} \vecnorm*{u_2 - \vect \Psi(v)}_{\mat \Sigma}^2 \right) \, \mu(\d v).
    \end{align*}
    The integrals on the right-hand side can be bounded as in the first step,
    which leads to the inequality
    \begin{equation}
        \label{eq:technical_lipschitz_main1}
        \forall (u_1, u_2) \in \real^{\dimu} \times \real^{\dimu}, \qquad
        \lvert \op P \mu(u_1) - \op P \mu(u_2) \rvert
        \leq C \lvert u_1 - u_2 \rvert \max \left\{ \frac{1}{1 + \abs{u_1}^q}, \frac{1}{1 + \abs{u_2}^q} \right\}.
    \end{equation}

    \paragraph{Step 3. Obtaining a coarse estimate}
    In view of the elementary inequality~\eqref{eq:elementary_inequality_gaussian_density} and
    the assumed Lipschitz continuity of~$h$,
    it holds that, for constant $C=C(\Gamma,\ell_h)$
    \begin{align}
        \notag
        &\Bigl\lvert \normal\bigl(\vect h(u_1), \mat \Gamma\bigr)(y) - \normal\bigl(\vect h(u_2), \mat \Gamma\bigr)(y) \Bigr\rvert \\
        \label{eq:technical_lipschitz_main2}
        \qquad & \leq C \vecnorm{u_1 - u_2}  \exp \left( - \frac{1}{4}  \vecnorm{y - h(u_1)}_{\mat \Gamma}^2 \right)
        + C \vecnorm{u_1 - u_2} \exp \left( - \frac{1}{4}  \vecnorm{y - h(u_2)}_{\mat \Gamma}^2 \right).
    \end{align}
    Using the decomposition
    \begin{align*}
        p(u_1, y) - p(u_2, y)
        &= \op P \mu(u_1) \, \normal\bigl(\vect h(u_1), \mat \Gamma\bigr) (y) - \op P \mu(u_2) \, \normal\bigl(\vect h(u_2), \mat \Gamma\bigr) (y) \\
        &= \bigl(\op P \mu(u_1) - \op P \mu(u_2)\bigr) \, \normal\bigl(\vect h(u_1), \mat \Gamma\bigr) (y) \\
        &\quad + \op P \mu(u_2) \Bigl( \normal\bigl(\vect h(u_1), \mat \Gamma\bigr) (y) - \normal\bigl(\vect h(u_2), \mat \Gamma\bigr) (y) \Bigr).
    \end{align*}
    and employing~\eqref{eq:bound_markov_kernel_density},~\eqref{eq:technical_lipschitz_main1} and~\eqref{eq:technical_lipschitz_main2},
    we deduce that
    \begin{align}
        \notag
        \lvert p(u_1, y) - p(u_2, y) \rvert
        &\leq C \abs{u_1 - u_2} \max \left\{ \frac{1}{1 + \abs{u_1}^q}, \frac{1}{1 + \abs{u_2}^q} \right\} \\
        \label{eq:lipschitz_density_new}
        &\qquad \times \max \left\{ \exp \left( - \frac{1}{4}  \vecnorm{y - h(u_1)}_{\mat \Gamma}^2 \right), \exp \left( - \frac{1}{4}  \vecnorm{y - h(u_2)}_{\mat \Gamma}^2 \right) \right\}.
    \end{align}
    Note that, for the function $h(u) = u$, the quantity multiplying $\abs{u_1 - u_2}$ on the right-hand does not tend to 0 along the sequence $\left(u_1^{(n)}, u_2^{(n)}, y^{(n)}\right) = (0, n, n)$.
    For our purposes in this work,
    we need the finer estimate~\eqref{eq:lemma_new_main_statement};
    establishing this bound is the aim of the next two steps.

    \paragraph{Step 4. Bounding the density $p(u, y)$}
    Recall that $p(u, y) = \op P \mu(u) \mathcal N\bigl(h(u), \mat \Gamma\bigr)(y)$.
    We prove in this step the inequality
    \begin{equation}
        \label{eq:density_uy}
        p(u, y) \leq C \min \left\{ \frac{1}{1 + \abs{u}^q}, \frac{1}{1 + \abs{y}^q}\right\},
    \end{equation}
    or equivalently
    \[
        \sup_{(u,y) \in \real^{\dimu} \times \real^{\dimy}}  p(u, y) \, \max \left\{ 1 + \abs{u}^q, 1 +\abs{y}^q \right\} < \infty.
    \]
    To this end, note that
    \begin{align*}
         p(u, y) \max \left\{ \abs{u}^q,  \abs{y}^q\right\}
         &= \op P \mu(u) \, \mathcal N\bigl(h(u), \mat \Gamma\bigr)(y) \max \left\{ \abs{u}^q,  \abs{y}^q\right\} \\
         &\leq \op P \mu(u) \, \mathcal N\bigl(h(u), \mat \Gamma\bigr)(y) \abs{u}^q
         + \op P \mu(u) \, \mathcal N\bigl(h(u), \mat \Gamma\bigr)(y) \abs{y}^q.
    \end{align*}
    The first term is bounded uniformly by {\bf Step 1}, estimate \eqref{eq:bound_markov_kernel_density}.
    For the second term,
    we use that
    \begin{equation}
        \label{eq:decomposition_yq}
        \lvert y \rvert^q \leq 2^{q-1} \bigl\lvert h(u) \bigr\rvert^{q} + 2^{q-1} \bigl\lvert y - h(u) \bigr\rvert^q,
    \end{equation}
    to obtain
    \[
        \op P \mu(u) \, \mathcal N\bigl(h(u), \mat \Gamma\bigr)(y) \abs{y}^q
        \leq C \op P \mu(u) \, \mathcal N\bigl(h(u), \mat \Gamma\bigr)(y) \Bigl( \abs{h(u)}^q + \abs{y - h(u)}^q \Bigr).
    \]
    The first term on the right-hand side is bounded uniformly,
    again by {\bf Step 1}, estimate \eqref{eq:bound_markov_kernel_density}, and using~\cref{assumpenum:assumption3}.
    The second term is also bounded uniformly because the function $x \mapsto \mathcal N(0, \mat \Gamma)(x) \abs{x}^q$ is uniformly bounded in $x$,
    by a value depending only on $\Gamma$ and $q$.

    \paragraph{Step 5. Obtaining the estimate~\eqref{eq:lemma_new_main_statement}}
    The claimed inequality is equivalent to
    \[
        \sup_{u_1, u_2, y}  \frac{ \bigl\lvert p(u_1, y) - p(u_2, y) \bigr\rvert }{\abs{u_1 - u_2}}
        \max \Bigl\{ \min \left\{ 1 + \abs{u_1}^q, 1 + \abs{u_2}^q \right\} , 1 + \abs{y}^q \Bigr\}
        < \infty,
    \]
    where the supremum is over $\real^{\dimu} \times \real^{\dimu} \times \real^{\dimy}$.
    By~\eqref{eq:lipschitz_density_new}, it holds that
    \[
        \sup_{u_1, u_2, y}  \frac{ \bigl\lvert p(u_1, y) - p(u_2, y) \bigr\rvert }{\abs{u_1 - u_2}}
        \min \left\{ 1 + \abs{u_1}^q, 1 + \abs{u_2}^q \right\} < \infty,
    \]
    so it remains to show that
    \[
        \sup_{u_1, u_2, y}  \frac{ \bigl\lvert p(u_1, y) - p(u_2, y) \bigr\rvert }{\abs{u_1 - u_2}} \abs{y}^q < \infty.
    \]
    By~\eqref{eq:density_uy},
    it is clear that the supremum is uniformly bounded if restricted to the set $\abs{u_1 - u_2} \geq 1$,
    so it suffices to show that
    \[
        \sup_{(u, \delta, y) \in \real^{\dimu} \times B(0, 1) \times \real^{\dimy}} \frac{ \bigl\lvert p(u, y) - p(u+\delta, y) \bigr\rvert }{\abs{\delta}}
        \abs{y}^q < \infty,
    \]
    where $B(0, 1)$ is the open ball of radius 1 centered radius at the origin in $\real^{\dimu}$.
    We use again~\eqref{eq:decomposition_yq} in order to bound
    \begin{align*}
        \frac{ \bigl\lvert p(u, y) - p(u+\delta, y) \bigr\rvert }{\abs{\delta}} \abs{y}^q
        &\leq C\frac{ \bigl\lvert p(u, y) - p(u+\delta, y) \bigr\rvert }{\abs{\delta}} \bigl\lvert h(u) \bigr\rvert^q \\
        &\qquad + C\frac{ \bigl\lvert p(u, y) - p(u+\delta, y) \bigr\rvert}{\abs{\delta}} \abs{y - h(u)}^q.
    \end{align*}
    The first term is bounded uniformly by~\eqref{eq:lipschitz_density_new} and the assumption that~$\bigl\lvert h(u) \bigr\rvert \leq \kappa_h \bigl(1 + \abs{u}\bigr)$.
    To conclude the proof,
    it remains to show that
    \begin{equation}
        \label{eq:final_equation_new_lemma}
        \sup_{(u, \delta, y) \in \real^{\dimu} \times B(0, 1) \times \real^{\dimy}}
        C\frac{ \bigl\lvert p(u, y) - p(u+\delta, y) \bigr\rvert }{\abs{\delta}} \abs{y - h(u)}^q < \infty.
    \end{equation}
    By~\eqref{eq:lipschitz_density_new} %the inequality $\abs{\delta} \leq 1$,
    %and the Lipschitz continuity of~$h$,
    it holds that
    \begin{align*}
        &\bigl\lvert p(u, y) - p(u+\delta, y) \bigr\rvert  \\
        &\qquad \leq  \frac{C \abs{\delta}}{1 + \abs{u}^q}
        \max \left\{ \exp \left( - \frac{1}{4}  \vecnorm{y - h(u)}_{\mat \Gamma}^2 \right), \exp \left( - \frac{1}{4}  \vecnorm{y - h(u + \delta)}_{\mat \Gamma}^2 \right) \right\}. \\
        % &\qquad \leq  \frac{C \abs{\delta}}{1 + \abs{u}^q} \exp \left( - \frac{1}{4}  \vecnorm{y - h(u)}_{\mat \Gamma}^2 \right)
        % \max \left\{ 1,  \exp \left( - \frac{1}{4}  \ip[\Big]{h(u) - h(u + \delta),2y -h(u) - h(u+\delta)}_{\Gamma} \right) \right\}.
        &\qquad \leq  \frac{C \abs{\delta}}{1 + \abs{u}^q} \exp \left( - \frac{1}{8}  \vecnorm{y - h(u)}_{\mat \Gamma}^2
        + \frac{1}{4} \abs[\Big]{h(u) - h(u + \delta)}_{\mat \Gamma}^2 \right).
        % &\qquad \leq  \frac{C \abs{\delta}}{1 + \abs{u}^q} \exp \left( - \frac{1}{4}  \vecnorm{y - h(u)}_{\mat \Gamma}^2
        % + C \abs[\Big]{y -h(u)}_{\mat \Gamma} \right).
    \end{align*}
    In the last line we used the inequality $|a + b|^2 \geq \frac{1}{2}|a|^2 - |b|^2$,
    which follows from Young's inequality.
    Since the function $x \mapsto \e^{-\frac{x^2}{8} + C} x^q$ is bounded uniformly in $x \in \real$ and by~\cref{assumpenum:assumption_lipschitz},
    the bound~\eqref{eq:final_equation_new_lemma} easily follows,
    concluding the proof.
\end{proof}

\section{Technical Results for Theorem \ref{theorem:main_theorem}}
\label{appendix:A}

We establish moment bounds in~\cref{sub:moment_bounds}, we recall that on Gaussian measures the action of the conditioning map and the Kalman transport map are equivalent in~\cref{sub:mean_field_is_conditioning_for_gaussians},
and we prove stability results in~\cref{sub:stability_results}.
\subsection{Moment Bounds}
\label{sub:moment_bounds}
\begin{lemma}
    [Moment Bounds]
    \label{lemma:bound_on_Pmu_new}
    Let~$\mu$ be a probability measure on~$\real^{\dimu}$ with bounded first and second order polynomial moments $\mathcal{M}_1(\mu),\mathcal{M}_2(\mu)<\infty$. Under~\cref{assumpenum:assumption2,assumpenum:assumption5},
    it holds that
    \begin{equation}
        \label{eq:bound_on_Pmu_new}
        \bigl\lvert \mathcal M(\op P \mu) \bigr\rvert \leq \kappa_{\vect \Psi} \bigl(1+\mathcal{M}_1(\mu)\bigr), \qquad
        \mat \Sigma \preccurlyeq \mathcal C(\op P \mu) \preccurlyeq \mat \Sigma +2\kappa_{\vect \Psi}^2\bigl(1+\mathcal{M}_2(\mu) \bigr) \mat I_{\dimu}.
    \end{equation}
\end{lemma}
\begin{proof}
The proof of this lemma follows the steps of the proof of \cite[Lemma B.1]{carrillo2022ensemble}; however, here different bounds reflecting the assumptions on $\vect \Psi$ and $\vect h$ in this paper are required. Using the definition of~$\op P$ in~\eqref{eq:Markov_kernel},
it holds that
\begin{align*}
    \mathcal M(\op P \mu)
    = \int_{\real^{\dimu}} u \, \op P \mu(u) \, \d u
    &= \frac{1}{\sqrt{(2\pi)^{\dimu} \det \mat \Sigma}}
    \int_{\real^{\dimu}} \int_{\real^{\dimu}} u \exp \left( - \frac{1}{2} \vecnorm[\big]{u - \vect \Psi(v)}_{\mat \Sigma}^2 \right) \, \mu(\d v) \, \d u \\
    &= \int_{\real^{\dimu}} \vect \Psi(v) \, \mu(\d v),
\end{align*}
where application of Fubini's theorem yields the last equality. The first inequality in~\eqref{eq:bound_on_Pmu_new} then follows from~\cref{assumpenum:assumption2}.
Obtaining the lower bound of the second inequality in~\eqref{eq:bound_on_Pmu_new} may be done identically to the proof of \cite[Lemma B.1]{carrillo2022ensemble}.
We turn our attention to obtaining the upper bound. Using~\cref{lemma:auxiliary_ineq_second_moment} and noting that $\vect w \vect w^\t \preccurlyeq (\vect w^\t \vect w) \mat I_{\dimu}$ for any vector $\vect w \in \real^{\dimu}$ by the Cauchy-Schwarz inequality,
we deduce that
\begin{align}
    \nonumber
    \mathcal C(\op P \mu)
    &\preccurlyeq \int_{\real^{\dimu}} u \otimes u \,  \op P \mu(u) \, \d u  \\
    \nonumber
    &= \frac{1}{\sqrt{(2\pi)^{\dimu} \det \mat \Sigma}} \int_{\real^{\dimu}} \int_{\real^{\dimu}} u \otimes u \exp \left( - \frac{1}{2} \vecnorm*{u - \vect \Psi(v)}_{\mat \Sigma}^2 \right) \mu(\d v) \, \d u \\
    \nonumber
    &= \int_{\real^{\dimu}} \bigl( \vect \Psi(v) \otimes \vect \Psi(v) + \mat \Sigma \bigr) \, \mu(\d v)\\
    \label{eq:covariance_Pmu_new}
    & \preccurlyeq \mat \Sigma + \left(\int_{\real^{\dimu}}|\vect \Psi(v)|^2\mu(\d v) \right)\mat I_{\dimu}
    \preccurlyeq \mat \Sigma + 2\kappa_{\vect \Psi}^2\bigl(1+\mathcal{M}_2(\mu) \bigr) \mat I_{\dimu} ,
\end{align}
which yields the desired result.
\end{proof}

\begin{lemma}
    \label{lemma:bound_on_QPmu_new}
    Let~$\mu$ be a probability measure on~$\real^{\dimu}$ with bounded first and second order polynomial moments $\mathcal{M}_1(\mu),\mathcal{M}_2(\mu)<\infty$. Under~\cref{assumpenum:assumption2,assumpenum:assumption3,assumpenum:assumption5},
    it holds that
    \begin{subequations}
            \begin{align}
                \label{eq:bound_on_QPmu_mean_new_1}
                \bigl\lvert \mathcal M^u(\op Q \op P \mu) \bigr\rvert
        &\leq \kappa_{\vect \Psi} \bigl(1+\mathcal{M}_1(\mu)\bigr), \\
        \label{eq:bound_on_QPmu_mean_new_2}
        \bigl\lvert \mathcal M^y(\op Q \op P \mu) \bigr\rvert
        &\leq \kappa_{\vect h} \sqrt{2 \Bigl( 1+\trace(\mat \Sigma)+2\kappa_{\vect \Psi}^2\bigl(1+\mathcal{M}_2(\mu)\bigr) \Bigr)}.
            \end{align}
    \end{subequations}
    Furthermore, it holds that
    \begin{subequations}
        \begin{align}
            \label{eq:bound_on_QPmu_covariance_new_upper}
            \mathcal C(\op Q \op P \mu)&\preccurlyeq
            \begin{pmatrix}
                4\kappa_{\vect \Psi}^2\bigl(1+\mathcal{M}_2(\mu) \bigr)  \mat I_{\dimu}  + 2 \mat \Sigma & \mat 0_{\dimu \times \dimy} \\
                \mat 0_{\dimy \times \dimu} & 4\kappa_{\vect h}^2\bigl(1+\trace(\mat \Sigma)+2\kappa_{\vect \Psi}^2\bigl(1+\mathcal{M}_2(\mu)\bigr) \bigr) \mat I_{\dimy}  + \mat \Gamma
            \end{pmatrix},\\
            \label{eq:bound_on_QPmu_covariance_new_lower}
            \mathcal C(\op Q \op P \mu)&
            \succcurlyeq
            \frac{\gamma\cdot \min \Bigl\{2\sigma , \gamma + 4\kappa_{\vect h}^2\Bigl(1+\trace(\mat \Sigma)+2\kappa_{\vect \Psi}^2\bigl(1+\mathcal{M}_2(\mu)\bigr) \Bigr)\Bigr\}}{2\gamma + 8\kappa_{\vect h}^2\Bigl(1+\trace(\mat \Sigma)+2\kappa_{\vect \Psi}^2\bigl(1+\mathcal{M}_2(\mu)\bigr) \Bigr) }
            \mat I_{\dimu + \dimy}.
        \end{align}
    \end{subequations}
\end{lemma}
\begin{proof}
    The inequalities~\eqref{eq:bound_on_QPmu_mean_new_1} and~\eqref{eq:bound_on_QPmu_mean_new_2} follow from the assumptions
and the fact that
\begin{align}
\label{eq:mean_QPmu}
    \mathcal M(\op Q \op P \mu) =
    \begin{pmatrix}
        \mathcal M(\op P \mu) \\
        \op P \mu [\vect h]
    \end{pmatrix}.
\end{align}
Indeed, from~\cref{lemma:bound_on_Pmu_new} we know that
\(
    \bigl\lvert \mathcal M(\op P \mu) \bigr\rvert
    \leq \kappa_{\vect \Psi} \bigl(1+\mathcal{M}_1(\mu)\bigr),
\)
which leads by Jensen's inequality to~\eqref{eq:bound_on_QPmu_mean_new_1}.
To deduce~\eqref{eq:bound_on_QPmu_mean_new_2},
we note that
\begin{subequations}
    \begin{align}
        \nonumber
        \bigl|P \mu [\vect h]\bigr|^2
        &\leq \int_{\real^{\dimu}}|\vect h(u)|^2 \, \op P \mu (\d u)\\
        \label{eq:Pmuh_intemediate}
        &\leq 2\kappa_{\vect h}^2 + 2\kappa_{\vect h}^2\int_{\real^{\dimu}}|u|^2 \, \op P \mu (\d u)\\
        \label{eq:Pmuh_final}
        &= 2\kappa_{\vect h}^2+ 2\kappa_{\vect h}^2\int_{\real^{\dimu}}\Bigl(|\vect \Psi(v)|^2 +\trace (\mat \Sigma) \Bigr) \, \mu(\d v),
    \end{align}
\end{subequations}
where~\eqref{eq:Pmuh_intemediate} follows by applying~\cref{assumpenum:assumption3} and Young's inequality and~\eqref{eq:Pmuh_final} follows by applying the properties of the Gaussian transition density.
The result then follows by applying~\cref{assumpenum:assumption2}.

To obtain the covariance bounds, we proceed using analogous steps to the proof of \cite[Lemma B.2]{carrillo2022ensemble}; however, here different bounds reflecting the assumptions on $\vect \Psi$ and $\vect h$ in this paper are required. We begin by establishing inequality~\eqref{eq:bound_on_QPmu_covariance_new_upper}.
To this end,
letting $\phi\colon \real^{\dimu} \to \real^{\dimu + \dimy}$ define the map $\phi(u) = \bigl(u, \vect h(u)\bigr)$, it holds that
\begin{equation}
    \label{eq:decomposition_covariance_new}
    \mathcal C(\op Q \op P \mu) =
    \mathcal C(\phi_{\sharp} \op P \mu) +
    \begin{pmatrix}
        \mat 0_{\dimu \times \dimu} & \mat 0_{\dimu \times \dimy} \\
        \mat 0_{\dimy \times \dimu} & \mat \Gamma
    \end{pmatrix}
    =
    \begin{pmatrix}
        \mathcal C^{uu}(\phi_{\sharp} \op P \mu) & \mathcal C^{uy}(\phi_{\sharp} \op P \mu) \\
        \mathcal C^{yu}(\phi_{\sharp} \op P \mu) & \mathcal C^{yy}(\phi_{\sharp} \op P \mu) + \mat \Gamma
    \end{pmatrix}.
\end{equation}
As shown in the proof of \cite[Lemma B.2]{carrillo2022ensemble} we have, for all $(\vect a, \vect b) \in \real^{\dimu} \times \real^{\dimy}$,
\[
    \begin{pmatrix}
        \vect a \\
        \vect b
    \end{pmatrix}
    \otimes
    \begin{pmatrix}
        \vect a \\
        \vect b
    \end{pmatrix}
    \preccurlyeq
    2
    \begin{pmatrix}
        \vect a \vect a^\t & \mat 0_{\dimu \times \dimy} \\
        \mat 0_{\dimy \times \dimu} & \vect b \vect b^\t
    \end{pmatrix},
\]
from which we establish,
using~\cref{lemma:auxiliary_ineq_second_moment,lemma:bound_on_Pmu_new} and~\cref{assumpenum:assumption2,assumpenum:assumption3,assumpenum:assumption5},
that
\begin{align}
    \mathcal C(\phi_{\sharp} \op P \mu)
    \nonumber
    &\preccurlyeq \int_{\real^{\dimu}} \begin{pmatrix} u \\ \vect h(u) \end{pmatrix} \otimes \begin{pmatrix} u \\ \vect h(u) \end{pmatrix} \, \op P \mu(u) \, \d u  \\
    \nonumber
    &\preccurlyeq  2\int_{\real^{\dimu}}
    \begin{pmatrix}
        u u^\t & \mat 0_{\dimu \times \dimy} \\
        \mat 0_{\dimy \times \dimu} & \vect h(u) \vect h(u)^\t
    \end{pmatrix} \, \op P \mu(u) \, \d u\\
    \label{eq:bound_covariance_pushforward}
    &\preccurlyeq
    2
    \begin{pmatrix}
        \mat \Sigma + 2\kappa_{\vect \Psi}^2\bigl(1+\mathcal{M}_2(\mu) \bigr) \mat I_{\dimu} & \mat 0_{\dimu \times \dimy} \\
        \mat 0_{\dimy \times \dimu} & 2\kappa_{\vect h}^2\bigl(1+\trace(\mat \Sigma)+2\kappa_{\vect \Psi}^2\bigl(1+\mathcal{M}_2(\mu)\bigr) \bigr) \mat I_{\dimy}
    \end{pmatrix},
\end{align}
where we applied~\eqref{eq:covariance_Pmu_new} and the calculation resulting in~\eqref{eq:Pmuh_intemediate} in the last inequality.
Noting this inequality in combination with~\eqref{eq:decomposition_covariance_new} yields the upper bound~\eqref{eq:bound_on_QPmu_covariance_new_upper}.
We now show that
\begin{equation}
 \begin{pmatrix}
        \vect a \\ \vect b
    \end{pmatrix}^\t
    \mathcal C(\op Q \op P \mu)
    \begin{pmatrix}
        \vect a \\ \vect b
    \end{pmatrix}
    \geq (1 - \varepsilon) \sigma \vecnorm*{\vect a}^2  + \left(\gamma - \left(\frac{1}{\varepsilon} - 1\right) \op P \mu\bigl[|h|^2\bigr]  \right) \vecnorm{\vect b}^2.
\end{equation}
in order to establish the lower bound~\eqref{eq:bound_on_QPmu_covariance_new_lower}.
The argument is similar to that in \cite[Lemma B.2]{carrillo2022ensemble} but is
explicit in certain constants whose dependencies on moments need to be
controlled in this paper.
We first note that for any~$\pi \in \mathcal P(\real^{\dimu} \times \real^{\dimy})$
and all $(\vect a, \vect b) \in \real^{\dimu} \times \real^{\dimy}$ it holds that
\begin{align*}
    \vecnorm*{\vect a^\t \mathcal C^{uy}(\pi) \vect b}
    &= \int_{\real^{\dimu} \times \real^{\dimy}} \Bigl(\vect a^\t \bigl(u - \mathcal M^u(\pi)\bigr)\Bigr) \Bigl(\vect b^\t \bigl(y - \mathcal M^y(\pi)\bigr)\Bigr) \pi(\d u \d y) \\
    &\leq \sqrt{\vect a^\t \mathcal C^{uu}(\pi) \vect a} \, \sqrt{\vect b^\t \mathcal C^{yy}(\pi) \vect b},
\end{align*}
by the Cauchy--Schwarz inequality. Hence, we deduce that for all~$\varepsilon \in (0,1)$ and for all $(\vect a, \vect b) \in \real^{\dimu} \times \real^{\dimy}$
\begin{align*}
    \begin{pmatrix}
        \vect a \\ \vect b
    \end{pmatrix}^\t
    \mathcal C(\phi_{\sharp} \op P \mu)
    \begin{pmatrix}
        \vect a \\ \vect b
    \end{pmatrix}
    &\geq (1 - \varepsilon) \vect a^\t \mathcal C^{uu}(\phi_{\sharp} \op P \mu) \vect a - \left(\frac{1}{\varepsilon} - 1\right) \vect b^\t \mathcal C^{yy}(\phi_{\sharp} \op P \mu) \vect b \\
    &\geq (1 - \varepsilon) \vect a^\t \mat \Sigma \vect a - \left(\frac{1}{\varepsilon} - 1\right) \op P \mu\bigl[|h|^2\bigr],
\end{align*}
where we applied Young's inequality for the bound in the first line, and~\eqref{eq:bound_on_Pmu_new} and
the bound~$\mathcal C^{yy}(\phi_{\sharp} \op P \mu) \preccurlyeq \op P \mu\bigl[|h|^2\bigr]I_{d_y}$
% NOTE: we do not need coefficient 2 here.
% 4\kappa_{\vect h}^2\Bigl(1+\trace(\mat \Sigma)+2\kappa_{\vect \Psi}^2\bigl(1+\mathcal{M}_2(\mu)\bigr) \Bigr) \mat I_{\dimy} $
for the bound in the second line.
We then apply~\eqref{eq:decomposition_covariance_new} to find
\begin{align*}
    \nonumber
    \begin{pmatrix}
        \vect a \\ \vect b
    \end{pmatrix}^\t
    \mathcal C(\op Q \op P \mu)
    \begin{pmatrix}
        \vect a \\ \vect b
    \end{pmatrix}
    &\geq (1 - \varepsilon) \vect a^\t \mat \Sigma \vect a - \left(\frac{1}{\varepsilon} - 1\right) \op P \mu\bigl[|h|^2\bigr] +  \vect b^\t \mat \Gamma \vect b \\
    &\geq (1 - \varepsilon) \sigma \vecnorm*{\vect a}^2  + \left(\gamma - \left(\frac{1}{\varepsilon} - 1\right) \op P \mu\bigl[|h|^2\bigr]  \right) \vecnorm{\vect b}^2.
\end{align*}
Choosing $\varepsilon = \frac{2\op P \mu\bigl[|h|^2\bigr]}{\gamma + 2\op P \mu\bigl[|h|^2\bigr]}$,
so that the coefficient of the $\vecnorm{\vect b}^2$ term is $\frac{\gamma}{2}$,
we obtain
\begin{equation}
    \label{eq:final_step_bound_covariance_QPmu}
    \begin{pmatrix}
        \vect a \\ \vect b
    \end{pmatrix}^\t
    \mathcal C(\op Q \op P \mu)
    \begin{pmatrix}
        \vect a \\ \vect b
    \end{pmatrix}
    \geq
    \frac{\gamma \sigma}
    {\gamma + 2 \op P \mu \bigl[ |h|^2 \bigr]} \vecnorm{a}^2 + \frac{\gamma}{2} \vecnorm{b}^2
    \geq
    \min \left\{ \frac{\gamma \sigma}{\gamma + 2 \op P \mu \bigl[ |h|^2 \bigr]}, \frac{\gamma}{2} \right\} \left( \vecnorm{a}^2 + \vecnorm{b}^2 \right).
\end{equation}
Since this is true for any~$(a, b) \in \real^{\dimu} \times \real^{\dimy}$,
it follows that
\[
    \mathcal C(\op Q \op P \mu)
    \succcurlyeq
    \gamma \frac{\min \left\{ 2\sigma, \gamma + 2 \op P \mu \bigl[ |h|^2 \bigr] \right\}}{2 \gamma + 4 \op P \mu \bigl[ |h|^2 \bigr]} \mat I_{\dimu + \dimy}.
\]
In order to deduce~\eqref{eq:bound_on_QPmu_covariance_new_lower},
we use~\eqref{eq:Pmuh_final} to find that
\[
    \op P \mu \bigl[ |h|^2 \bigr]
    \leq 2\kappa_{\vect h}^2
    \Bigl(1+\trace(\mat \Sigma)+2\kappa_{\vect \Psi}^2 \bigl(1+\mathcal{M}_2(\mu)\bigr) \Bigr),
\]
from which we conclude the desired result.
\end{proof}

\begin{lemma}
    [Moment Bound for the True Filtering Distribution]
    \label{lemma:true_filter_bounded_moments}
    Let $q\geq2$ be an integer. Assume that the probability measures $(\mu_j)_{j \in \range{0}{J}}$ are obtained from the dynamical system~\eqref{eq:true_filtering2}
    initialized at the probability measure $\mu_0\in~\mathcal P(\real^{\dimu})$ with bounded $q$th order polynomial moment $\mathcal{M}_q(\mu_0)<\infty$.
    If \cref{assumpenum:assumption1,assumpenum:assumption2,assumpenum:assumption3,assumpenum:assumption5} hold
    then there exists~$C = C\bigl(\mathcal{M}_q(\mu_0), J,\kappa_y, \kappa_{\vect \Psi}, \kappa_{\vect h}, \mat \Sigma, \mat \Gamma\bigr)$
    such that
    \[
        \max_{j \in \range{0}{J}} \mathcal{M}_q(\mu_j)\leq C.
    \]
\end{lemma}

\begin{proof}
We have $\mu_{j+1} = \op B_j \op Q \op P \mu_j$, for $j \in \range{0}{J-1}$.
Equivalently, using the notation from~\eqref{eq:true_filtering}, it holds that
$\mu_{j+1} = \op L_j \op P \mu_j$
for each $j\in\range{0}{J-1}$. Hence, we note that
\begin{equation}
\mu_{j+1} = \frac
    {\displaystyle \exp \left( - \frac{1}{2} \bigl\lvert y_{j+1}^{\dagger} - \vect h(u) \bigr\rvert_{\mat \Gamma}^2 \right)  \op P \mu_j(u)}
    {\displaystyle \int_{\real^{\dimu}} \exp \left( - \frac{1}{2} \bigl\lvert y_{j+1}^{\dagger} - \vect h(U) \bigr\rvert_{\mat \Gamma}^2 \right) \op P \mu_j(U) \, \d U}.
\end{equation}
It is readily observed that
\begin{align}
    \nonumber
    \mathcal{M}_q(\mu_{j+1}) &= \frac
    {\displaystyle \int_{\real^{\dimu}}|u|^q \exp \left( - \frac{1}{2} \bigl\lvert y_{j+1}^{\dagger} - \vect h(u) \bigr\rvert_{\mat \Gamma}^2 \right)  \op P \mu_j(u) \d u}
    {\displaystyle \int_{\real^{\dimu}} \exp \left( - \frac{1}{2} \bigl\lvert y_{j+1}^{\dagger} - \vect h(U) \bigr\rvert_{\mat \Gamma}^2 \right) \op P \mu_j(U) \, \d U}\\
    \label{eq:q_momentPmu}
    &\leq \frac
    {\displaystyle \int_{\real^{\dimu}}|u|^q \, \op P \mu_j(u) \d u}
    {\displaystyle \int_{\real^{\dimu}} \exp \left( - \frac{1}{2} \bigl\lvert y_{j+1}^{\dagger} - \vect h(U) \bigr\rvert_{\mat \Gamma}^2 \right) \, \op P \mu_j(U) \, \d U}.
\end{align}
We first bound from above the numerator of~\eqref{eq:q_momentPmu}; indeed, note that
\begin{subequations}
\begin{align}
\nonumber
    \int_{\real^{\dimu}}|u|^q \op P \mu_j(u) \d u &= \int_{\real^{\dimu}}|u|^q \left(\int_{\real^{\dimu}}\exp \left( - \frac{1}{2} \vecnorm*{u - \vect \Psi(v)}_{\mat \Sigma}^2 \right) \, \mu_j(\d v) \right)\d u\\
    \label{eq:fubini_Pmu}
    &= \int_{\real^{\dimu}} \left(\int_{\real^{\dimu}}|u|^q\exp \left( - \frac{1}{2} \vecnorm*{u - \vect \Psi(v)}_{\mat \Sigma}^2 \right) \, \d u \right)\mu_j(\d v)\\
    \label{eq:qmoment_Pmu2}
    &\leq C \int_{\real^{\dimu}} \bigl( 1 +\vecnorm*{\vect \Psi(v)}^q\bigr)\mu_j(\d v)\\
    \label{eq:qmoment_Pmu3}
    %&\leq  C \Bigl(1+\mu_j^{(1)}+\mu_j^{(2)} \Bigr),
    &\leq  C \Bigl(1+\mathcal{M}_{q}(\mu_j) \Bigr),
\end{align}
\end{subequations}
where~\eqref{eq:fubini_Pmu} follows from Fubini's theorem, the inequality~\eqref{eq:qmoment_Pmu2} from properties of Gaussians and~\cref{assumpenum:assumption5},
and~\eqref{eq:qmoment_Pmu3} from application of~\cref{assumpenum:assumption2}.
We note that in~\eqref{eq:qmoment_Pmu3} the constant $C$ depends on $\kappa_{\vect \Psi}, \mat \Sigma$.
Now, to obtain a lower bound on the denominator of~\eqref{eq:q_momentPmu}, we observe that
\begin{align}
    \nonumber
    \int_{\real^{\dimu}} \exp \left( - \frac{1}{2} \bigl\lvert y_{j+1}^{\dagger} - \vect h(u) \bigr\rvert_{\mat \Gamma}^2 \right) \op P \mu_j(u) \, \d u &\geq \exp \left( - \frac{1}{2} \int_{\real^{\dimu}} \bigl\lvert y_{j+1}^{\dagger} - \vect h(u) \bigr\rvert_{\mat \Gamma}^2 \op P \mu_j(u) \right)\\
    \nonumber
    &\geq C \exp \left( -  \norm*{\mat \Gamma^{-1}} \int_{\real^{\dimu} }\lvert h(u)\rvert^2  \op P \mu_j(u) \right) \\
    \label{eq:similar_reasoning}
    &\geq C \exp \Bigl( - 4\kappa^2_{\vect h}\kappa^2_{\vect \Psi} \norm*{\mat \Gamma^{-1}} \mathcal{M}_2(\mu_j) \Bigr),
\end{align}
where the first inequality follows by application of Jensen's inequality and~\eqref{eq:similar_reasoning} follows from the calculation leading to~\eqref{eq:Pmuh_final}.
We note that the $C$ in~\eqref{eq:similar_reasoning} is a constant depending on $\kappa_y, \kappa_{\vect \Psi}, \kappa_{\vect h}, \mat \Sigma, \mat \Gamma$.
Therefore, by combining~\eqref{eq:similar_reasoning} with~\eqref{eq:qmoment_Pmu3},
it is possible to deduce from \eqref{eq:q_momentPmu} that
\begin{equation}
    \label{eq:bound_moments_true}
    \mathcal{M}_q(\mu_{j+1})\leq C \exp \Bigl(4\kappa^2_{\vect h}\kappa^2_{\vect \Psi}\norm*{\mat \Gamma^{-1}}\mathcal{M}_2(\mu_j)\Bigr)\Bigl(1+\mathcal{M}_q(\mu_j) \Bigr),
\end{equation}
where $C$ is a constant depending on $\kappa_{\vect h}, \kappa_{\vect \Psi}, \kappa_{y}, \mat \Sigma, \mat \Gamma$.
\end{proof}

\begin{remark}

    \begin{itemize}
        \item
            In some situations,
            the bound~\eqref{eq:bound_moments_true} is overly pessimistic.
            For example, if $h$ satisfies~\cref{assumpenum:assumption3}
            as well as the inequality $\bigl\lvert h(u) - y^{\dagger}_{j+1} \bigr\rvert \geq c_{\ell} \left(|u| - 1\right)$ for all $u \in \real^{\dimu}$
  %          \[
   %             \forall u \in \real^{\dimu}, \qquad
    %            \bigl\lvert h(u) - y^{\dagger}_{j+1} \bigr\rvert
     %           \geq c_{\ell} \left(|u| - 1\right)
      %      \]
            for some positive $c_\ell$,
            then by~\cite[Proposition A.3]{2023arXiv231207373J} for all $q > 0$
            there is $C = C(c_\ell, q)$ such that
            \[
                \forall \mu \in \mathcal P(\real^{\dimu}), \qquad
                \op L_j \mu \Bigl[ |x|^q \Bigr]
                \leq C \mu \Bigl[ |x|^q \Bigr].
            \]
            In this setting, better control of the moments can be achieved than in~\eqref{eq:bound_moments_true}.

        \item
            In obtaining~\eqref{eq:bound_moments_true},
            we did not use any information on~$\op P \mu_j$ other than the moment bound~\eqref{eq:qmoment_Pmu3}.
            With this approach, the presence of an exponential in the bound~\eqref{eq:bound_moments_true} is to be expected.
            Indeed, consider the case where $\dimu = 2$, $\dimy = 1$, $\mat \Gamma = \frac{1}{4}$, $y^\dagger_{j+1} = 0$,
            and
            \[
                h(u) = u_1, \qquad u = \begin{pmatrix} u_1 \\ u_2 \end{pmatrix},
                \qquad \mu_R = \frac{ x_R^2 \delta_{(R,0)} + \delta_{(0, x_R)} } {x_R^2 + 1},
                \qquad x_R = \e^{\frac{R^2}{2}}.
            \]
            Then, as $R \to \infty$, $\mu_R \left[ |u|^2 \right]\sim R^2$, while
 %           \[
  %              \mu_R \left[ |u|^2 \right]
   %             = \frac{x_R^2 R^2 + x_R^2}{x_R^2 + 1} \sim R^2,
    %            \qquad
     %       \]
      %      However
            \[
                \op L_j \mu_R \left[ |u|^2 \right]
                = \frac{\int_{\real^2} (u_1^2 + u_2^2) \exp \left( - 2 u_1^2 \right) \, \mu_R(\d u_1 \d u_2)}
                {\int_{\real^2} \exp \left( - 2 u_1^2 \right) \, \mu_R(\d u_1 \d u_2)}
                = \frac{x_R^2 \e^{-2 R^2} R^2  + x_R^2}{x_R^2 \e^{-2 R^2} + 1} \sim x_R^2 = \e^{R^2}.
            \]
            Thus, for large~$R$,
            the second moment of $\op L_j \mu_R$ is approximately equal to the exponential of the second moment of $\mu_R$.
    \end{itemize}
\end{remark}

\begin{lemma}    [Moment Bound for the Approximate Filtering Distribution]
    \label{theorem:approximate_filter_bounded_moments}
    Let $q \ge 2$ be an integer.
    Assume that the probability measures~$(\mu^{\kalman}_j)_{j \in \range{0}{J}}$ are obtained from the dynamical system~\eqref{eq:compact_mf_enkf}
    initialized at the probability measure $\mu^{\kalman}_0\in~\mathcal P(\real^{\dimu})$ with bounded $q$th polynomial order moment $\mathcal{M}_q\bigl(\mu_0^{\kalman}\bigr)<\infty$.
    If \cref{assumpenum:assumption1,assumpenum:assumption2,assumpenum:assumption3,assumpenum:assumption5} hold
    then there exists~$C = C\Bigl(\mathcal{M}_q\bigl(\mu_0^{\kalman}\bigr),J,\kappa_y, \kappa_{\vect \Psi}, \kappa_{\vect h}, \mat \Sigma, \mat \Gamma\Bigr)$
    such that
    \[
        \max_{j \in \range{0}{J}} \mathcal{M}_q\bigl(\mu^{\kalman}_j\bigr)\leq C.
    \]
\end{lemma}
\begin{proof}
    We begin by noting that $\mu^{\kalman}_{j+1} = \op T_j \op Q \op P \mu^{\kalman}_j,$
    for $j \in \range{0}{J-1}$. Thus
    \begin{align}
    \nonumber
    \mathcal{M}_q\bigl(\mu^{\kalman}_{j+1} \bigr)
    &= \int_{\real^{\dimu}} |u|^q\op T_j\op Q \op P \mu_j^{\kalman}(\d u) \\
    \nonumber
    &=\int_{\real^{\dimu} \times \real^{\dimy}} \vecnorm[\Big]{\mathscr{T} (u, y;\op Q \op P \mu_j^{\kalman}, y^\dagger_{j+1})}^q \,\op Q \op P \mu_j^{\kalman}(u,y)\,\d y\, \d u \\
    \nonumber
    &=\int_{\real^{\dimu}\times \real^{\dimy}} \vecnorm[\Big]{u +\mathcal C^{uy} (\op Q \op P \mu_j^{\kalman}) \mathcal C^{yy}(\op Q \op P \mu_j^{\kalman})^{-1}\bigl(y^\dagger_{j+1}-y \bigr)}^q\,\op Q \op P \mu_j^{\kalman}(u,y)\,\d y\, \d u\\
    \nonumber
    &\leq C\int_{\real^{\dimu}\times \real^{\dimy}}|u|^q \,\op Q \op P \mu_j^{\kalman}(u,y)\,\d y\, \d u \\
    \label{eq:boundedness_of_covariance}
    &\qquad \qquad + C \Bigl(1+ \mathcal{M}_2\bigl(\mu^{\kalman}_{j} \bigr)\Bigr)^{2 q} \cdot \left(1+\int_{\real^{\dimu}\times \real^{\dimy}}|y|^q \,\op Q \op P \mu_j^{\kalman}(u,y)\,\d y\, \d u \right),
    \end{align}
    where in~\eqref{eq:boundedness_of_covariance} the constant $C$ depends on $\kappa_y, \kappa_{\vect \Psi}, \kappa_{\vect h}, \mat \Sigma, \mat \Gamma$
    and where the dependence on the second moment of $\mu_j^{\kalman}$ in the second term is derived from~\eqref{eq:bound_on_QPmu_covariance_new_upper} and \eqref{eq:bound_on_QPmu_covariance_new_lower}.
    By using the definitions of $\op Q$ and $\op P$ and by applying reasoning analogous to~\eqref{eq:qmoment_Pmu3}, we deduce that
    \begin{align}
        \nonumber
        \mathcal{M}_q\bigl(\mu^{\kalman}_{j+1} \bigr) &\leq C \Bigl(1+ \mathcal{M}_2\bigl(\mu^{\kalman}_{j} \bigr)\Bigr)^{2q}\cdot \Bigl(1+ \mathcal{M}_q\bigl(\mu^{\kalman}_{j} \bigr)\Bigr),\\
        \nonumber &\leq C\Bigl(1+ \mathcal{M}_q\bigl(\mu^{\kalman}_{j} \bigr)\Bigr)^{2q+1}
    \end{align}
    where $C$ is a constant depending on $\kappa_y, \kappa_{\vect \Psi}, \kappa_{\vect h}, \mat \Sigma, \mat \Gamma$. Iteration gives the desired result.
\end{proof}

\begin{lemma}
    \label{lemma:moment_bound}
    Let $\mu_1, \mu_2 \in \mathcal P(\real^n)$ have finite second moments, then
    the following bounds hold:
    \begin{align*}
        \bigl\lvert \mathcal M(\mu_1) - \mathcal M(\mu_2) \bigr\rvert &\leq \frac{1}{2} d_g(\mu_1, \mu_2), \\
        \bigl\lVert \mathcal C(\mu_1) - \mathcal C(\mu_2) \bigr\rVert &\leq \left(1 + \frac{1}{2} \vecnorm{\mathcal M(\mu_1) + \mathcal M(\mu_2)}\right) \, d_g(\mu_1, \mu_2).
    \end{align*}
\end{lemma}
\begin{proof}
The lemma as stated may be found in \cite[Lemma B.4]{carrillo2022ensemble}, where a complete proof is given.
% Recall that $g(v) := 1 + \vecnorm{v}^2$.
%     Let $\vect m_i = \mathcal M(\mu_i)$ and $\mat \Sigma_i = \mathcal C(\mu_i)$ for $i = 1, 2$.
%     Notice that $|2a^\t u|\le g(u)$ if $|a|=1$, so
%     \begin{align}
%         \label{eq:bound_diff_means}
%         \vecnorm{\vect m_1 - \vect m_2}
%         = \sup_{\vecnorm{\vect a} = 1} \abs*{ \vect a^\t \bigl(\vect m_1 - \vect m_2\bigr) }
%         = \sup_{\vecnorm{\vect a} = 1} \abs*{\mu_1\bigl[\vect a^\t u\bigr] - \mu_2\bigl[\vect a^\t u\bigr]}
%         \leq \frac{1}{2} d_g(\mu_1, \mu_2),
%     \end{align}
%     where the supremum is over the unit sphere in $\real^n$,
%     centered at the origin and in the Euclidean distance.
%     Similarly
%     \begin{align}
%         \notag
%         \norm*{\mat \Sigma_1 - \mat \Sigma_2}
%         &= \sup_{\vecnorm{\vect a} = 1} \abs*{\vect a^\t \mat \Sigma_1 \vect a - \vect a^\t \mat \Sigma_2 \vect a} \\
%         \notag
%         &\leq \sup_{\vecnorm{\vect a} = 1}\left\{ \abs*{\mu_1\bigl[\lvert \vect a^\t u \rvert^2\bigr] - \mu_2\bigl[\lvert \vect a^\t u \rvert^2\bigr]}
%         + \abs*{\mu_1\bigl[\vect a^\t u\bigr]^2 - \mu_2\bigl[\vect a^\t u\bigr]^2}\right\} \\
%         \notag
%         &\leq  d_g(\mu_1, \mu_2) +  \sup_{\vecnorm{\vect a} = 1}\left|\mu_1\bigl[\vect a^\t u\bigr] + \mu_2\bigl[\vect a^\t u\bigr]\right|\left|\mu_1\bigl[\vect a^\t u\bigr] - \mu_2\bigl[\vect a^\t u\bigr]\right|\\
%         \label{eq:bound_diff_cov}
%         &\leq \left(1 + \frac{1}{2} \vecnorm{\vect m_1 + \vect m_2}\right) \, d_g(\mu_1, \mu_2),
%     \end{align}
%     which concludes the proof.
\end{proof}

%%%%%%%%%%%%%%%%%%%%%%%%%%%%%%%%%%%%%%%%%%%%%%%%%
\subsection{Action of \texorpdfstring{$\op T_j$}{} on Gaussians}
\label{sub:mean_field_is_conditioning_for_gaussians}
\begin{lemma}
    [$\op B_j \op G = \op T_j \op G$]
    \label{lemma:mean_field_map}
    Fix $y^\dagger_{j+1}\in\real^{\dimy}$ and let $\pi$ be a Gaussian measure over $\real^{\dimu} \times \real^{\dimy}$.
        Then the probability measure $\op B_j \pi$, with $\op B_j$ defined in~\eqref{eq:definition_B} is equivalent to the probability measure $\op T_j\pi =\mathscr T (\placeholder, \placeholder;\pi, y^\dagger_{j+1})_\sharp \pi$, where $\mathscr T$ is defined in~\eqref{eq:mean_field_map}.
\end{lemma}
\begin{proof}
The lemma as stated may be found in \cite[Lemma B.5]{carrillo2022ensemble}, where a complete proof is given.
\end{proof}

\subsection{Stability Results}
\label{sub:stability_results}

\begin{lemma}
    [Map $\op P$ is Lischitz]
    \label{lemma:lipschitz_p}
    Suppose that~$\Sigma$ and~$\Psi$ satisfy~\cref{assumpenum:assumption2,assumpenum:assumption5}, respectively.
    Then it holds that
    \[
        \forall (\mu, \nu) \in \mathcal P(\real^{\dimu}) \times \mathcal P(\real^{\dimu}), \qquad
        d_g(\op P \mu, \op P \nu)
        \leq \Bigl(1 + 2\kappa_{\vect \Psi}^2 + \trace(\mat \Sigma)\Bigr) \, d_g (\mu, \nu).
    \]
\end{lemma}
\begin{proof}
    By definition of~$\op P$,
    it holds that
    \[
        \op P \mu(\d u) = \int_{\real^{\dimu}} p(v, \d u)  \, \mu(\d v),
        \qquad \text{ where } p(v, \d u) := \frac{\exp \left( - \frac{1}{2} \vecnorm*{u - \vect \Psi(v)}_{\mat \Sigma}^2 \right)}{\sqrt{(2\pi)^{\dimu} \det \mat \Sigma}} \, \d u.
    \]
    Let $g(u)=1+|u|^2$, as in \cref{definition:weighted_total_variation},
    and take any function $f\colon \real^{\dimu} \to \real$ such that $|f| \leq g$.
    \Cref{assumpenum:assumption2} implies that
    \begin{align*}
        \left\lvert \int_{\real^{\dimu}} f(u)  \, p(v, \d u)  \right\rvert
        \leq \int_{\real^{\dimu}} g(u)  \, p(v, \d u)
        &= 1+\left\lvert \vect \Psi(v) \right\rvert^2  + \trace(\mat \Sigma) \\
        &\leq 1 + 2\kappa_{\Psi}^2 + 2 \kappa_{\Psi}^2 \abs{v}^2 + \trace(\mat \Sigma)
        \leq \Bigl( 1 + 2\kappa_{\Psi}^2 + \trace(\mat \Sigma) \Bigr) g(v).
    \end{align*}
    By Fubini's theorem,
    it follows that
    \begin{align*}
        \Bigl\lvert \op P \mu [f] - \op P \nu [f] \Bigr\rvert
        &= \biggl\lvert \int_{\real^{\dimu}} \left( \int_{\real^{\dimu}} f(u)  \, p(v, \d u)  \right) \bigl(\mu(\d v) - \nu(\d v)\bigr)  \biggr\rvert \\
        &= \Bigl( 1 + 2\kappa_{\Psi}^2 + \trace(\mat \Sigma) \Bigr) \Bigl\lvert \mu [g] - \nu [g] \Bigr\rvert
        \leq \Bigl(1 + 2\kappa_{\vect \Psi}^2 + \trace(\mat \Sigma)\Bigr) \, d_g(\mu, \nu),
    \end{align*}
    thus concluding the proof.
\end{proof}

\begin{lemma}
    [Map $\op Q$ is Lipschitz]
    \label{lemma:lipschitz_Q}
    Suppose that $\Gamma$ and $h$ satisfy~\cref{assumpenum:assumption3,assumpenum:assumption5}, respectively.
    Then it holds that
    \begin{equation}
        \label{eq:lipschitz_Q}
        \forall (\mu, \nu) \in \mathcal P(\real^{\dimu}) \times \mathcal P(\real^{\dimu}), \qquad
        d_{g}(\op Q \mu, \op Q \nu)
        \leq \Bigl(1 + 2\kappa_{\vect h}^2 + \trace(\mat \Gamma) \Bigr) d_{g}(\mu, \nu).
    \end{equation}
\end{lemma}
\begin{proof}
    Take $f \colon \real^{\dimu} \times \real^{\dimy} \to \real$ satisfying $|f| \leq g$,
    where $g(u, v) = 1 + |u|^2 + |y|^2$.
    By Fubini's theorem,
    it holds that
    \[
         \op Q \mu[f] - \op Q \nu[f]
        =  \int_{\real^{\dimu}} \Pi f(u)  \, \Bigl( \mu(\d u) - \nu(\d u) \Bigr),
        \qquad \text{ where } \Pi f(u) := \int_{\real^{\dimy}}  f(u, y)  \, \normal\bigl(\vect h(u), \mat \Gamma\bigr) (\d y).
    \]
    The function $\Pi f \colon \real^{\dimu} \to \real$ satisfies
    \begin{align*}
        \forall u \in \real^{\dimu}, \qquad
        \left\lvert \Pi f(u) \right\rvert
        &\leq \int_{\real^{\dimy}}  \bigl\lvert f(u, y) \bigr\rvert  \, \normal\bigl(\vect h(u), \mat \Gamma\bigr) (\d y) \\
        &\leq \int_{\real^{\dimy}}  \left( 1 + |u|^2 + |y|^2 \right)  \, \normal\bigl(\vect h(u), \mat \Gamma\bigr) (\d y)  \\
        &= 1 + |u|^2 + \bigl\lvert h(u) \bigr\rvert^2 + \trace(\Gamma)
        \leq \Bigl(1 + 2\kappa_{\vect h}^2 + \trace(\mat \Gamma) \Bigr) \left(1 + \vecnorm{u}^2\right).
    \end{align*}
    Therefore,
    we deduce that
    \[
        \Bigl\lvert \op Q \mu[f] - \op Q \nu[f]  \Bigr\rvert
        \leq (1 + 2\kappa_{\vect h}^2 + \trace(\mat \Gamma) \Bigr) \, d_g(\mu, \nu),
    \]
    which concludes the proof.
\end{proof}

\begin{lemma}[Stability of Map $\op B_j$]
    \label{lemma:missing_lemma}
    Suppose that~\cref{assumpenum:assumption1,assumpenum:assumption2,assumpenum:assumption3,assumpenum:assumption5} are satisfied.
    Then for any probability measure $\mu \in \mathcal P(\real^{\dimu})$ with $\mathcal{M}_2(\mu)<\infty$  there exists a constant $C_{\op B} = C_{\op B}\bigl(\mathcal{M}_2(\mu), \kappa_y, \kappa_{\vect \Psi}, \kappa_{\vect h}, \mat \Sigma, \mat \Gamma\bigr) > 0$
    such that
    \[
        \forall j \in \llbracket 0, J \rrbracket,
        \qquad
        d_{g}(\op B_j \op G \op Q \op P \mu, \op B_j \op Q \op P \mu)
        \leq C_{\op B} d_{g}(\op G \op Q \op P \mu, \op Q \op P \mu).
    \]
\end{lemma}
\begin{proof}
    The proof of this lemma follows the steps of the proof of \cite[Lemma B.9]{carrillo2022ensemble}; however, each step requires different bounds reflecting the assumptions on $\vect \Psi$ and $\vect h$ in this paper. For ease of exposition we write $y^\dagger = y^\dagger_{j+1}$.
    We define the $y$-marginal densities
    \[
        \alpha_\mu(y):= \int_{\real^{\dimu}} \op G \op Q \op P\mu(u, y) \, \d u\,,
        \qquad
        \beta_\mu(y):=\int_{\real^{\dimu}} \op Q \op P\mu(u, y) \, \d u\,.
    \]
    By definition of~$\op B_j$,
    it holds that
    \begin{align}
        \notag
        &d_{g}(\op B_j \op G \op Q \op P \mu, \op B_j \op Q \op P \mu)
        = \int_{\real^{\dimu}} \left(1+|u|^2\right) \left| \frac{\op G \op Q \op P\mu(u,y^\dagger)}{\alpha_\mu(y^\dagger)} - \frac{\op Q \op P\mu(u,y^\dagger)}{\beta_\mu(y^\dagger)}\right| \, \d u\\
        \notag
        &\qquad \qquad \le \frac{1}{\alpha_\mu(y^\dagger)}\int_{\real^{\dimu}} \left(1+|u|^2\right) \bigl\lvert \op G \op Q \op P\mu(u,y^\dagger)-\op Q \op P\mu(u,y^\dagger)\bigr\rvert \, \d u \\
        \label{eq:main_equation}
        &\qquad \qquad \qquad + \left| \frac{\alpha_\mu(y^\dagger)-\beta_\mu(y^\dagger)}{\alpha_\mu(y^\dagger) \beta_\mu(y^\dagger)}\right|\,\int_{\real^{\dimu}} \left(1+|u|^2\right) \op Q \op P\mu(u,y^\dagger) \, \d u.
    \end{align}

    \paragraph{Step 1: bounding $\alpha_{\mu}(y^\dagger)$ and $\beta_{\mu}(y^\dagger)$ from below}
    The distribution $\alpha_{\mu}(\placeholder)$ is Gaussian with mean~$\mathcal M^y(\op Q \op P \mu)$ and covariance
    \begin{equation}
        \label{eq:covariance_matrix_second_gaussian}
        \mathcal C^{yy}(\op Q \op P \mu) = \mat \Gamma + \op P \mu\bigl[\vect h \otimes \vect h\bigr] - \op P \mu\bigl[\vect h\bigr] \otimes \op P \mu\bigl[\vect h\bigr].
    \end{equation}
    Clearly $\mathcal C^{yy}(\op Q \op P \mu)\succcurlyeq \mat \Gamma$.
    Furthermore,
    it holds that $\mathcal C^{yy}(\op Q \op P \mu) \preccurlyeq \mat \Gamma + 2\kappa_{\vect h}^2 \bigl(1+\trace(\mat \Sigma)+2\kappa_{\vect \Psi}^2\bigl(1+\mathcal{M}_2(\mu)\bigr) \bigr)\mat I_{\dimy}$ by~\eqref{eq:bound_covariance_pushforward}.
    Therefore, noting that~\eqref{eq:bound_on_QPmu_mean_new_2} implies that $\left\lvert \mathcal M^y(\op Q \op P \mu) \right\rvert \leq C$
    for a constant $C$ depending only on the parameters $\mathcal M_2(\mu), \kappa_{\vect h},\kappa_{\vect \Psi}, \mat \Sigma$,
    we obtain
        \begin{align}
            \nonumber
            \alpha_\mu(y)
            &= \frac{1}{\sqrt{(2\pi)^{\dimu} \det \bigl( \mathcal C^{yy}(\op Q \op P \mu) \bigr)}} \exp \left( - \frac{1}{2} \bigl(y - \mathcal M^{y} (\op Q \op P \mu) \bigr)^\t \mathcal C^{yy}(\op Q \op P \mu)^{-1} \bigl(y - \mathcal M^{y} (\op Q \op P \mu) \bigr) \right) \\
            \label{eq:lower_bound_alpha}
            &\geq \frac{\exp\Bigl(- \frac{1}{2}  \bigl(|y| + C\bigr)^2  \norm*{\mat \Gamma^{-1}}\Bigr)}{\sqrt{(2\pi)^{\dimy} \det \bigl(\mat \Gamma +  2\kappa_{\vect h}^2 \bigl(1+\trace(\mat \Sigma)+2\kappa_{\vect \Psi}^2\bigl(1+\mathcal{M}_2(\mu)\bigr) \bigr) \mat I_{\dimy} \bigr)}}.
        \end{align}
    The function $\beta_{\mu}$ can be bounded from below using similar reasoning.
    Indeed, by~\cref{assumpenum:assumption2,assumpenum:assumption3,assumpenum:assumption5} we have that for all $y \in \real^{\dimy}$,
    \begin{subequations}
    \begin{align}
        \nonumber
        \beta_{\mu}(y)
        &= \int_{\real^{\dimu}} \op Q \op P \mu(u, y) \, \d u
        = \int_{\real^{\dimu}} \frac{\exp\left(- \frac{1}{2} \bigl(y - \vect h(u)\bigr)^\t \mat \Gamma^{-1}  \bigl(y - \vect h(u)\bigr) \right)}{\sqrt{(2\pi)^{\dimy} \det(\mat\Gamma)}} \op P\mu(u) \, \d u \\
        \label{eq:jensen}
        &\geq \frac{1}{\sqrt{(2\pi)^{\dimy} \det (\mat \Gamma)}} \exp\Bigl(- \norm*{\mat \Gamma^{-1}}|y|^2 - \norm*{\mat \Gamma^{-1}}\kappa^2_{\vect h} \int_{\real^{\dimu}} \bigl(2 + |u|^2\bigr) \op P\mu(u) \, \d u\Bigr)\\
        \label{eq:lower_bound_beta}
        % &\geq \frac{C\exp\Bigl(- \norm*{\mat \Gamma^{-1}}|y|^2 - 2\norm*{\mat \Gamma^{-1}}\kappa^2_{\vect h}\kappa^2_{\vect \Psi} \mathcal{M}_2(\mu)\Bigr)}{\sqrt{(2\pi)^{\dimy} \det (\mat \Gamma)}},
        &\geq C\exp\Bigl(- \norm*{\mat \Gamma^{-1}}|y|^2 \Bigr),
    \end{align}
    \end{subequations}
    where we applied Jensen's inequality in~\eqref{eq:jensen} and
    the constant in~\eqref{eq:lower_bound_beta} depends on $\mathcal M_2(\mu), \kappa_{\vect \Psi}, \kappa_{\vect h}, \mat \Sigma, \mat \Gamma$.
    \paragraph{Step 2: bounding the first term in~\eqref{eq:main_equation}}
    First note that,
    for fixed $u \in \real^{\dimu}$,
    the functions~$\op Q \op P \mu(u, \placeholder)$ and~$\op G \op Q \op P \mu(u, \placeholder)$ are Gaussians up to constant factors,
    with covariance matrices given respectively by~$\mat \Gamma$ and
    \begin{equation}
        \label{eq:conditioned_covariance}
        \mathcal C^{yy}(\op Q \op P \mu) - \mathcal C^{yu} (\op Q \op P \mu) \mathcal C^{uu}(\op Q \op P \mu)^{-1}\mathcal C^{uy}(\op Q \op P \mu),
    \end{equation}
    where we have used the formula for the covariance of the conditional distribution of a Gaussian. We note that $\mathcal C^{yu} (\op Q \op P \mu) \mathcal C^{uu}(\op Q \op P \mu)^{-1}\mathcal C^{uy}(\op Q \op P \mu)$ is positive semi-definite, hence by~\eqref{eq:covariance_matrix_second_gaussian} and its upper bound it follows that the matrix~\eqref{eq:conditioned_covariance} is bounded from above by~${2}\kappa_{\vect h}^2\bigl(1+\trace(\mat \Sigma)+2\kappa_{\vect \Psi}^2\bigl(1+\mathcal{M}_2(\mu)\bigr) \bigr) \mat I_{\dimy}  + \mat \Gamma$.
    As shown in the proof of \cite[Lemma B.9]{carrillo2022ensemble}, the matrix \eqref{eq:conditioned_covariance} is bounded from below by $\Gamma$ (see \cite[Equation (B.20)]{carrillo2022ensemble}).
% In view of~\eqref{eq:decomposition_covariance_new},
    % we find that
    % \begin{align*}
    %     &\mathcal C^{yy}(\op Q \op P \mu) - \mathcal C^{yu} (\op Q \op P \mu) \mathcal C^{uu}(\op Q \op P \mu)^{-1}\mathcal C^{uy}(\op Q \op P \mu) \\
    %     &\qquad =
    %     \mat \Gamma + \Bigl( \mathcal C^{yy}(\phi_\sharp \op P  \mu) - \mathcal C^{yu} (\phi_\sharp \op P  \mu) \mathcal C^{uu}(\phi_\sharp \op P  \mu)^{-1}\mathcal C^{uy}(\phi_\sharp \op P  \mu) \Bigr)
    %     \succcurlyeq \mat \Gamma.
    % \end{align*}
    % The term in brackets is the Schur complement of the block $\mathcal C^{uu} (\phi_\sharp \op P  \mu)$ of the matrix $\mathcal C(\phi_\sharp \op P  \mu)$,
    % and so it is indeed positive semi-definite.
    It follows from using~\cref{lemma:gaussians_relation_1_and_inf_norms} in~\cref{app:A2},
    with parameter $\tau=\tau\bigl(\mathcal{M}_2(\mu), \kappa_{\vect h}, \kappa_{\vect \Psi}, \mat \Sigma, \mat \Gamma\bigr)$, that
    \begin{equation}
        \label{eq:bound_first_term}
        \int_{\real^{\dimu}} \left(1+|u|^2\right) \left| \op G \op Q \op P\mu(u,y)-\op Q \op P\mu(u,y)\right|\d u \leq C d_{g} (\op G \op Q \op P\mu, \op Q \op P \mu),
    \end{equation}
     where $C$ is a constant depending on $\mathcal{M}_2(\mu), \kappa_{\vect h}, \kappa_{\vect \Psi}, \mat \Sigma, \mat \Gamma$.
     We refer to \cite[Equation (B.21)]{carrillo2022ensemble} for the detailed steps used to establish this bound.

    \paragraph{Step 3: bounding the second term in~\eqref{eq:main_equation}}
    In view of~\eqref{eq:bound_first_term}, it holds that
    \[
        \bigl\lvert \alpha_\mu(y)-\beta_\mu(y) \bigr\rvert \leq \int_{\real^{\dimu}} \left(1+|u|^2\right) \bigl\lvert  \op G \op Q \op P\mu(u,y)-\op Q \op P\mu(u,y)\bigr \rvert\,\d u
         \leq C d_{g} (\op G \op Q \op P\mu, \op Q \op P \mu).
    \]
    Now,
    since $\op Q \op P \mu(u, \placeholder) / \op P \mu(u)$ defines a Gaussian density which has covariance $\mat \Gamma$ and is bounded uniformly from above by $\bigl((2\pi)^{\dimy} \det (\mat \Gamma)\bigr)^{-1/2}$,
    we also have that
    \[
        \int_{\real^{\dimu}} \left(1+|u|^2\right) \op Q \op P\mu(u,y) \, \d u
        \leq \int_{\real^{\dimu}}  \frac{ \left(1 + |u|^2\right) \op P \mu(u)}{\sqrt{(2\pi)^{\dimy} \det(\Gamma})} \, \d u
        = \frac{1 + \trace\bigl(\mathcal C(\op P\mu)\bigr) + \vecnorm*{\mathcal M(\op P\mu)}^2}{\sqrt{(2\pi)^{\dimy} \det (\mat \Gamma})}.
    \]
    By the moment bounds in~\cref{lemma:bound_on_Pmu_new},
    the right-hand side is bounded from above by a constant which depends on~$\mathcal{M}_2(\mu), \kappa_{\vect \Psi},\mat\Sigma$.
    % the first and second moments of $\op P \mu$ are

    \paragraph{Step 4: concluding the proof}
    Putting together the above bounds, we conclude that
    \begin{align*}
        &d_{g}(\op B_j \op G \op Q \op P \mu, \op B_j \op Q \op P \mu)
        \le \frac{C\bigl(\mathcal{M}_2(\mu),\kappa_{\vect \Psi}, \kappa_{\vect h}, \mat \Sigma, \mat \Gamma\bigr)}{\alpha_{\mu}(y^\dagger_{j+1})} \left(1 + \frac{1}{\beta_\mu(y^\dagger_{j+1})} \right)  d_{g} (\op G \op Q \op P\mu, \op Q \op P \mu).
    \end{align*}
    Applying the inequalities~\eqref{eq:lower_bound_alpha} and~\eqref{eq:lower_bound_beta} yields the desired result.
\end{proof}

\begin{lemma}
[Stability of Map $\op T_j$]
    \label{lemma:lipschitz_mean_field_affine}
    Suppose that \cref{assumption:ensemble_kalman} is satisfied.
    Then, for all $R \geq 1$, it holds that for any $\pi \in \mathcal P_R(\real^{\dimu} \times \real^{\dimy})$ and~$p \in \{\op Q \op P \mu: \mu \in \mathcal P(\real^{\dimu}) \,\text{and }\mathcal{M}_{2\cdot\max\{3+\dimu, 4+\dimy\}}(\mu) < \infty\}\subset \mathcal P(\real^{\dimu}\times\real^{\dimy})$, there is
    $L_{\op T} := L_{\op T}\bigl(R, \mathcal{M}_{2\cdot\max\{3+\dimu, 4+\dimy\}}(\mu), \kappa_y, \kappa_{\vect \Psi}, \kappa_{\vect h}, \ell_{\vect h}, \mat \Sigma, \mat \Gamma\bigr)$, such that
    \[
        \forall j \in \range{1}{J}, \qquad
        d_{g}(\op T_j \pi, \op T_j p)
        \leq L_{\op T} \, d_{g}(\pi, p).
    \]
%    for some $L_{\op T} = L_{\op T}\bigl(R, \mathcal{M}_{\max\{3+\dimu, 4+\dimy\}}(\mu), \kappa_y, \kappa_{\vect \Psi}, \kappa_{\vect h}, \ell_{\vect h}, \mat \Sigma, \mat \Gamma\bigr)$.
\end{lemma}

\begin{proof}
    By the results of \cref{lemma:bound_on_QPmu_new}, it holds that $\op Q \op P \mu \in \mathcal P_{\widetilde R}(\real^{\dimu} \times \real^{\dimy})$ for some $\widetilde R(\mathcal M_2(\mu), \kappa_{\vect \Psi}, \kappa_{\vect h}, \Sigma, \Gamma) \ge 1$.
    Using analogous notation to the one found in the proof of \cite[Lemma B.10]{carrillo2022ensemble} we define
    \[
        r = \max \Bigl\{R, \widetilde R, \kappa_y \Bigr\}.
    \]
     Letting $\mat K = \mathcal C(\pi)$, $\mat S = \mathcal C(p)$ and $y^\dagger = y^\dagger_{j+1}$, we also define the affine maps $\mathscr T^{\pi}$ and $\mathscr T^{p}$ corresponding to use
    of covariance information at the probability measures $\pi$ and $p= \op Q \op P \mu$,
    \begin{alignat*}{2}
        \mathscr T^{\pi}(u, y) &= u + \mat A_{\pi} (y^\dagger - y), \qquad & \mat A_{\pi} := \mat K_{uy} \mat K_{yy}^{-1}, \\
        \mathscr T^{p}(u, y) &= u + \mat A_{p} (y^\dagger - y), \qquad & \mat A_{p} := \mat S_{uy} \mat S_{yy}^{-1}.
    \end{alignat*}
    By a straightforward application of the triangle inequality, it holds that
    \begin{equation}
        \label{eq:main_equation_filtering}
        d_{g}(\op T_j \pi, \op T_j p)
        \leq  d_{g}(\mathscr T^{\pi}_{\sharp} \pi, \mathscr T^{\pi}_{\sharp} p) + d_{g} (\mathscr T^{\pi}_{\sharp} p, \mathscr T^{p}_{\sharp} p).
    \end{equation}
    In the following steps we will separately bound the two terms on the right hand side of~\eqref{eq:main_equation_filtering}. Before proceeding we outline two auxiliary bounds that will be used in the rest of the proof. These bounds are identical to the ones found in the proof of \cite[Lemma B.10]{carrillo2022ensemble}; we include statements here for expository purposes.
    Noting that the operator 2-norm of any submatrix is bounded from above by the operator 2-norm of the full matrix, we observe (see \cite[Equation (B.23)]{carrillo2022ensemble}) that
    \begin{equation}
        \label{eq:lipschitz_mean_field_auxiliary_1}
        \matnorm{\mat A_{\pi}} \leq \matnorm{\mat K_{uy}} \matnorm{\mat K_{yy}^{-1}}
        \leq \matnorm{\mat K} \matnorm{\mat K^{-1}} \leq r^4.
    \end{equation}
    Note that the above bound similarly holds for~$\mat A_p$.
    Using \eqref{eq:lipschitz_mean_field_auxiliary_1} and assuming without loss of generality that $r>1$, we deduce (see \cite[Equation (B.24)]{carrillo2022ensemble}) that
    \begin{equation}
        \label{eq:lipschitz_mean_field_auxiliary_2}
        \matnorm{\mat A_{\pi} - \mat A_p}
        \leq \matnorm{(\mat K_{uy} - \mat S_{uy}) \mat K_{yy}^{-1}} + \matnorm{\mat S_{uy} \left(\mat K_{yy}^{-1} - \mat S_{yy}^{-1}\right)}
        \leq 2r^6  \matnorm{\mat K - \mat S} \leq 2r^6 (1 + 2r) \, d_{g}(\pi, p),
    \end{equation}
    where the second inequality follows again from fact that the 2-norm of any submatrix is bounded from above by the 2-norm of the full matrix,
    while the last inequality follows from the result in~\cref{lemma:moment_bound}.

    \paragraph{Bounding the first term in~\eqref{eq:main_equation_filtering}}
    With an identical argument to the one used in the proof of \cite[Lemma B.10]{carrillo2022ensemble} we have that
    % \ec{The first term is bounded as is done in the proof of \cite{carrillo2022ensemble}[Lemma B.10]: we include the short argument here for completeness}. Let $f$ satisfy $\abs{f} \leq g$.
    % It is readily observed that
    % \[
    %     \bigl\lvert \mathscr T^{\pi}_{\sharp}\pi[f] - \mathscr T^{\pi}_{\sharp} p[f] \bigr\rvert
    %     = \bigl\lvert \pi[f\circ \mathscr T^{\pi}] - p[f\circ \mathscr T^{\pi}] \bigr\rvert.
    % \]
    % For any $(u, y) \in \real^{\dimu} \times \real^{\dimy}$, we have that
    % \begin{align*}
    %     \bigl\lvert f\circ \mathscr T^{\pi}(u, y) \bigr\rvert &= \abs*{f\Bigl(u + \mat A_{\pi} \bigl(y^\dagger - y\bigr)\Bigr)} \leq g\Bigl(u + \mat A_{\pi} \bigl(y^\dagger - y\bigr)\Bigr) \\
    %                                &= 1 + \bigl\lvert u + \mat A_{\pi} (y^\dagger - y) \bigr\rvert^2 \leq 1 + 3\vecnorm{u}^2 +  3 \lvert \mat A_{\pi} y^\dagger \rvert^2 + 3 \vecnorm{\mat A_{\pi} y}^2 \\
    %                                &\leq 3 \left(1 + \lvert \mat A_{\pi} y^\dagger \rvert^2\right) \max \left\{ 1, \matnorm{A_{\pi}}^2\right\} g(u, y).
    % \end{align*}
    % Hence, by~\eqref{eq:lipschitz_mean_field_auxiliary_1} we conclude that
    \begin{equation}
        \label{eq:stability_T_first_term}
        d_{g}(\mathscr T^{\pi}_{\sharp} \pi, \mathscr T^{\pi}_{\sharp} p)
        \leq 3 \left(1 + r^{10}\right) r^8 \, d_{g}(\pi, p).
    \end{equation}

    \paragraph{Bounding the second term in~\eqref{eq:main_equation_filtering}}
    Let $f$ again satisfy~$\abs{f} \leq g$. We note that
    \[
        \Bigl\lvert \mathscr T^{\pi}_{\sharp}p[f] - \mathscr T^{p}_{\sharp} p[f] \Bigr\rvert
        = \Bigl\lvert p[f\circ \mathscr T^{\pi}] - p[f\circ \mathscr T^{p}] \Bigr\rvert
        = \Bigl\lvert p[f\circ \mathscr T^{\pi} - f\circ \mathscr T^{p}] \Bigr\rvert.
    \]
    The last term may be expressed as
    \begin{align}
        \nonumber\Bigl\lvert p[f\circ \mathscr T^{\pi} - f\circ \mathscr T^{p}] \Bigr\rvert
        &= \abs*{\int_{\real^{\dimy}}\int_{\real^{\dimu}}
        \Bigl(f\left(u + \mat A_{\pi}\bigl(y^\dagger - y\bigr)\right) - f\left(u + \mat A_{p}\bigl(y^\dagger - y\bigr)\right)\Bigr) \, p(u,y) \, \d u \, \d y}\\
        \nonumber
        &=\int_{\real^{\dimy}} \int_{\real^{\dimu}} \bigl(f(u + \mat A_{\pi}  z ) - f(u + \mat A_{p}  z )\bigr) \,  p (u, y^\dagger -  z) \, \d u \, \d z \\
        \label{eq:filtering_second_term}
        &=\int_{\real^{\dimy}} \int_{\real^{\dimu}} f(v) \Bigl( p (v - \mat A_{\pi} z, y^\dagger -  z) -  p (v - \mat A_{p} z, y^\dagger -  z) \Bigr) \, \d v \, \d z,
    \end{align}
    where we used a change of variables in the second equality.
    Since $\mathcal{M}_{2\cdot\max\{3+\dimu, 4+\dimy\}}(\mu)<\infty$ by assumption, it follows from~\cref{lemma:lipschitz_density_new} and from the inequality $\min\{a,b\} \leq \sqrt{a} \sqrt{b}$ that there is a constant $C$ so that
    \begin{align*}
        &\Bigl\lvert  p (v - \mat A_{\pi} z, y^\dagger -  z) -  p (v - \mat A_{p} z, y^\dagger -  z) \Bigr \rvert \\
        & \qquad \leq  C\abs{\mat A_{\pi}z  - \mat A_{p} z} \cdot \max \left\{ \frac{1}{1 + \abs{v - \mat A_{\pi} z}^{\max\{3+\dimu, 4+\dimy\}}}, \frac{1}{1 + \abs{v - \mat A_{p} z}^{\max\{3+\dimu, 4+\dimy\}}} \right\} \\
        &\qquad\qquad\times \frac{1}{1 + \abs{y^{\dagger} - z}^{\max\{3+\dimu, 4+\dimy\}}} .
    \end{align*}
    %\uv{Should there not be a $\min$ in the above equation? Or are we using $\min\{a,b\} \leq \sqrt{a} \sqrt{b}$, but if so the powers are incorrect?}
    We apply this inequality
    to bound for fixed $z \in \real^{\dimy}$ the inner integral in~\eqref{eq:filtering_second_term}.
    Considering only the terms that depend on~$v$ gives
    \begin{align}
        \label{eq:first_integral_line}
        &\int_{\real^{\dimu}} \lvert f(v) \rvert  \max \left\{ \frac{1}{1 + \abs{v - \mat A_{\pi} z}^{\max\{3+\dimu, 4+\dimy\}}}, \frac{1}{1 + \abs{v - \mat A_{p} z}^{\max\{3+\dimu, 4+\dimy\}}} \right\} \, \d v \\
        \notag
        &\qquad \leq
        \int_{\real^{\dimu}} \frac{\lvert f(v) \rvert}{1 + \abs{v - \mat A_{\pi} z}^{\max\{3+\dimu, 4+\dimy\}}} \, \d v
        + \int_{\real^{\dimu}}   \frac{\lvert f(v) \rvert}{1 + \abs{v - \mat A_{p} z}^{\max\{3+\dimu, 4+\dimy\}}} \, \d v \\
        \notag
        &\qquad \leq
        \int_{\real^{\dimu}} \frac{1 + \lvert v \rvert^2}{1 + \abs{v - \mat A_{\pi} z}^{\max\{3+\dimu, 4+\dimy\}}} \, \d v
        + \int_{\real^{\dimu}}   \frac{1 + \lvert v \rvert^2}{1 + \abs{v - \mat A_{p} z}^{\max\{3+\dimu, 4+\dimy\}}} \, \d v
        \\
        \notag
        &\qquad \leq
        \int_{\real^{\dimu}} \frac{1 + \lvert w + \mat A_{\pi} z \rvert^2}{1 + \abs{w}^{\max\{3+\dimu, 4+\dimy\}}} \, \d w
        + \int_{\real^{\dimu}}   \frac{1 + \lvert w + \mat A_{p} z \rvert^2}{1 + \abs{w}^{\max\{3+\dimu, 4+\dimy\}}} \, \d w,
    \end{align}
    where we used that $\lvert f(v) \rvert \leq 1 + \vecnorm{v}^2$,
    as well as a change of variable in the last line.
    It follows that the integral in~\eqref{eq:first_integral_line}
    is bounded from above by
    \[
        C \Bigl(1 + \lvert \mat A_{\pi} z \rvert^2 + \lvert \mat A_{p} z \rvert^2  \Bigr) \leq C r^8 \left(1 + \vecnorm{z}^2 \right),
    \]
    where the inequality follows from~\eqref{eq:lipschitz_mean_field_auxiliary_1}.
    Finally, the resulting integral in the~$z$ variable can be bounded analogously, which gives
    \begin{align}
        \nonumber
        d_{g} (\mathscr T^{\pi}_{\sharp} p, \mathscr T^{p}_{\sharp} p)
        &\leq C r^8 \int_{\real^{\dimy}} \frac{ 1 + \vecnorm{z}^2 }{1 + \abs{y^{\dagger} - z}^{\max\{3+\dimu, 4+\dimy\}}} \lvert \mat A_{\pi} z - \mat A_{p} z \rvert  \, \d z \\
        \label{eq:stability_T_second_term}
        &\leq C r^{11} \matnorm{\mat A_{\pi} - \mat A_{p}} \leq C r^{18} d_{g}(\pi, p),
    \end{align}
    where the last inequality follows from~\eqref{eq:lipschitz_mean_field_auxiliary_2}.
    Combining~\eqref{eq:stability_T_first_term} and \eqref{eq:stability_T_second_term} yields the desired result.
\end{proof}

\section{Technical Results for Approximation Result \texorpdfstring{in~\cref{prop:auxiliary_close_to_affine}}{}}
\label{appendix:C}

In~\cref{lemma:stabilityG} we recall the local Lipschitz continuity result for the operator $\op G$ established in \cite{carrillo2022ensemble}. \cref{lemma:auxiliary_close_to_affine_new} establishes that the filtering distribution is a locally Lipschitz function of $(\Psi,h)$, viewed as a mapping from Banach space equipped with the
$\|\cdot\|_\infty$ norm into the space of probability measures metrized using the $d_g$ distance.
This is preceded by \cref{lemma:3bounds} which establishes bounds used to prove this Lipschitz property. The two lemmas do not require $\vect \Psi_0$ and $\vect h_0$ to be affine,
but simply require that they both satisfy~\cref{assumpenum:assumption2,assumpenum:assumption3}.
This is in contrast with the more specific setting of~\cref{prop:auxiliary_close_to_affine},
which imposes an affine assumption on $(\vect \Psi_0, \vect h_0)$.

\begin{lemma}
    \label{lemma:stabilityG}
    For all $R\geq1$, there exists $L_{\op G} = L_{\op G}(R,n)$ so that for any $\mu_1,\mu_2\in \mathcal{P}_R(\real^n)$ it holds that
    \[
    d_g(\op G\mu_1,\op G\mu_2) \leq L_{\op G}(R,n)\cdot d_g(\mu_1,\mu_2).
    \]
\end{lemma}
\begin{proof}
    The lemma as stated may be found in \cite[Lemma B.12]{carrillo2022ensemble}, where a complete proof is given.
\end{proof}

\begin{lemma}
    \label{lemma:3bounds}
    Suppose that the matrices $(\mat \Sigma, \mat \Gamma)$ satisfy~\cref{assumpenum:assumption5}.
    Fix $\kappa_{\Psi}, \kappa_h > 0$ and assume that~$\Psi_0 \colon \real^{\dimu} \to \real^{\dimu}$ and $\vect h_0 \colon \real^{\dimu} \to \real^{\dimy}$ are functions satisfying~\cref{assumpenum:assumption2,assumpenum:assumption3}.
    %\ec{Furthermore for~$\varepsilon \geq 0$, denote by~$B_{\Psi_0,h_0}(\varepsilon)$ the set of all functions $(\Psi,h)$ that also satisfy~\cref{assumpenum:assumption2,assumpenum:assumption3} and such that $(\Psi,h)\in B_{L^\infty}\bigl((\Psi_0,h_0),\varepsilon \bigr)$.}
    %\ams{If we d.efine the $L^\infty$ ball in the introduction notation section then the previous display is not needed; indeed in the main body of the text you never defined the notation for the $L^\infty$ ball but nor did you state the preceding display.}
    Then the following statements hold:
    %\ams{It would be more succint to express RHS of next two displays using moment notation introduced above. Likewise for $\mathcal R$ here and elsewhere.}
    \begin{itemize}
        \item
            There is a constant $C_p = C_p(\kappa_{\Psi}, \Sigma)$ such that for all $\varepsilon \in [0, 1]$ and all $(\Psi,h) \in B_{L^\infty}\bigl((\Psi_0,h_0),\varepsilon \bigr)$,
            \begin{equation}
                \label{eq:statement_P}
                \forall \mu \in \mathcal P(\real^{\dimu}), \qquad
                d_g(\op P_0 \mu, \op P \mu)
                \leq C_p \varepsilon \cdot
                \bigl( 1+\mathcal{M}_2(\mu) \bigr)%\int_{\real^{\dimu}} \bigl(1 + |v|^2\bigr) \, \mu(\d v).
            \end{equation}
        \item
            There is a constant $C_q = C_q(\kappa_{h}, \Gamma)$ such that for all $\varepsilon \in [0, 1]$ and all $(\Psi,h) \in B_{L^\infty}\bigl((\Psi_0,h_0),\varepsilon \bigr)$,
            \begin{equation}
                \label{eq:statement_Q}
                \forall \mu \in \mathcal P(\real^{\dimu}), \qquad
                d_g(\op Q_0 \mu, \op Q \mu)
                \leq C_q \varepsilon \cdot
                \bigl( 1+\mathcal{M}_2(\mu) \bigr)%\int_{\real^{\dimu}} \bigl(1 + |v|^2\bigr) \, \mu(\d v).
            \end{equation}
        \item
            There is $C_b = C_b(\kappa_{h}, \Gamma)$ such that
            for all $y^\dagger \in \real^{\dimy}$,
            all $\varepsilon \in [0, 1]$,
            all~$(\Psi,h) \in B_{L^\infty}\bigl((\Psi_0,h_0),\varepsilon \bigr)$,
            and all probability measures $(\mu_0, \mu) \in \mathcal P(\real^{\dimu}) \times \mathcal P(\real^{\dimu})$,
            \begin{equation}
                \label{eq:statement_BQ}
                d_g\Bigl(\op B(\op Q_0 \mu_0; y^{\dagger}), \op B (\op Q \mu, y^\dagger)\Bigr)
                \leq \exp(C_b \mathcal R) \Bigl(\varepsilon + d_g(\mu_0, \mu) \Bigr),
            \end{equation}
            where $\mathcal R \in [1, \infty]$ is given by
            \[
                \mathcal R = \max \left\{ \left\lvert y^\dagger \right\rvert^2, 1+  \mathcal M_2(\mu_0), 1 + \mathcal M_2(\mu) \right\}.
            \]
    \end{itemize}
    Here $\op P_0$ and $\op Q_0$ denote the maps associated to~$(\Psi_0, h_0)$,
    and $\op P$ and $\op Q$ are the maps associated to~$(\Psi, h)$.
\end{lemma}
\begin{proof}
    Throughout this proof, $C_p$ denotes a constant depending only on $(\kappa_{\Psi}, \Sigma)$,
    and $C_q$ is a constant that depends only on $(\kappa_{h}, \Gamma)$.
    Both may change from line to line.

    \paragraph{Proof of~\eqref{eq:statement_P}}
    It holds that
    \begin{equation*}
        \op P_0 \mu(u) - \op P \mu(u)
        = C_p \int_{\real^{\dimu}}
            \exp \left( - \frac{1}{2} \bigl\lvert u - \Psi_0(v) \bigr\rvert_{\Sigma}^2 \right)  -
            \exp \left( - \frac{1}{2} \bigl\lvert u - \Psi(v) \bigr\rvert_{\Sigma}^2 \right)
         \, \mu(\d v).
    \end{equation*}
    By the elementary inequality~\eqref{eq:elementary_inequality_gaussian_density},
    the integrand on the right-hand side is bounded in absolute value by
    \begin{align*}
        2\bigl\lvert \Psi_0(v) - \Psi(v) \bigr\rvert
        \left( \exp \left( - \frac{1}{4} \bigl\lvert u - \Psi_0(v) \bigr\rvert_{\Sigma}^2 \right) + \exp \left( - \frac{1}{4} \bigl\lvert u - \Psi(v) \bigr\rvert_{\Sigma}^2 \right) \right).
    \end{align*}
    By Young's inequality, it holds that $|a + b|^2 \geq \frac{1}{2} |a|^2 - |b|^2$,
    and so this is bounded by
    \begin{align*}
        4 \bigl\lvert \Psi_0(v) - \Psi(v) \bigr\rvert
        \left( \exp \left( - \frac{1}{8} \bigl\lvert u - \Psi_0(v) \bigr\rvert_{\Sigma}^2 + \frac{1}{4} \bigl\lvert \Psi_0(v) - \Psi(v) \bigr\rvert_{\Sigma}^2 \right) \right)
        \leq C_p \varepsilon \exp \left( - \frac{1}{8} \bigl\lvert u - \Psi_0(v) \bigr\rvert_{\Sigma}^2 \right).
    \end{align*}
    It follows, by Fubini's theorem, that
    \begin{align*}
        d_g(\op P_0 \mu, \op P \mu)
        &\leq C_p \varepsilon \int_{\real^{\dimu}} \int_{\real^{\dimu}}
        \bigl(1 + |u|^2\bigr)
        \exp \left( - \frac{1}{8} \bigl\lvert u - \Psi_0(v) \bigr\rvert_{\Sigma}^2 \right)
        \d u \, \mu(\d v) \\
        &\leq C_p \varepsilon \int_{\real^{\dimu}} \Bigl(1 + \bigl\lvert \Psi_0(v) \bigr\rvert^2 \Bigr)\, \mu(\d v)
        \leq C_p \varepsilon \, (1 + \kappa_{\Psi}^2) \int_{\real^{\dimu}} \left(1 + |v|^2\right) \, \mu(\d v).
    \end{align*}
    This concludes the proof of~\eqref{eq:statement_P}.

    \paragraph{Proof of~\eqref{eq:statement_Q}}
    Recall that
    \[
        d_g(\op Q_0 \mu, \op Q \mu)
        = C_q \int_{\real^{\dimu}} \int_{\real^{\dimy}}
            g(u, y)
            \left(
            \exp \left( - \frac{1}{2} \bigl\lvert y - h_0(u) \bigr\rvert_{\Gamma}^2 \right)  -
            \exp \left( - \frac{1}{2} \bigl\lvert y - h(u) \bigr\rvert_{\Gamma}^2 \right)
            \right)
            \, \d y \, \mu(\d u),
     \]
     where $g(u, y) = 1 + |u|^2 + |v|^2$.
     Using the same reasoning as above we obtain that
     \begin{equation}
         \label{eq:diff_exp}
         \left\lvert
             \exp \left( - \frac{1}{2} \bigl\lvert y - h_0(u) \bigr\rvert_{\Gamma}^2 \right)  -
             \exp \left( - \frac{1}{2} \bigl\lvert y - h(u) \bigr\rvert_{\Gamma}^2 \right)
         \right\rvert
         \leq C_q \varepsilon \exp \left( - \frac{1}{8} \bigl\lvert y - h_0(u) \bigr\rvert_{\Gamma}^2 \right).
     \end{equation}
     Therefore,
     we deduce that
     \[
        d_g(\op Q_0 \mu, \op Q \mu)
        \leq C_q \varepsilon \int_{\real^{\dimu}} \Bigl(1 + \bigl\lvert h_0(u) \bigr\rvert^2 \Bigr) \, \mu(\d u) \leq C_q \varepsilon (1 + \kappa_{h}^2) \int_{\real^{\dimu}} \bigl(1 + |u|^2\bigr) \, \mu(\d u),
    \]
    which proves~\eqref{eq:statement_Q}.

    \paragraph{Proof of~\eqref{eq:statement_BQ}}
    We assume for simplicity that $\mu_0$ and $\mu$ have densities,
    but note that this is not required.
    Let $\nu_0 = \op B(\op Q_0 \mu_0; y^{\dagger})$ and $\nu = \op B (\op Q \mu, y^\dagger)$ and
    recall that
    \[
        \nu_0(u) =
        \frac{\exp \left( - \frac{1}{2} \bigl\lvert y^\dagger - h_0(u) \bigr\rvert_{\Gamma}^2 \right)  \mu_0(u)}
        {\int_{\real^{\dimu}} \exp \left( - \frac{1}{2} \bigl\lvert y^\dagger - h_0(U) \bigr\rvert_{\Gamma}^2 \right) \mu_0(U) \, \d U}
        =: \frac{f_0(u)}{\int_{\real^{\dimu}} f_0(U) \, \d U}
        =: \frac{f_0(u)}{Z_0}
    \]
    and similarly
    \[
        \nu(u) =
        \frac{\exp \left( - \frac{1}{2} \bigl\lvert y^\dagger - h(u) \bigr\rvert_{\Gamma}^2 \right)  \mu(u)}
        {\int_{\real^{\dimu}} \exp \left( - \frac{1}{2} \bigl\lvert y^\dagger - h(U) \bigr\rvert_{\Gamma}^2 \right) \mu(U) \, \d U}
        =: \frac{f(u)}{\int_{\real^{\dimu}} f(U) \, \d U}
        =: \frac{f_0(u)}{Z}.
    \]
    Note that
    \begin{align*}
        d_g(\nu_0, \nu)
        &= \int_{\real^{\dimu}}
        \bigl(1 + |u|^2\bigr)
        \left\lvert\frac{f_0(u)}{Z_0} - \frac{f(u)}{Z} \right\rvert
        \, \d u \\
        &= \frac{1}{Z} \int_{\real^{\dimu}}
        \bigl(1 + |u|^2\bigr)
        \bigl\lvert f_0(u) - f(u) \bigr\rvert \, \d u
        + \left\lvert \frac{1}{Z_0} - \frac{1}{Z} \right\rvert \int_{\real^{\dimu}} \bigl(1 + |u|^2\bigr) \, f_0(u) \, \d u.
    \end{align*}
    In order to bound the first term,
    we write
    \begin{align}
        \notag
        f_0(u) - f(u)
        &=
        \left( \exp \left( - \frac{1}{2} \bigl\lvert y^\dagger - h_0(u) \bigr\rvert_{\Gamma}^2 \right) - \exp \left( - \frac{1}{2} \bigl\lvert y^\dagger - h(u) \bigr\rvert_{\Gamma}^2 \right) \right) \mu_0(u) \\
        \label{eq:three_terms_close_gaussian}
        &\qquad
        + \exp \left( - \frac{1}{2} \bigl\lvert y^\dagger - h(u) \bigr\rvert_{\Gamma}^2 \right) \Bigl(\mu_0(u) - \mu(u) \Bigr).
    \end{align}
    Using~\eqref{eq:diff_exp},
    we obtain that
    \begin{equation}
        \label{eq:bound_diff_fs}
        \bigl\lvert f_0(u) - f(u) \bigr\rvert
        \leq
        C_q \varepsilon \exp \left( - \frac{1}{8} \bigl\lvert y^\dagger - h_0(u) \bigr\rvert_{\Gamma}^2 \right) \mu_0(u)
        + \exp \left(- \frac{1}{2} \bigl\lvert y^\dagger - h(u) \bigr\rvert_{\Gamma}^2 \right) \Bigl\lvert \mu_0(u) - \mu(u) \Bigr\rvert,
    \end{equation}
    and so
    \begin{align*}
        \int_{\real^{\dimu}}
        \bigl(1 + |u|^2\bigr)
        \bigl\lvert f_0(u) - f(u) \bigr\rvert \, \d u
        &\leq
        C_q \varepsilon \mathcal R
        + d_g(\mu_0, \mu).
    \end{align*}
    Therefore it holds that
    \[
        d_g(\nu_0, \nu) \leq \mathcal R\left( \frac{C_q \varepsilon}{Z} +  \frac{\lvert Z_0 - Z \rvert}{Z_0 Z} \right)
        + \frac{1}{Z} d_g(\mu_0, \mu).
    \]
    By~\eqref{eq:bound_diff_fs} it holds that
    \[
        \lvert Z_0 - Z \rvert
        \leq \int_{\real^{\dimu}} \lvert f_0(U) - f(U) \rvert \, \d U
        \leq C_q \varepsilon + d_g(\mu_0, \mu),
    \]
    and so we obtain finally
    \[
        d_g(\nu_0, \nu) \leq \mathcal R \left( \frac{1}{Z} +  \frac{1}{Z_0 Z} \right)
        \Bigl( C_q \varepsilon + d_g(\mu_0, \mu) \Bigr).
    \]
    Furthermore $Z_0$ is bounded from below because,
    by Jensen's inequality,
    \begin{align*}
        Z_0 &= \int_{\real^{\dimu}} \exp \left( - \frac{1}{2} \bigl\lvert y^\dagger - h_0(U) \bigr\rvert_{\Gamma}^2 \right) \mu_0(U) \, \d U
        \geq \exp \left( - \frac{1}{2} \int_{\real^{\dimu}} \bigl\lvert y^\dagger - h_0(U) \bigr\rvert_{\Gamma}^2 \, \mu_0(U) \, \d U \right) \\
        &\geq \exp \left(- \bigl\lvert y^\dagger \bigr\rvert_{\Gamma}^2 - C_q \kappa_h^2 \int_{\real^{\dimu}} \bigl(1 + |u|^2\bigr) \mu_0(U) \, \d U \right)
        = \exp \left( - C_q \mathcal R \right).
    \end{align*}
    The same bound holds for~$Z$,
    and so we obtain the result.
\end{proof}

\begin{lemma}
\label{lemma:auxiliary_close_to_affine_new}
Suppose that the data~$Y^\dagger =\{y_j^\dagger\}_{j=1}^J$ and the matrices~ $(\mat \Sigma, \mat \Gamma)$ satisfy~\cref{assumpenum:assumption1,assumpenum:assumption5}.
    Fix $\kappa_{\Psi}, \kappa_h > 0$ and assume that $\Psi_0 \colon \real^{\dimu} \to \real^{\dimu}$ and $\vect h_0 \colon \real^{\dimu} \to \real^{\dimy}$ are functions satisfying~\cref{assumpenum:assumption2,assumpenum:assumption3}, respectively.
    Let $(\mu^0_{j})_{j \in \range{1}{J}}$ and $(\mu_{j})_{j \in \range{1}{J}}$ denote the true filtering distributions associated with
    functions~$(\Psi_0, \vect h_0)$ and~$(\Psi, \vect h)$, respectively, initialized at the same Gaussian probability measure~$\mu_0 = \normal(m_0, C_0) \in~\mathcal G(\real^{\dimu})$.
    For all $J \in \mathbb{Z}^+$ there is $C = C(m_0,C_0,\kappa_y, \kappa_{\Psi}, \kappa_h, \mat \Sigma, \mat \Gamma, J) > 0$ such that for all $\varepsilon \in [0, 1]$
    and all~$(\Psi, h) \in B_{L^\infty}\bigl((\Psi_0,h_0),\varepsilon \bigr)$,
    it holds that
    \begin{equation}
        \label{eq:first_equation_corollary_new}
        \max_{j \in \range{0}{J}}
        d_g(\mu^0_j, \mu_j)
        \leq C \varepsilon.
    \end{equation}
    \end{lemma}
\begin{proof}
    By~\cref{lemma:true_filter_bounded_moments},
    the filtering distributions have bounded second moments.
    Let
    \[
        \mathcal R = \max_{j \in \range{0}{J-1}}
        \left( \left\lvert y_{j+1}^\dagger \right\rvert^2 , \,1 + \mathcal{M}_2\bigl(\mu^0_j\bigr),\, 1+ \mathcal{M}_2\bigl(\mu_j\bigr) \right).
    \]
    Throughout this proof, $C$ denotes a constant whose value is irrelevant in the context,
    depends only on the constants $m_0,C_0,\kappa_y, \kappa_{\Psi}, \kappa_h, \mat \Sigma, \mat \Gamma, k$ (but neither on $\varepsilon$, nor on $\Psi$ and $h$)
    and may change from line to line.

    The statement is obviously true for $J = 0$.
    Reasoning by induction, we assume that the statement is true up to $J = k$
    and show that there is $C > 0$ such that
    \[
        \forall \varepsilon \in [0, 1], \qquad
        \forall (\Psi,h) \in B_{L^{\infty}}\bigl((\Psi_0,h_0), \varepsilon\bigr), \qquad
        d_g(\mu^0_{k+1}, \mu_{k+1}) \leq C \varepsilon.
    \]
    To this end,
    let $\op P_0$ and $\op Q_0$ denote the maps associated to~$(\Psi_0, h_0)$,
    fix $\varepsilon \in [0, 1]$,
    and fix~$(\Psi, h) \in B_{L^{\infty}}\bigl((\Psi_0,h_0), \varepsilon\bigr)$.
    Using~\eqref{eq:statement_BQ}, then the triangle inequality,
    and finally~\eqref{eq:statement_P} and~\cref{lemma:lipschitz_p},
    we have that
    \begin{align}
        \notag
        d_g(\mu^0_{k+1}, \mu_{k+1}) & =
        d_g\bigl(\op B_{k} \op Q_0 \op P_0 \mu^0_{k}, \op B_{k} \op Q \op P \mu_{k}\bigr)\\
        \notag
        &\leq \e^{C \mathcal R} \Bigl( \varepsilon + d_g\bigl(\op P_0 \mu^0_{k}, \op P \mu_{k}\bigr) \Bigr) \\
        \notag
        % \label{eq:bound_term_P_new}
        &\leq \e^{C \mathcal R} \Bigl( \varepsilon + d_g\bigl(\op P_0 \mu^0_{k}, \op P \mu^0_{k}\bigr) + d_g\bigl(\op P \mu^0_{k}, \op P \mu_{k}\bigr) \Bigr) \\
        \notag
        &\leq \e^{C \mathcal R} \Bigl( \varepsilon + C \varepsilon \mathcal R + C d_g(\mu^0_k, \mu_k) \Bigr).
    \end{align}
    Using the induction hypothesis gives us the desired bound.
\end{proof}

\section{Technical Result for Theorem \ref{theorem:enkf_theorem}}
\label{appendix:D}
In \cite{le2009large} machinery is established to prove Monte Carlo error estimates between
the finite particle ensemble Kalman filter and its mean field limit. We use such results
as a component in proving~\cref{theorem:enkf_theorem} and, in so-doing, explicit
dependence on moments must be tracked. In this section
we give a self-contained proof of~\cite[Theorem 5.2]{le2009large},
following the analysis closely\footnote{The result of~\cite{le2009large} holds more generally for functions $\vect \Psi$ that are locally Lipschitz and that grow at most polynomially at infinity~\cite[Assumption B]{le2009large}. However, our proof uses that $\Psi$ is globally Lipschitz, for instance in \eqref{eq:lip_meanfield1}.} and, in addition, tracking dependence on moments; this moment dependence
may be useful in the context of future work generalizing what we do in this paper.
This leads to the following error estimate stating the desired Monte Carlo error estimate.

\begin{lemma}
    \label{lemma:convergence_to_MF}
    Assume that the probability measures  $(\mu^{\kalman}_j)_{j \in \range{0}{J}}$ and $(\mu^{\kalmanN}_j)_{j \in \range{0}{J}}$ are obtained, respectively, from the dynamical systems~\eqref{eq:compact_mf_enkf} and ~\eqref{eq:kalman_measure},
    initialized at the Gaussian probability measure $\mu_0^{\kalman}\in~\mathcal G(\real^{\dimu})$ and at the empirical measure $\mu_{0}^{\kalmanN} =\frac1N\sum_{i=1}^N\delta_{u_0^{(i)}}$ for $u_0^{(i)}\sim \mu^{\kalman}_0$ i.i.d. samples.
    That~is,
    \[
        \mu^{\kalman}_{j+1} = \op T_j \op Q \op P \mu_j^{\kalman}, \qquad \mu^{\kalmanN}_{j+1}  = \frac1N\sum_{i=1}^N\delta_{u_{j+1}^{(i)}},
    \]
    where $u_{j+1}^{(i)}$ evolve according to the iteration in~\eqref{eq:ensemble_kalman_particle}.
    Suppose that the data~$Y^\dagger =\{y_j^\dagger\}_{j=1}^J$ and the matrices~$(\mat \Sigma, \mat \Gamma)$ satisfy~\cref{assumpenum:assumption1,assumpenum:assumption5}.
    We assume that the vector field $\vect \Psi$ satisfies~\cref{assumpenum:assumption2} and that $\vect h$ is a linear transformation,
    i.e.\ that \cref{assumpenum:affine_assumption2} is satisfied and hence also~\cref{assumpenum:assumption3}.
    Furthermore, if the vector field $\vect \Psi$ additionally satisfies $|\Psi|_{C^{0,1}} \leq \ell_{\vect \Psi} < \infty$,
    then for all $\phi$ satisfying~\cref{assumption:metric_vector_fields},
    there exists a constant~$C = C\bigl(\mathcal{M}_{q}(\mu^{\kalman}_0),J,R_\phi,L_\phi,\varsigma,\kappa_y, \kappa_{\vect \Psi}, \kappa_{\vect h}, \ell_{\vect \Psi}, \mat \Sigma, \mat \Gamma\bigr)$ where $q\coloneqq \max \{4^{J+1}, 4\varsigma, 2(\varsigma+1) \}$
    such that
    \[
        \left(\mathbb{E}~\Bigl|\mu^{\kalmanN}_J[\phi] -\mu^{\kalman}_J[\phi]\Bigr|^2\right)^{\frac{1}{2}} \leq \frac{C}{\sqrt{N}}.
    \]
\end{lemma}

\begin{proof}
    To prove the proposition, we apply the coupling argument used in \cite{le2009large}. Using similar notation to the one in \cite{le2009large}, to the interacting $N$-particle system $\bigl\{u^{(i)}_j\bigr\}_{n=1}^N$ evolving according to the ensemble Kalman dynamics~\eqref{eq:ensemble_kalman_particle},
    we couple $N$ copies of the mean field dynamics $\bigl\{\overline u^{(j)}_j\bigr\}_{n=1}^N$ evolving according to the mean field ensemble Kalman dynamics~\eqref{eq:ensemble_kalman_mean_field}.
    The mean field replicas are synchronously coupled to the interacting particle system,
    in the sense that they are initialized at the same initial condition and driven by the same noises; namely, the two particle systems are initialized at i.i.d. samples $u_0^{(i)} = \overline{u}_0^{(i)} \sim \mu_0^{\kalman}$ for $i=1,\ldots,N$.
    For simplicity of notation, the forecast particles are denoted by the letter~$v$,
    and we drop the hat notation from the forecast and simulated observations.
    Furthermore, we add a bar $\overline \placeholder$ to all the variables related to the synchronously coupled mean field particles,
    including the probability measures $\mfl_j$.
    We also define for~$\pi \in \mathcal P(\real^{\dimu \times \dimy})$ the Kalman gain
    \[
        \mathcal K(\pi) := \cov^{uh}\left(\pi\right) \left( \cov^{hh}\left(\pi\right) + \Gamma \right)^{-1}.
    \]
    With this notation, the interacting particle system,
    and synchronously coupled system read as follows:
    \[
    \begin{minipage}{.45\linewidth}
        {\large \emph{Interacting particle system}}
    \begin{subequations}
        \begin{alignat*}{2}
            \notag
            &\text{\bf Initialization: } u^{(i)}_0 = \overline{u}^{(i)}_0 \\
            &v^{(i)}_{j+1} = \Psi \Bigl(u^{(i)}_j\Bigr) + \xi^{(i)}_j, \\
            &y^{(i)}_{j+1} = h\Bigl(v^{(i)}_{j+1}\Bigr) + \eta^{(i)}_{j+1}, \\
            &u^{(i)}_{j+1} = v^{(i)}_{j+1} + \mathcal K\left(\emp_{j+1}\right) \left(y^\dagger_{j+1} - y^{(i)}_{j+1} \right).
        \end{alignat*}
    \end{subequations}
    \end{minipage}
    \begin{minipage}{.5\linewidth}
    \end{minipage}
        \begin{minipage}{.45\linewidth}
            {\large \emph{Synchronous coupling}}
            \begin{subequations}
                \begin{alignat*}{2}
                    \notag
            &\text{\bf Initialization: } \overline{u}^{(i)}_0 = u^{(i)}_0 \\
            &\overline{v}^{(i)}_{j+1} = \Psi \Bigl(\overline{u}^{(i)}_j\Bigr) + \xi^{(i)}_j, \\
            &\overline{y}^{(i)}_{j+1} = h\Bigl(\overline{v}^{(i)}_{j+1}\Bigr) + \eta^{(i)}_{j+1}, \\
            &\overline{u}^{(i)}_{j+1} = \overline{v}^{(i)}_{j+1} + \mathcal K\left(\mfl_{j+1}\right) \left(y^\dagger_{j+1} - \overline{y}^{(i)}_{j+1} \right).
                \end{alignat*}
            \end{subequations}
        \end{minipage}
\]
    \vspace{.3cm}

    \noindent Note that the synchronously coupled particles are independent and identically distributed. With this set-up we proceed with the proof. Let $\phi$ satisfy~\cref{assumption:metric_vector_fields}.
    By applying the triangle inequality, we deduce that
    \begin{align}
        \notag
        \left(\mathbb{E}~\Bigl|\mu^{\kalmanN}_J[\phi] -\mu^{\kalman}_J[\phi]\Bigr|^2\right)^{\frac{1}{2}}
        &\leq \left(\mathbb{E}~\Bigl|\mu^{\kalmanN}_J[\phi] -\frac1N\sum_{i=1}^N\phi\bigl(\overline{u}^{(i)}_J \bigr)\Bigr|^2\right)^{\frac{1}{2}} \\
        \label{eq:bound_on_D_01}
        &\qquad + \left(\mathbb{E}~\Bigl|\frac1N\sum_{i=1}^N\phi\bigl(\overline{u}^{(i)}_J\bigr) -\mu^{\kalman}_J[\phi]\Bigr|^2\right)^{\frac{1}{2}}.
    \end{align}

    \paragraph{Step 1: bounding the second term in~\eqref{eq:bound_on_D_01}}
        Noting that the random variables $\phi\bigl(\overline{u}^{(i)}_J \bigr) - \mu_J^{\kalman}[\phi]$ are i.i.d.\ and expanding the square,
        we obtain
        \begin{subequations}
            \begin{align*}
                \left(\mathbb{E}~\biggl\lvert\sum_{i=1}^N\frac1N\Bigl(\phi\bigl(\overline{u}^{(i)}_J\bigr) -\mu^{\kalman}_J[\phi]\Bigr)\biggr\rvert^2\right)^{\frac{1}{2}}
                = \frac{1}{\sqrt{N}}\Bigl( \mathbb{E}\bigl|\phi\bigl(\overline{u}^{(1)}_J \bigr) - \mu_J^{\kalman}[\phi] \bigr|^2\Bigr)^{\frac{1}{2}},
            \end{align*}
        \end{subequations}
    Since by \cref{assumption:metric_vector_fields} it holds that $|\phi(u)|\leq R_\phi\bigl(1+|u|^{{\varsigma}+1} \bigr)$ for any $u\in\real^{\dimu}$, it follows that
    \begin{equation}
        \label{eq:bound_on_D_11}
        \left(\mathbb{E}~\Bigl|\frac1N\sum_{i=1}^N\phi\bigl(\overline{u}^{(i)}_J\bigr) -\mu^{\kalman}_J[\phi]\Bigr|^2\right)^{\frac{1}{2}} \leq \frac{C\Bigl(\mathcal{M}_{2(\varsigma+1)}\bigl(\mu_0^{\kalman}\bigr),R_\phi,J,\kappa_y, \kappa_{\vect \Psi}, \kappa_{\vect h}, \mat \Sigma, \mat \Gamma\Bigr)}{\sqrt{N}},
    \end{equation}
    where we used~\cref{theorem:approximate_filter_bounded_moments} to bound $\mathcal{M}_{2(\varsigma+1)}\bigl(\mu_J^{\kalman}\bigr)$ in terms of $\mathcal{M}_{2(\varsigma+1)}\bigl(\mu_0^{\kalman}\bigr)$.
    %\textcolor{red}{The constant $C$ should inherit the dependence from~\cref{theorem:approximate_filter_bounded_moments} then, no? [remove this comment later]}
    \paragraph{Step 2: bounding the first term in~\eqref{eq:bound_on_D_01}}
    For the first term,
    by Jensen's inequality, exchangeability,
    and finally by \cref{assumption:metric_vector_fields} on~$\phi$ together with the Cauchy--Schwarz inequality,
    it holds without loss of generality that
    \begin{align}
    \notag
        \left(\expect~\Bigl|\frac1N\sum_{i=1}^N \phi\Bigl(u^{(i)}_J \Bigr)-\phi\Bigl(\overline{u}^{(i)}_J \Bigr)\Bigr|^2\right)^{\frac{1}{2}}
        &\leq
         \left(\frac1N\sum_{i=1}^N \expect\Bigl| \phi\Bigl(u^{(i)}_J \Bigr)-\phi\Bigl(\overline{u}^{(i)}_J \Bigr)\Bigr|^2 \right) ^{\frac{1}{2}} \notag \\
        &= \notag \left(\expect\Bigl| \phi\Bigl(u^{(1)}_J \Bigr)-\phi\Bigl(\overline{u}^{(1)}_J \Bigr)\Bigr|^2 \right)^{\frac{1}{2}} \\
        &\label{eq:first_term}\leq 3L_{\phi} \left(\expect \left\lvert u^{(1)}_J - \overline{u}^{(1)}_J \right\rvert^4 \right)^{\frac{1}{4}}
        \left( \expect \left[1 + \left\lvert u^{(1)}_J \right\rvert^{4\varsigma} + \left\lvert \overline{u}^{(1)}_J \right\rvert^{4\varsigma} \right] \right)^{\frac{1}{4}}.
    \end{align}
    By~\cref{theorem:approximate_filter_bounded_moments}, there exists $C=C\Bigl(\mathcal{M}_{4\varsigma}\bigl(\mu_0^{\kalman}\bigr),J,\kappa_y, \kappa_{\vect \Psi}, \kappa_{\vect h}, \mat \Sigma, \mat \Gamma\Bigr)$ so that
    \[
        \mathbb{E} \left\lvert \overline{u}^{(1)}_J \right\rvert ^{4\varsigma}
        \leq C\quad\text{and}\quad
        \mathbb{E}\left\lvert u^{(1)}_J \right\rvert^{4\varsigma} \leq
            2^{4\varsigma -1} \left( \mathbb{E} \left\lvert \overline{u}^{(1)}_J \right\rvert ^{4\varsigma}
            + \mathbb{E} \left\lvert u^{(1)}_J - \overline{u}^{(1)}_J  \right\rvert ^{4\varsigma}  \right)
        \leq
            2^{4\varsigma -1} \left( C + \mathbb{E} \left\lvert u^{(1)}_J - \overline{u}^{(1)}_J  \right\rvert ^{4\varsigma}  \right).
    \]

    %\textcolor{red}{%
    %    The second bound requires a moment bound for the finite particle ensemble Kalman filter,   which is not provided by~\cref{theorem:approximate_filter_bounded_moments},but we can get a bound using propagation of chaos.[Please remove this if you agree]}

    Hence it remains to bound terms of the form $\expect \bigl\lvert u^{(1)}_J - \overline{u}^{(1)}_J \bigr\rvert^{p}$
    with $p = 4$ and $p = 4\varsigma$.
    Such a bound will follow from a \textit{propagation of chaos} result; which will be the object of Step 4. In Step 3 we show auxiliary moment bounds for the mean field dynamics.

    \paragraph{Step 3: moment bounds for the mean field dynamics}%
    \label{par:Moment bound1}
    By~\cref{theorem:approximate_filter_bounded_moments}, for any $j\in \range{0}{J-1}$ there exists a constant $C=C\Bigl(\mathcal{M}_{p}\bigl(\mu_0^{\kalman}\bigr),J,\kappa_y, \kappa_{\vect \Psi}, \kappa_{\vect h}, \mat \Sigma, \mat \Gamma\Bigr)$ so that
    \begin{equation}
        \label{eq:moment_mf1}
        \left( \expect \left\lvert \overline{u}_j^{(1)} \right\rvert^p \right)^{\frac{1}{p}} \leq C.
    \end{equation}
    From this it immediately follows, by~\cref{assumpenum:assumption2,assumpenum:assumption3}, that there exist constants $C_v,C_y$ depending on $\mathcal{M}_{p}\bigl(\mu_0^{\kalman}\bigr),J,\kappa_y, \kappa_{\vect \Psi}, \kappa_{\vect h}, \mat \Sigma, \mat \Gamma$ so that
    \begin{equation}
        \label{eq:moment_vy1}
        \left( \expect \left\lvert \overline{v}^{(1)}_{j+1} \right\rvert^p \right)^{\frac{1}{p}} \leq C_v,
        \qquad
        \left( \expect \left\lvert \overline{y}^{(1)}_{j+1} \right\rvert^p \right)^{\frac{1}{p}} \leq C_y.
    \end{equation}
    Furthermore, it holds that for the Kalman gain that
    \begin{equation}
        \Bigl\lVert \mathcal K\bigl(\mfl_{j+1}\bigr) \Bigr\rVert^p
        \leq
         \left\lVert \mathcal C^{uh}\bigl(\mfl_{j+1}\bigr) \right\rVert^p \left\lVert \Gamma^{-1} \right\rVert^p
         =
        \left\lVert \Gamma^{-1} \right\rVert^p
        \left\lVert \mathcal C^{uu}\bigl(\mfl_{j+1}\bigr) H^\t \right\rVert^p \leq C,
        \label{eq:moment_gain1}
    \end{equation}
    where $C$ depends on $\mathcal{M}_{2}\bigl(\mu_0^{\kalman}\bigr),J,\kappa_y, \kappa_{\vect \Psi}, \kappa_{\vect h}, \mat \Sigma, \mat \Gamma$.

    \paragraph{Step 4: propagation of chaos}%
    \label{par:Step 4. Propagation of chaos}

    In this step,
    we prove a propagation of chaos result stating for all $p$,
    there exists a constant $C_p = C_p\Bigl(\mathcal{M}_{4^J\cdot p}\bigl(\mu_0^{\kalman}\bigr),J,\kappa_y, \kappa_{\Psi}, \kappa_h, \ell_{\Psi}, \ell_h, \Sigma, \Gamma\Bigr)$ such that
    \begin{equation}
        \label{eq:propoagation_of_chaos}
        \sup_{N \in \nat} \sqrt{N}\left( \expect \left\lvert u^{(1)}_J - \overline{u}^{(1)}_J \right\rvert^p \right)^{\frac{1}{p}}
        \leq C_{p}.
    \end{equation}
    To prove this result,
    we follow the strategy in~\cite{le2009large} and reason by induction.
    Fix $j \in \range{0}{J-1}$ and assume that for all $p \in \nat$
    there is $C_{j,p} =C_{j,p}\Bigl(\mathcal{M}_{4^j\cdot p}\bigl(\mu_0^{\kalman}\bigr),J,\kappa_y, \kappa_{\Psi}, \kappa_h, \ell_{\Psi}, \ell_h, \Sigma, \Gamma\Bigr)$ such that
    \begin{equation}
        \label{eq:induction_chaos}
        \sup_{N \in \nat} \sqrt{N}\left( \expect \left\lvert u^{(1)}_j - \overline{u}^{(1)}_j \right\rvert^p \right)^{\frac{1}{p}}
        \leq C_{j,p}.
    \end{equation}
    By Lipschitz continuity of $\Psi$ and $h$,
    we deduce immediately that
    \begin{equation}
        \label{eq:lip_meanfield1}
        \sup_{N \in \nat} \sqrt{N}\left( \expect \left\lvert v^{(1)}_{j+1} - \overline{v}^{(1)}_{j+1} \right\rvert^p \right)^{\frac{1}{p}}
        \leq \ell_{\Psi} C_{j,p},
        \qquad
        \sup_{N \in \nat} \sqrt{N}\left( \expect \left\lvert y^{(1)}_{j+1} - \overline{y}^{(1)}_{j+1} \right\rvert^p \right)^{\frac{1}{p}}
        \leq \ell_{h} \ell_{\Psi} C_{j,p}.
    \end{equation}
    Now we write
    \begin{align*}
        u^{(1)}_{j+1} - \overline{u}^{(1)}_{j+1}
        &=
        \left(v^{(1)}_{j+1} - \overline{v}^{(1)}_{j+1}\right)
        + \mathcal K\bigl(\mfl_{j+1}\bigr) \left(\overline{y}^{(1)}_{j+1} - y^{(1)}_{j+1}\right) \\
        &\qquad + \Bigl( \mathcal K\bigl(\emp_{j+1}\bigr) -  \mathcal K\bigl(\mfl_{j+1}\bigr) \Bigr)\left(y^{\dagger}_{j+1} - y^{(1)}_{j+1}\right).
    \end{align*}
    Thus, by the triangle and H\"older inequalities,
    \begin{align*}
        \sqrt{N} \left( \expect \left\lvert u^{(1)}_{j+1} - \overline{u}^{(1)}_{j+1} \right\rvert^p \right)^{\frac{1}{p}}
        &\leq
        \ell_{\Psi} C_{j,p}
        + \left\lVert \mathcal K\bigl(\mfl_{j+1}\bigr) \right\rVert  \ell_{h} \ell_{\Psi} C_{j,p}
         \\
        &\qquad + \sqrt{N} \left( \expect \left\lVert  \mathcal K\bigl(\emp_{j+1}\bigr) -  \mathcal K\bigl(\mfl_{j+1}\bigr) \right\rVert^{2p} \right)^{\frac{1}{2p}} \left(\expect \left\lvert y^{\dagger}_{j+1} - y^{(1)}_{j+1} \right\rvert^{2p}\right)^{\frac{1}{2p}}.
    \end{align*}
    The first two terms on the right-hand side are bounded uniformly in~$N$ in view of~\eqref{eq:moment_gain1}.
    In order to bound the last term,
    we note that the term
    $
    \expect \bigl\lvert y^{\dagger}_{j+1} - y^{(1)}_{j+1} \bigr\rvert^{2p}
    $
    is bounded uniformly in~$N$ by a constant which depends on $\mathcal{M}_{4^j\cdot 2p}\bigl(\mu_0^{\kalman}\bigr)$; this follows by~\eqref{eq:moment_vy1} and~\eqref{eq:lip_meanfield1}.
    To complete the inductive step, it remains to show that
    \[
        \sup_{N \in \nat}
        \sqrt{N} \left( \expect \left\lVert  \mathcal K\Bigl(\emp_{j+1}\Bigr) -  \mathcal K\Bigl(\mfl_{j+1}\Bigr) \right\rVert^{2p} \right)^{\frac{1}{2p}}
        < \infty.
    \]
    To this end,
    in line with the classical propagation of chaos approach,
    we decompose
    \begin{align}
        \notag
        \sqrt{N}\left( \expect \left\lVert  \mathcal K\Bigl(\emp_{j+1}\Bigr) -  \mathcal K\Bigl(\mfl_{j+1}\Bigr) \right\rVert^{2p} \right)^{\frac{1}{2p}}
        &\leq
        \sqrt{N}\left( \expect \left\lVert  \mathcal K\Bigl(\emp_{j+1}\Bigr) -  \mathcal K\Bigl(\empmfl_{j+1}\Bigr) \right\rVert^{2p} \right)^{\frac{1}{2p}} \\
        &\qquad +
        \sqrt{N}\left( \expect \left\lVert  \mathcal K\Bigl(\empmfl_{j+1}\Bigr) -  \mathcal K\Bigl(\mfl_{j+1}\Bigr) \right\rVert^{2p} \right)^{\frac{1}{2p}},
        \label{eq:two_terms_gain}
    \end{align}
    where we introduced the empirical measure associated with the mean field particles:
    \[
        \empmfl_{j+1} :=
        \frac{1}{N} \sum_{i=1}^{N} \delta_{\left(\overline{v}^{(i)}_{j+1}, \overline{y}^{(i)}_{j+1}\right)}.
    \]
    To conclude the proof of propagation of chaos,
    we must prove that
    \begin{itemize}
        \item
            The first term on the right-hand side of~\eqref{eq:two_terms_gain} is bounded uniformly in~$N$,
            which is achieved by employing the induction hypothesis~\eqref{eq:induction_chaos},
            as well as~\eqref{eq:lip_meanfield1}.
        \item
            The second term on the right-hand side of~\eqref{eq:two_terms_gain} is also bounded uniformly in~$N$.
            This follows from classical arguments based on the law of large number in~$L^p$ spaces.
    \end{itemize}
Let us first bound the second term.
    To this end,
    note that
    \begin{align*}
        \mathcal K\Bigl(\empmfl_{j+1}\Bigr) -  \mathcal K\Bigl(\mfl_{j+1}\Bigr)
        &= \cov^{uh} \left(\mfl_{j+1}\right)
        \left( \left( \Gamma + \cov^{hh}\Bigl(\empmfl_{j+1}\Bigr) \right)^{-1} - \left( \Gamma + \cov^{hh}\Bigl(\mfl_{j+1}\Bigr) \right)^{-1} \right) \\
        &\qquad + \left( \cov^{uh} \Bigl(\empmfl_{j+1}\Bigr) - \cov^{uh} \Bigl(\mfl_{j+1}\Bigr) \right) \left( \Gamma + \cov^{hh}\left(\empmfl_{j+1}\right) \right)^{-1}.
    \end{align*}
    Therefore, using that $A^{-1} - B^{-1} = B^{-1} (B - A) A^{-1}$,
    we have that
    \begin{align*}
        \left\lVert \mathcal K\Bigl(\empmfl_{j+1}\Bigr) -  \mathcal K\Bigl(\mfl_{j+1}\Bigr) \right\rVert
        &\leq \left\lVert \cov^{uh} \Bigl(\mfl_{j+1}\Bigr) \right\rVert
        \lVert \Gamma^{-1} \rVert^2 \left\lVert \cov^{hh}\Bigl(\empmfl_{j+1}\Bigr) - \cov^{hh}\Bigl(\mfl_{j+1}\Bigr) \right\rVert \\
        &\qquad + \lVert \Gamma^{-1} \rVert \left\lVert  \cov^{uh} \Bigl(\empmfl_{j+1}\Bigr) - \cov^{uh} \Bigl(\mfl_{j+1}\Bigr) \right \rVert \\
        &\leq C(\Gamma, H) \Bigl( 1 + \left\lVert \cov^{uu} \Bigl(\mfl_{j+1}\Bigr) \right\rVert  \Bigr) \left\lVert  \cov^{uu} \Bigl(\emp_{j+1}\Bigr) - \cov^{uu} \Bigl(\mfl_{j+1}\Bigr) \right \rVert.
    \end{align*}
    The term $\bigl\lVert \cov^{uu} \bigl(\mfl_{j+1}\bigr) \bigr\rVert$ is bounded by~\eqref{eq:moment_vy1},
    and the $L^{p}$ norm of the second term tends to zero with rate~$N^{-\frac{p}{2}}$ by classical law of large number arguments,
    see for example~\cite[Lemma 3]{MR4234152} and~\cite[Lemma 3]{vaes2024}; indeed it holds that %include moment dependence from Lemma3
    \[\expect\left\lVert  \cov^{uu} \Bigl(\emp_{j+1}\Bigr) - \cov^{uu} \Bigl(\mfl_{j+1}\Bigr) \right \rVert^{2p} \leq \frac{C\Bigl(\mathcal{M}_{4p}\bigl(\mu^{\kalman}_0\bigr)\Bigr)}{N^p}.\]
    It remains to bound the other term.
    Reasoning similarly to above, we have that
    \begin{equation}
        \left\lVert \mathcal K\Bigl(\emp_{j+1}\Bigr) -  \mathcal K\Bigl(\empmfl_{j+1}\Bigr)  \right\rVert
        \leq
        C(\Gamma, H) \Bigl( 1 + \left\lVert \cov^{uu} \Bigl(\empmfl_{j+1}\Bigr) \right\rVert  \Bigr) \left\lVert  \cov^{uu} \Bigl(\emp_{j+1}\Bigr) - \cov^{uu} \Bigl(\empmfl_{j+1}\Bigr) \right \rVert.
    \end{equation}
    By~\cite[Lemma 2]{vaes2024},
    it holds that
    \[
        \left\lVert  \cov^{uu} \Bigl(\emp_{j+1}\Bigr) - \cov^{uu} \Bigl(\empmfl_{j+1}\Bigr) \right \rVert^2
        \leq 2\left( \frac{1}{N} \sum_{i=1}^{N} \left\lvert v^{(i)}_{j+1} - \overline{v}^{(i)}_{j+1} \right\rvert^2 \right)
        \left( \frac{1}{N} \sum_{i=1}^{N} \left\lvert v^{(i)}_{j+1} \right\rvert^2 + \frac{1}{N} \sum_{i=1}^{N} \left\lvert \overline{v}^{(i)}_{j+1} \right\rvert^2 \right).
     \]
    Using the Cauchy-Schwarz and Jensen's inequalities, moment bounds and finally the induction hypothesis~\eqref{eq:propoagation_of_chaos} enables to conclude that
    \begin{equation*}
        \left(\expect \left\lVert  \cov^{uu} \Bigl(\emp_{j+1}\Bigr) - \cov^{uu} \Bigl(\empmfl_{j+1}\Bigr) \right \rVert^{2p}\right)^{\frac{1}{2p}} \leq \frac{C\Bigl(\mathcal{M}_{4^{j+1}p}\bigl(\mu^{\kalman}_0\bigr),J,\kappa_y, \kappa_{\Psi}, \kappa_h, \ell_{\Psi}, \ell_h, \Sigma, \Gamma\Bigr)}{\sqrt{N}}.
    \end{equation*}

    \paragraph{Step 5: concluding the proof}
    Using the propagation of chaos result from~\eqref{eq:propoagation_of_chaos} in~\eqref{eq:first_term} gives that
    \[
        \left(\expect~\Bigl|\frac1N\sum_{i=1}^N \phi\left(u^{(i)}_J \right)-\phi\left(\overline{u}^{(i)}_J \right)\Bigr|^2\right)^{\frac{1}{2}} \leq \frac{C}{\sqrt{N}},
    \]
    for some constant $C$ depending on $\mathcal{M}_{q}\bigl(\mu^{\kalman}_0\bigr),J,L_\phi,\varsigma,\kappa_y, \kappa_{\Psi}, \kappa_h, \ell_{\Psi}, \ell_h, \Sigma, \Gamma$ with $q=\max \{4^{J+1}, 4\varsigma \}$. Combining this result with \eqref{eq:bound_on_D_11} yields the desired result.
   \end{proof}

\end{document}

%% file: preamble_main.tex
\newif\ifsiamart
\makeatletter%
\@ifclassloaded{siamart220329}%
  {\siamarttrue}%
  {\siamartfalse}%
\makeatother%

\ifsiamart\else
\usepackage{authblk}
\fi

\usepackage{silence}
\usepackage[utf8]{inputenc}
\usepackage[margin=0.85in]{geometry}
\usepackage{xcolor}
\usepackage{graphicx}
\usepackage{bm}
\usepackage{bbm}
\usepackage{amsmath}
\usepackage{amssymb}
\usepackage{stmaryrd}
\usepackage{amsfonts}
\usepackage{mathrsfs}
\usepackage{mathtools}
\usepackage{xparse}
\usepackage{epstopdf}
\usepackage{enumitem}
\usepackage{comment}

\ifsiamart\else
\usepackage{amsthm}
\usepackage{hyperref}
\hypersetup{%
    linktocpage=true,
    colorlinks=true,
    citecolor=red,
    linkcolor=blue,
    urlcolor=darkgreen,
}
\usepackage[nameinlink,capitalise]{cleveref}
\newcommand{\email}[1]{\href{mailto:#1}{#1}}
\newenvironment{keywords}{\small \noindent\begin{quote}\textbf{Keywords.}}{\end{quote}}
\newenvironment{AMS}{\vspace{.2cm}\small \noindent \begin{quote}\textbf{AMS subject classification.}}{\end{quote}}
\fi

\definecolor{darkred}{rgb}{.7,0,0}
\definecolor{darkgreen}{rgb}{.1,.7,0}

\usepackage{setspace}

\ifsiamart\else
\fi

\DeclarePairedDelimiter\matnorm{\lVert}{\rVert}
\DeclarePairedDelimiter\vecnorm{\lvert}{\rvert}
\DeclarePairedDelimiter\abs{\lvert}{\rvert}
\DeclarePairedDelimiter\norm{\lVert}{\rVert}

\DeclareMathOperator{\e}{e}

\DeclareMathOperator*{\trace}{tr}

\DeclareMathOperator{\cov}{\mathcal C}

\newcommand{\kalman}{{\rm EK}}
\newcommand{\kalmanN}{{\rm EK}, N}

\newcommand{\dimu}{d_u}
\newcommand{\dimy}{d_y}
\renewcommand{\t}{\mathsf T}

\newcommand{\kl}[2]{\mathrm {KL} (#1 \| #2)}

\newcommand{\placeholder}{\mathord{\color{black!33}\bullet}}%
\newcommand{\ind}{\perp\!\!\!\!\perp}
\newcommand{\expect}{\mathbb{E}}

\newcommand{\range}[2]{\llbracket #1, #2 \rrbracket}
\newcommand{\mat}{}
\newcommand{\op}{\mathsf}
\newcommand{\normal}{\mathcal{N}}
\newcommand{\nat}{\mathbb N}

\newcommand{\real}{\mathbb R}

\newcommand{\vect}[1]{#1}

\renewcommand{\d}{\mathrm d}

\renewcommand{\leq}{\leqslant}
\renewcommand{\geq}{\geqslant}
\renewcommand{\le}{\leqslant}
\renewcommand{\ge}{\geqslant}

\newcommand{\CC}{\mathcal{C}}

\newcommand{\yd}{{y}^\dag}

\ifsiamart
\newtheorem{remark}[theorem]{Remark}
\newtheorem{assumption}{Assumption}
\renewcommand{\theassumption}{A}
\else
\theoremstyle{plain}

\newtheorem{theorem}{Theorem}[section]
\newtheorem{lemma}[theorem]{Lemma}

\newtheorem{proposition}[theorem]{Proposition}

\newtheorem{remark}[theorem]{Remark}
\newtheorem{definition}[theorem]{Definition}

\newtheorem{assumption}{Assumption}
\renewcommand{\theassumption}{H}
\numberwithin{equation}{section}
\crefname{lemma}{Lemma}{Lemmas}
\crefname{remark}{Remark}{Remarks}
\crefname{assumption}{Assumption}{Assumptions}
\crefname{proposition}{Proposition}{Propositions}
\crefname{equation}{}{}
\Crefname{equation}{Equation}{Equations}
\fi

\crefname{section}{Section}{Sections}
\crefname{subsection}{Subsection}{Subsections}

\usepackage{tikz}
\usepackage{tikz-cd}
\usepackage{pgfplotstable}
\pgfplotsset{compat=1.14}
\usetikzlibrary{patterns}
\usetikzlibrary{calc}
\usetikzlibrary{angles}
\usetikzlibrary{quotes}
\usetikzlibrary{external}

\usepackage[url=true,backend=biber,style=trad-abbrv,url=false,doi=false,isbn=false]{biblatex}
\DeclareFieldFormat{volume}{volume \textbf{#1}}
\DeclareFieldFormat[article]{volume}{\textbf{#1}}

\ifsiamart\else
\let\oldparagraph=\paragraph
\renewcommand\paragraph[1]{\oldparagraph{#1.}}
\fi

\newlist{assumpenum}{enumerate}{5}
\setlist[assumpenum]{font={\bfseries}}
\crefname{assumpenumi}{Assumption}{Assumptions}

\newlist{itemenum}{enumerate}{5}
\setlist[itemenum]{font={\bfseries}}
\crefname{itemenumi}{Assumption}{Assumptions}

\newcommand{\emp}{\pi^{\kalman,N}}
\newcommand{\empmfl}{\overline \pi^{\kalman,N}}
\newcommand{\mfl}{\overline \pi^{\kalman}}

\makeatletter
\def\th@plain{%
  \thm@notefont{}
  \itshape
}
\def\th@definition{%
  \thm@notefont{}
  \normalfont
}
\makeatother